\documentclass[11pt, oneside]{article} 
\usepackage{geometry} 
\usepackage{amsmath,amsthm,amssymb}
\usepackage{upgreek}
\usepackage{bbm}
\usepackage[mathscr]{euscript}
\usepackage{hyperref}
\hypersetup{
    colorlinks=true,
    linkcolor=black,
    urlcolor=blue,
    }
\usepackage{tikz-cd}
\usepackage{pgfplots}
\pgfplotsset{compat=1.12}
\usetikzlibrary{arrows.meta,angles,quotes}
\usepackage{longtable}

\newcommand{\norm}[1]{\left\lVert#1\right\rVert}
\newcommand{\innerpr}[1]{\langle\,#1\,\rangle}
\newcommand{\interior}{\text{int}\,}
\newcommand{\bbR}{\mathbb{R}}
\newcommand{\bbZ}{\mathbb{Z}}
\newcommand{\vb}{\mathcal{B}}
\newcommand{\A}{\mathcal{A}}
\newcommand{\D}{\mathcal{D}}
\newcommand{\E}{\mathcal{E}}
\newcommand{\F}{\mathcal{F}}
\newcommand{\M}{\mathscr{M}}
\newcommand{\T}{\mathcal{T}}
\newcommand{\pf}{\mathbbm{f}}
\newcommand{\pr}{\mathbbm{r}}
\newcommand{\ph}{\mathbbm{h}}
\newcommand{\dotin}{\ \dot{\in}\ }

\newtheorem{theorem}{Theorem}[section]
\newtheorem{lemma}[theorem]{Lemma}
\newtheorem{corollary}[theorem]{Corollary}
\newtheorem{proposition}[theorem]{Proposition}
\newtheorem{definition}[theorem]{Definition}
\newtheorem{remark}[theorem]{Remark}
\newtheorem*{maintheoremA}{Theorem A}
\newtheorem*{maintheoremB}{Theorem B}
\newtheorem*{goal*}{Goal}

\usepackage{chngcntr}
\numberwithin{figure}{section}
\numberwithin{table}{section}
\numberwithin{equation}{section}

\newcommand{\keywords}[1]
{\textbf{Keywords:} #1}

\title{Construction and upper bound on the minimum genus of an embedded surface with Anosov geodesic flow}

\author{
Victor Donnay\thanks{Department of Mathematics, Bryn Mawr College, Bryn Mawr, PA, USA,  vdonnay@brynmawr.edu}, 
Daniel Visscher\thanks{Department of Mathematics, Ithaca College, Ithaca, NY, USA, dvisscher@ithaca.edu.}
}

\begin{document}
\maketitle

\begin{abstract}
    We create examples of smooth, compact surfaces in $\bbR^3$ for which the geodesic flow is Anosov. We  determine their genus, thereby giving a (non-sharp) upper bound for the minimal genus of an embedded surface with Anosov geodesic flow. These   examples  are  explicit physically realizable Anosov systems. 
\end{abstract}

{\let\thefootnote\relax\footnotetext{\noindent {\bf 2020 Mathematics Subject Classification (MSC):} Primary 37D20,   
Secondary 37D40, 53A05, 53C22.  } }
{\let\thefootnote\relax\footnotetext{\noindent \keywords{geodesic flow, embedded surface, Anosov system} }}

\tableofcontents

\newpage
\section*{Notation}

\begin{longtable}{l p{5in}}

$(a,b) \in \Omega$                  & rectangular coordinates (Eqn~\ref{eqn:globalcoordinatesystem})\\

$C_i(s_0,s_1)$                      & $i=1,2,3$ geodesic control constants  (Eqn~\ref{eqn:Ci}) \\

$\D_0 \subset \bbR^2$               & hexagonal packing of unit disks (Sec~\ref{sec:circlepacking}) \\
$\D_\rho$                           & set of disks with same centers as $\D_0$, each scaled to radius $1-\rho \Delta r$ (Def~\ref{def:Drho})\\
$D_{\rho}\subset \D_{\rho} $ & a specific disk of radius $1-\rho \Delta r$ (Def~\ref{def:Drho})\\

$E = E_{(s_0,s_1)}$                 & first fundamental form component for metric $g_{(s_0,s_1)}$\\
$E_\M$                              & first fundamental form component for the embedded metric on the scaled model space $\M_{s_0}$ (Lemma~\ref{lemma:fundformdiffterms})\\
$\Delta E$                          & the difference $E_{(s_0,s_1)} - E_\M$ in the corresponding metric component due to the embedding map (Lemma~\ref{lemma:fundformdiffterms})\\

$\pf$                               & unsmoothed tube profile function (Eqn~\ref{eqn:unsmoothedtubefunction}) \\
$f$                                 & smoothed tube profile function (Prop~\ref{prop:smoothprofilefunction}) \\

$g_{Eucl}$ & the Euclidean metric on $\bbR^3$ (Eqn~\ref{eqn:modelspacemetric})\\
$g_0$                               & Euclidean metric on the plane (Sec~\ref{sec:circlepacking}) \\
$g_{\M_{0}}$                        & restriction of the Euclidean metric $g_{Eucl}$ to the model space $\M_{0}$ (Eqn~\ref{eqn:modelspacemetric})\\
$g_s = g_{(s_0,s_1)}$               & pullback metric of embedded surface: $g_s = X_{s_1}^* g_{Eucl} \big|_{X(\M_{s_0})}$ (Eqn~\ref{eqn:pullbackmetric}) \\
                                    & sometimes in coordinates: $g_s = (X_{s_1} Y_{s_0})^*g_{emb}$ (Sec~\ref{sec:coordinates})\\
$\Delta g$                          & the difference between the embedded metric $g_s$ and the model metric $g_0$ in coordinates (Sec~\ref{sec:coordinates}) \\
$\gamma_0$                          & a geodesic on the Euclidean plane (Sec~\ref{sec:contdep}) \\
$\gamma_s$                          & a geodesic in the metric $g_s$ (Sec~\ref{sec:contdep}) \\
                                    
$K_{pos}$                           & upper bound on the (positive) curvature of $g_s$ outside a given set (Thm~\ref{thm:Anosov}) \\
$K_{neg}$                           & upper bound on the (negative) curvature of $g_s$ inside a given set (Thm~\ref{thm:Anosov}) \\
$K_{\rho}$                          & upper bound on the (negative) curvature of $g_s$ inside $\T_{\rho}$ ($\rho=1/4,\, 1/2$; Eqn~\ref{eqn:Krhodef})\\

$\lambda_s(v)$                      & ratio of the $g_s$-length of $v$ to its $g_0$-length (Eqn~\ref{eqn:lambdadef})\\
$\lambda_s^{lb}$ (resp., $\lambda_s^{ub}$) & lower bound (resp., upper bound) for $\lambda_s(v)$ for any $v \in \Omega$ (Lemma~\ref{lemma:lengthratio2})\\

$\M_{0}$                            & model space (Sec~\ref{sec:tubes})  \\
$\M_{s_0}$                          & model space scaled to height $s_0$ (Sec~\ref{sec:tubescaling})\\

$\Omega = \bbR^2 \setminus \D_1$    & rectangular coordinate space (Sec~\ref{sec:rectangularcoordinates})\\

$\Delta r$                          & $= 1- \cos (\pi/6)$, the minimum radial distance a straight line in $\bbR^2$ is guaranteed to penetrate a disk in $\D_0$ (Eqn~\ref{eqn:DeltarDefn}) \\

$s_0$                               & scaling parameter for the model space (Sec~\ref{sec:tubescaling}) \\
$s_1$                               & embedding parameter (Sec~\ref{sec:embeddingmap}) \\

$\T = \T_{s_0} \subset \M_{s_0}$    & tubes in the model space (Sec~\ref{sec:tubescaling}) \\
$\T_\rho \subset \M_{s_0}$          & tubes sitting above $\D_\rho$ (Def~\ref{def:Tubesrho})\\
$T_{\rho}\subset \T_{\rho} $ & a specific tube  in the collection $\T_{\rho}$ (Def~\ref{def:Tubesrho})\\
$T_{ret}$                           & bound on the time it takes a geodesic to return to a given set in a good way (Def~\ref{def:strongfhp}) \\
$\Delta t$                          & lower bound on amount of time spent inside of a given set (Def~\ref{def:strongfhp})\\
$T_{ret}(\rho)$                     & optimal time for which $\D_1 \subset \D_{\rho}$ has the  $T_{ret}=T_{ret}(\rho)$ strong finite horizon property in the $g_0$ metric (Prop~\ref{prop:DrhoStrongFiniteHorizon})\\
$\Delta t(\rho)$                    & corresponding $\Delta t$ value for the sets $\D_1 \subset \D_{\rho}$ (Prop~\ref{prop:DrhoStrongFiniteHorizon})\\
$T_{gc}$                            & bound on the time for which the distance between partner geodesics is bounded (often by $\Delta r/4$) (Sec \ref{sec:contdep})\\
$T_{tube}(p,v)$                     & first time either $\gamma_{0}(p, v, t)$ or $\gamma_{s}(p, v_s, t)$ leaves $\Omega$ (Eqn~\ref{eqn:Ttube}) \\

$u^*$                               & ``synthetic'' Riccati solution (Def~\ref{def:uRiccatisoln}) \\
$u_\gamma$                          & Riccati solution along geodesic (Pf of Thm~\ref{thm:Anosov}) \\

$X^{sc}_{s_0}$                      & map $\bbR^3 \to \bbR^3$ that vertically scales the tubes of the model space (Sec~\ref{sec:tubescaling}) \\
$X^{emb}_{s_1}$                           & embedding map $\bbR^3 \to \bbR^3$ (Sec~\ref{sec:embeddingmap}) \\
$X_s =X^{emb}_{s_1}\circ X^{sc}_{s_0}$ & map $\bbR^3 \to \bbR^3$ that composes the scaling and embedding maps (Sec~\ref{sec:embeddingmap}) \\

$Y_0$                               & coordinate map on the model space $\M_0$ (Sec~\ref{sec:coordinates}) \\
$Y_{s_0}$                           & coordinate map on the scaled model space $\M_{s_0}$ (Sec~\ref{sec:coordinates}) \\
$Y_{(s_0,s_1)}$                     & coordinate map on the embedded model space $X_{(s_0,s_1)}(\M_0)$ (Sec~\ref{sec:coordinates}) \\

\end{longtable}

\begin{center}
\begin{tikzcd}[row sep=large, column sep=large]
\text{model space} & \M_0 \subset \bbR^3 \ar[r,"X^{sc}_{s_0}"] \ar[rr,bend left=40,"X_{(s_0,s_1)}"] & \M_{s_0} \subset \bbR^3 \ar[r,"X^{emb}_{s_1}"] & \bbR^3 \\
\text{coordinates} & \mathcal{U} \subset \bbR^2 \ar[u,"Y_0"'] \ar[ur,bend right=5,"Y_{s_0}"] \ar[urr,bend right=10,"Y_{(s_0,s_1)}"']
\end{tikzcd}
\end{center}

\newpage
\section{Introduction}

Geodesic flows on smooth, compact surfaces of negative curvature have long served as prototypical examples of chaotic dynamical systems. Starting with Hadamard in 1898~\cite{Hadamard-98} and continuing with work in the 1930's and 40's by Hopf~\cite{Hopf-39, Hopf-48}, Hedlund~\cite{Hedlund-39}, and others, properties such as sensitive dependence on initial conditions, ergodicity, and mixing were established for such geodesic flows, first in constant and then in variable negative curvaturve. In 1967, Anosov~\cite{Anosov-67} recognized that the underlying phenomena responsible for the chaotic dynamics was \emph{uniform hyperbolicity}, and this fundamental property of a dynamical system was subsequently named after him. 

Since compact surfaces of negative curvature cannot be embedded in $\bbR^3$, a natural question about Anosov systems is whether there are any that are ``physically realizable''. In 2003, Hunt and MacKay \cite{Hunt-MacKay-03} established the existence of a physically realizable smooth Anosov system via linkages, which model the movement of point masses connected by rigid rods.  At about the same time, in answer to a question of M. Hermann, Donnay and Pugh~\cite{Donnay-Pugh-04} proved the existence of physically realizable surfaces (i.e., compact surfaces isometrically embedded in $\bbR^3$) with Anosov geodesic flow.   

Both of these results use limiting arguments. Hunt and Mackay produced an Anosov system that is a limit of linkage systems (but not one itself), so that, using the openness of Anosov systems, there are sufficiently small parameter values that yield Anosov linkages. Donnay and Pugh showed that for a sufficiently large pair of concentric spheres and  a sufficiently large number of negatively curved tubes placed between them, with the tubes placed  in such a way that the effects of the negative curvature outweigh the effects of the positive curvature, the geodesic flow would be Anosov. In neither case is the value of ``sufficiently'' determined.\footnote{Hunt and MacKay do give explicit parameter values that are computationally promising, but note that ``to make a proof, however, would probably require computer-assisted estimates." Indeed, the current paper uses computer assistance for computing bounds.}

In 2016, Kourganoff~\cite{Kourganoff-16-linkage} made further progress by constructing an Anosov linkage system in a more explicit fashion. This linkage system, different from the Hunt-MacKay one, comes with explicit values for the lengths of the links and most of the masses. However, the result still requires one of the masses to be ``sufficiently small''. 

In this paper, we return to the category of geodesic flows on embedded surfaces and give the first completely explicit, constructed example of a physically realizable smooth system that is Anosov.\footnote{We note that Sinai billiards is a natural candidate for a physical Anosov system, but the billiard flow is not smooth.}  Once a geodesic flow on a surface has been shown to be Anosov then it is known to be ergodic and mixing \cite{Anosov-67}, Bernoulli \cite{Ratner-74} and have exponential decay of correlations \cite{Dolgopyat-98}, \cite{Liverani-04}. 

Our approach builds on the alternative proof of the Donnay-Pugh result given in \cite{Donnay-Visscher-18}. There we 
created a non-compact model space consisting of two copies of the flat plane joined by tubes of negative curvature which are attached to a finite horizon pattern of disks (see Figure \ref{fig:modelspace}). That model system is Anosov and hence so are all nearby systems, including ones coming from compact embedded surfaces. In this paper we essentially quantify the size of the Anosov neighborhood of the model space\footnote{The model space in this paper is slightly different than the one in \cite{Donnay-Visscher-18}, due to a different finite horizon configuration.} which leads to the proof of Theorem A.  We use a mix of analytic and computational methods 
\cite{Donnay-Visscher-mathematica}, with steps to ensure that computational methods result in true bounds. 

\begin{maintheoremA}
    There exists a compact surface embedded in $\bbR^3$ of genus 17,288,843,803 with Anosov geodesic flow. 
\end{maintheoremA}

This result can also be viewed in the tradition of geodesic flow results that consider how far away from surfaces of negative curvature one can go and still retain the Anosov property. Eberlein \cite{Eberlein-73-I,Eberlein-73-II} found conditions under which surfaces of non-positive curvature would be Anosov: every geodesic needs to go through a point of negative curvature.  His results also extend to surfaces of primarily non-positive  curvature but with small amounts of positive curvature.

Too much positive curvature, however, is an obstruction to being Anosov. Klingberg \cite{Klingenberg-74} showed that if a surface had conjugate points, then its geodesic flow could not be Anosov. Work of E. Hopf \cite{Hopf-48} then implies that no Riemannian metric on two-sphere or two-torus can be Anosov, since all metrics on $S^2$ and all non-flat metrics on $T^2$ have conjugate points and those surfaces---whether embedded or not---cannot support Anosov geodesic flow. While all surfaces of genus $g\geq 2$ have metrics of strictly negative curvature and therefore Anosov geodesic flow, these metrics are not realizable as surfaces isometrically embedded in $\bbR^3$.  

From this point of view, Donnay and Pugh's result shows that the minimum genus of a compact embedded surface with an Anosov geodesic flow is well-defined, and the above discussion shows that a lower bound for this number is 2. Our result then gives a first upper bound:

\begin{maintheoremB}
    The minimum genus of a compact embedded surface with Anosov geodesic flow is bounded above by $1.73\times 10^{10}$.
\end{maintheoremB}

While the genus of the constructed example (given in Theorem A) is an exact number, it is certainly not a sharp upper bound for the minimum genus of an embedded surface with Anosov geodesic flow: readily thought-of imporovements will marginally improve this result, though also complicate the analysis. The current paper carries through this idea: a simple version of our methods yields a surface of genus 233,129,289,619; bounding improvements pursued in Section 8 yield the surface of genus 17,288,843,803.

The complexity of this example, reflected in the relatively large genus, stems from the difficutly of ensuring hyperbolic dynamics everywhere while in $\bbR^3$. Changing the ambient space in which the embedding occurs alters the problem significantly: for example, Kourganoff \cite{Kourganoff-16-embed} has given examples of Anosov surfaces of genus $\geq 11$ that embed into $S^3$ (but not $\bbR^3$). 

For surfaces embedded in $\bbR^3$, successful methods have interspersed regions of positive and negative curvature in such a way that the small amounts of positive curvature a geodesic encounters in between visits to negative curvatures does not undo the exponential divergence caused by the negative curvature (\cite{Donnay-Pugh-03}, \cite{Donnay-Visscher-18}). The Gauss-Bonnet Theorem links the total amount of curvature on a surface  $M$  with the genus via the formula
    $$\int_M \, K(p) \, dA = 2 \pi (2 - 2g).$$
The larger the genus, the more negative this total curvature becomes, and the easier it is to construct embedded metrics with Anosov geodesic flow.

\subsection*{Overview of the construction}

To construct the embedded surface,  start with a finite horizon pattern of disks on the flat torus. Make a second copy of this flat torus with disks, place it above the first and connect the two by tubes of negative curvature that sit above the disks (Figure~\ref{fig:modelspace}). Take the metric induced by the Euclidean one on $\bbR^3$.  

\begin{figure}[h!]
    \centering
    \includegraphics[width=0.7\linewidth]{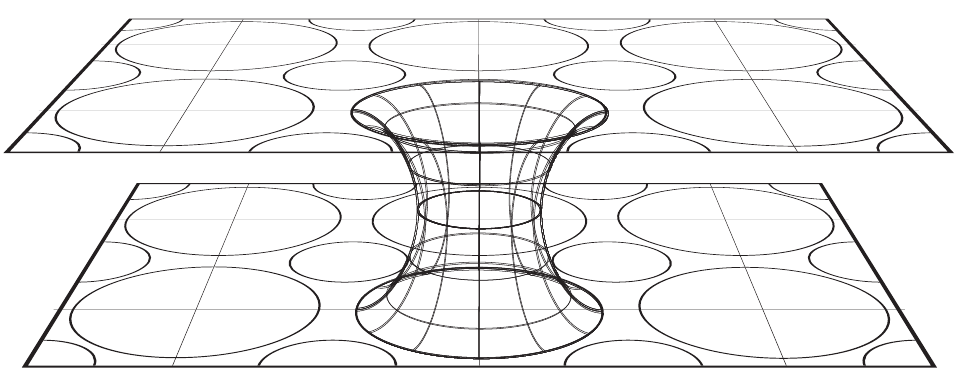}
    \caption{Two flat tori with finite horizon disk patterns and one of the joining tubes of negative curvature drawn in. This finite horizon pattern is from~\cite{Donnay-Visscher-18}; the finite horizon pattern used in this paper is shown in Figure~\ref{fig:circlepacking}.}
    \label{fig:modelspace}
\end{figure}

By the finite horizon condition of the disk configuration, all geodesics will enter a tube sitting above the disk in some finite time and with some positive angle. Geodesics then spend some positive time inside a part of the tube where the curvature is strictly negative. By compactness, we can make all these quantities uniform: there exists a region $S$, a subset of the tubes, in which the curvature is bounded above by $K_{neg} < 0$, so that all geodesics will enter $S$ in time at most $T_{ret} > 0$ and will then  spend at least time $\Delta t > 0$ there. We formulalize these proporties through the definition of a strong finite horizon property (see Definitions \ref{def:strongfhp} and \ref{def:strongfhpgeom}). Outside the disks the surface is flat so the curvature is zero curvature. These properties are enough to prove the Anosov property (see Theorem \ref{thm:Anosov}).

Now periodically extend the disk pattern on the flat torus to cover all of $\bbR^2$. Taking two copies of $\bbR^2$ and connecting them by tubes sitting over the disks gives a non-compact surface which we call the model space  $\M_0$. Denote by $g_{\M_0}$ the metric it inherits from the Euclidean metric on $\bbR^3$. 
Since this surface is just a cover of the initial compact surface, its geodesic flow is also Anosov.   All sufficiently small perturbations of its metric will retain the Anosov property.  

Define an ``embedding map'' $X:\bbR^3 \to \bbR^3$ which,  when applied to the model space $\M_0 \subset \bbR^3$,  will yield an immersed---and for some parameters embedded---surface. Let $X: \bbR^3 \to \bbR^3$ be given by
    $$
    X(u, v, w) = (X_1(u, v, w), X_2(u, v, w), X_3(u, v, w))
    $$
with 
\begin{align}
    X_1(u, v, w)
        & = \left( R_1 +( R_2+w) \cos\left( \frac{v}{R_2}\right) \right) \cos\left( \frac{  u }{R_1}\right) \notag \\
    X_2(u, v, w) 
        &= \left(R_1 +( R_2+w) \cos\left( \frac{  v}{R_2}\right) \right)\sin\left( \frac{  u }{R_1}\right)\\
    X_3(u, v, w) 
        &=  ( R_2+w) \sin\left( \frac{  v}{R_2}\right) \notag
\end{align}
and parameters $R_1$ and $R_2$. This map takes a plane $w=\text{constant}$ and sends it to an immersed  torus, with major radius $R_1$ and minor radius $R_2+w$. As long as $R_2+w < R_1$, the image of this plane is an embedded torus.

The pullback metric by this embedding map gives a new metric $g = DX^*(g_{Eucl})$ on $\bbR^3$. This pullback metric satisfies
\begin{align}\label{formulaforpullbackmetrics1}
    \innerpr{\zeta, \eta}_{g} = \innerpr{DX \zeta, DX \eta}_{g_{Eucl}} = 
    \left( 1+ \frac{(R_2+w) \cos \left(\frac{v}{R_2}\right)}{R_1} \right)^2 \zeta_1 \eta_1
    + \left( 1 +\frac{ w}{R_2} \right)^2 \zeta_2 \eta_2
    + \zeta_3 \eta_3
\end{align}
for vectors $\zeta=(\zeta_1, \zeta_2, \zeta_3)$ and $\eta=(\eta_1, \eta_2, \eta_3)$ in $T_p(\bbR^3)$  expressed in Euclidean coordinates. 

Create a one-parameter family of mappings $X_{s_1}$ by defining $R_1({s_1})$ and $ R_2({s_1})$ with ${s_1}\in(0, 1]$ and require that as ${s_1}\to 0, R_1({s_1}), R_2({s_1})\to \infty$ and $ R_2({s_1})/R_1({s_1}) \to 0$. The pull-back metric   then converges to the Euclidean metric on $\bbR^3$ as ${s_1}\to 0$. Let $g_{s_1}$ be the restriction of the pull-back metric to the model space: $g_{s_1} = DX_{s_1}^*(g_{Eucl})|_{\M_0}$.   For all ${s_1}$ sufficiently small, $g_{s_1}$      will be a small perturbation of the initial model space metric $g_{\M_0}$. Hence the geodesic flow on $(\M_0, g_{s_1})$ will be Anosov. 

Among these sufficiently small ${s_1}$ values is a discrete set for which $X_{s_1}(\M_0)$ is an embedded surface. For the image to be embedded, and not just immersed, one needs the periodicity of the fundamental region of the finite horizon configuration and the periodicity of the embedding map, determined by the radii $R_1({s_1})$ and $R_2({s_1})$, to align (see Section~\ref{sec:genuscalculation}). 

So far this simply gives another proof of the Donnay-Pugh Theorem. To prove Theorem A, we need to quantify all aspects of the construction by determining the set $S$ and the values of $K_{neg}$, $T_{ret} > 0$ and $\Delta t > 0$. Recall the definitions of these terms: $S$ is a a subset of the tubes in which the curvature is bounded above by $K_{neg} < 0$,  all geodesics will enter $S$ in time at most $T_{ret} > 0$ and will then  spend at least time $\Delta t > 0$ there.

If the positive curvature outside the tubes, bounded above by some $K_{pos}> 0$, is not too large, then the uniform time spent in strictly negative curvature, combined with the bounded time $T_{ret}$ spent in small positive curvature, allows us to prove a strictly invariant cone condition (Definition \ref{def:uRiccatisoln}). This condition, expressed in terms of solutions of the Ricatti equation, in turn implies  the Anosov property (Theorem \ref{thm:Anosov}).   

How can one determine numerical values for these quantities?  We start by examining the flat plane with a finite horizon collection of disks and determining bounds for the behavior of the geodesics (straight lines) in the flat metric $g_0$. Create a finite horizon disk configuration using a hexagonal packing of unit disks (Section \ref{sec:circlepacking}). The simplicity of this particular configuration leads to explicit values for the  strong finite horizon property for the straight line geodesics $\gamma_0$. All geodesics that start outside a disk will enter a disk in time at most 2 with angle at least $\pi/6$ (Theorem \ref{thm:diskarrangementangletimebounds}), which then implies other versions of the strong finite horizon property (Corollary \ref{thm:D1StrongFiniteHorizon} and Proposition \ref{prop:DrhoStrongFiniteHorizon}).

We next define a two-parameter family of pullback metrics $g_s = g_{(s_0,s_1)}$ on the model space $\M_0$. The parameter $s_1$ is the embedding parameter described above and $s_0$ is a scaling parameter described below. Choose the following explicit values for the functions $R_i(s_1)$    that  determine the radii of the immersed tori (Section \ref{sec:embeddingmap}): 
    $$R_1(s_1) = \frac 1{\pi s_1^2}, \quad R_2(s_1)=\frac{1}{\pi s_1}.$$

Considering geodesics $\gamma_s$ on $(\M_0, g_s)$ as perturbations of $g_0$ straight line geodesics (until $\gamma_s$ enters sufficiently far into the tubes) leads to a proof that the strongly finite horizon property of the flat metric carries over to a (slightly weaker) strongly finite horizon property on $(\M_0, g_s)$ (Theorem~\ref{thm:Cibounds->geodcontrol->finitehorizon}). The proof  uses ``geodesic control" techniques (Section \ref{sec:contdep}) to determine how far a $\gamma_s$ geodesic can move away from the corresponding $\gamma_0$ straight line  geodesic before they together enter into a tube/disk and is carried out using  rectangular coordinates (Section \ref{sec:rectangularcoordinates}). 

The negatively curved tubes are surfaces of revolution constructed out of arcs of circles (Section \ref{sec:tubes}).  Their ends are modified to attach smoothly to the flat planes on the boundary of the disks (Section~\ref{sec:smoothing}). To ensure that the the pull-back metric $g_s$ on $\M_0$ (Section \ref{sec:embeddingmap}) is a  small perturbation of the flat metric $g_0$ on the outer part of the tubes, we use a scaling parameter $s_0$. The scaling parameter  causes  the two planes in the model space to come  closer together\footnote{This construction of squeezing the two planes together is  reminiscent of an argument by Arnol'd (\cite{Arnold-63}, p.~184). In explaining why the dynamics of billiard motion with a convex scatter should be chaotic, he considered the billiard motion to be the limit of  geodesic motion. Take two flat tori that are parallel, connect them by a tube of negative curvature and then have the distance between the flat tori go to zero ($s_0\to 0$ in our notation). The billiard motion can be thought of (intuitively) as the limit of the chaotic geodesic flow motion. The boundary of the tube corresponds to the  convex scattering obstacle.} and flattens the outer part of the tubes  (Section \ref{sec:tubescaling}). 

Although it is straightforward to calculate the negative curvature of rotationally symmetric tubes in the Euclidean metric, it is more challenging to bound their curvature under the pullback metric $g_s$. Introducing a type of  angular coordinates (Section \ref{sec:angularcoordinates}) allows us to  determine an upper bound $K_{neg}<0 $ for the curvature in $S$ (Section \ref{sec:curvboundsinA14}). The use of  polar coordinates (Section \ref{sec:metricinpolarcoordinates}) leads to  an  upper bound $K_{pos}>0$ for the curvature outside of $S$ (Section \ref{sec:curvboundsoutsideT14}).  

We put all these pieces together (Section \ref{sec:embeddedsurfacesAnosovpart1}) and determine $s=(s_0, s_1)$ values for which the geodesic flow on ($\M_0, g_s)$ satisfies both the strongly finite horizon condition and strictly invariant cone condition and hence is Anosov (Theorem~\ref{thm:AnosovCondition}). From among these values, we determine the discrete set of $s_1$ parameter values for which the image surface $X_{s} (\M_0)$ is  embedded and give a formula  for its genus (Lemma~\ref{lemma:genusformula}). 

We use Mathematica to carry out various calculations that arise in solving our analytic equations (Section \ref{sec:numericallycomputegenus}) and discuss the steps taken to ensure that these calculations produce true bounds (Section \ref{sec:validatingMathematica}). The result is a proof of the Theorem A with genus 233,129,289,619 (Theorem~\ref{thm:genusbound}). Improving the  estimates for geodesic control, negative curvature in the tubes, time spent in negative curvature and solutions of the  Ricatti equation (Section \ref{sec:SophisticatedEstimates}) leads to the example of genus 17,288,843,803. A Mathematica notebook containing all the calculations can be found at \cite{Donnay-Visscher-mathematica}.

\subsection*{Connections to other results}

Could our bounds on the genus of Anosov metrics be improved if we replaced our cone-field condition by other conditions for Anosov systems?  Eberlein \cite{Eberlein-73-I}, \cite{Eberlein-73-II} showed that a geodesic flow on a surface with no focal points and where every trajectory goes through a point of negative curvature will be Anosov. Our method of proof (Lemma \ref{lemma:uknegukpos}) of the invariant cone condition (Definition \ref{def:uRiccatisoln}) implies the no focal points condition. Could the cone condition be relaxed in such a way that the no focal points condition still holds?  We note that Gulliver \cite{Gulliver-75}  created examples of metrics on surfaces which do have focal points but are still Anosov (and hence have no conjugate points). It would  be interesting to see whether our construction could be modified to create Anosov examples with  focal points. 

Guglielmo and Ruggiero \cite{Guglielmo-Ruggiero-26} have explored the structure of the set of metrics  with Anosov geodesic flow. They look at Anosov surfaces with no focal points and with   regions of positive curvature.  They show for such surfaces that there exists a smooth curve of conformal deformations of the metric that preserves the Anosov property and connects the initial metric  to a Riemannian metric of negative curvature. Their result assumes that the positive curvature is contained in disjoint disks (generalized bubbles). In our Anosov examples, the positive curvature is not contained in such bubbles. Can one prove their result for our systems?   

Our surfaces can be used to give examples of  families  of metrics that are Anosov but for which the limit metric is not Anosov but still has no conjugate points. These examples illustrate a conjecture by  Jane and Ruggiero \cite{Jane-Ruggiero-14}: the closure of the set of compact Riemannian surfaces with Anosov geodesic flows in the $C^2$ topology is the set of Riemannian metrics without conjugate points. One such family occurs from shrinking the size of the disks so that in the limit the finite horizon fails and one has a periodic orbit lying completely in zero curvaure. Or, perturb the symmetric tube of negative curvature so that in the limit the neck of the tube has zero curvature. For the periodic orbit around the neck, the Anosov condition fails.

\section{Theoretical framework}

\subsection{A quantified strong finite horizon property}~\label{subsection: FiniteHorizonDisks}
The following defines a quantified version of a finite horizon property that tracks how long it takes for a geodesic to enter a target set $S$ {\it in a substantial way}. We give two ways of quantifying ``substantial'': first giving a lower bound for how long a geodesic stays in that set, and second specifying an interior set that the geodesic must reach.

\begin{definition}[Time version]\label{def:strongfhp}
    Let $M$ be a surface with Riemannian metric $g$, and let $S$ be a closed set on $M$. Then we say the set $S$ has the \emph{$(T_{ret}, \Delta t)$ strong finite horizon property} if for any $g$-geodesic $\gamma$ there exists a sequence of times $t_i^\pm$, $i \in \bbZ$ with $\lim_{i \to \pm \infty} t_i^+ = \pm \infty$, such that for all $i$,
    \begin{enumerate}
        \item $t_i^- \leq t_i^+ < t_{i+1}^- \leq t_{i+1}^+$ 
        \item for all $t \in [t_i^-, t_i^+]$, $\gamma(t) \in S$
        \item $t_i^+-t_i^- \geq \Delta t$
        \item $t_{i+1}^- -t_i^+ \leq T_{ret}$
    \end{enumerate}
    We say such a geodesic enters $S$ in a $\Delta t$ good way at time $t_i^-$.
\end{definition}

The return time $T_{ret}$ is a bound on the time it takes a geodesic that is outside of $S$ to next return to $S$ in a $\Delta t$ good way. 

Property 1 implies that $T_{ret} >0$ and $\Delta t \geq 0$. Note that $T_{ret}$ is an upper bound, and so a smaller value of $T_{ret}$ gives a stronger condition. Similarly, since $\Delta t$ is a lower bound, a larger value of $\Delta t$ gives a stronger condition. Since $\Delta t>0$ places an additional constraint on when to count a geodesic as having reached $S$, larger $\Delta t$ generally means the optimal $T_{ret}$ must also be larger. 

\begin{definition}[Set version]\label{def:strongfhpgeom}
    Let $M$ be a surface with Riemannian metric $g$, and let $S' \subseteq S$ be closed sets on $M$. Then we say the pair of sets $S' \subseteq S$ has the \emph{$T_{ret}$ strong finite horizon property} if for any $g$-geodesic $\gamma$ there exists a sequence of times $t_i^\pm$, $i \in \bbZ$ with $\lim_{i \to \pm \infty} t_i^+ = \pm \infty$, such that for all $i$,
    \begin{enumerate}
        \item $t_i^- \leq t_i^+ < t_{i+1}^- \leq t_{i+1}^+$ 
        \item for all $t \in [t_i^-, t_i^+]$, $\gamma(t) \in S$
        \item there exists $t_i^*\in [t_i^-, t_i^+]$ such that $\gamma(t_i^*) \in S'$
        \item $t_{i+1}^- -t_i^+ \leq T_{ret}$
    \end{enumerate}
    We say such a geodesic enters $S$ in an $S'$ good way at time $t_i^-$.
\end{definition}

Given sets $S' \subseteq S$ and a metric $g$ that have the strong finite horizon property for some $T_{ret}$, there is a smallest such $T_{ret}$ (we call this the optimal $T_{ret}$). If the pair of sets $S' \subseteq S$ has the $T_{ret}$ strong finite horizon property then there exist values of $\Delta t \geq 0$ for which the set  $S$ has a $(T_{ret},\Delta t)$ strong finite horizon property. For such a system, there is an optimal (i.e., largest) $\Delta t$. Given the pair of sets $S' \subseteq S$, we say that the pair of times $(T_{ret}, \Delta t)$ is optimal if $T_{ret}$ is the optimal strong finite-horizon time for $S' \subseteq S$, and if $\Delta t$ is the optimal time for this value of $T_{ret}$.

We note that in the Euclidean case, the correspondence between  $S'$ and the optimal $\Delta t$ (largest possible) can be given explicitly (see Proposition~\ref{prop:DrhoStrongFiniteHorizon}); in the general Riemannian case, we are only able to give a lower bound on the optimal $\Delta t$ for a given $S'$.

The cases $\Delta t=0$ or $S'=S$ reduce to a quantified standard finite horizon property:

\begin{definition}
    $\S \subset M$ has the $T_{ret}$-finite horizon property if for all $(p,v) \in SM$, there exists a $0 \leq t \leq T_{ret}$ such that $\gamma_{(p,v)}(t) \in \S$. 
\end{definition}

Yet another way to quantify how substantially a geodesic enters a set $S$ is by specifying a minimum angle at which the geodesic enters the set. The amount of time $\Delta t$ spent in $S$ is the most relevant for proving dynamical properties, but often other geometric indicators are easier to obtain.

For a set $S$ with the $(T_{ret},\Delta t)$-strong finite horizon property, we make the following canonical choice of the $t_i^\pm$: let $t_i^-$ be the times for which $\gamma(t_i^-) \in \partial S$ (entering) and $\gamma(t_i^+)$ the next time $\gamma(t_i^+) \in \partial S$ (leaving) provided $t_i^+-t_i^- \geq \Delta t$ (Property 3). If there is no such $t_i^+$ (resp. $t_i^-$)---i.e., the trajectory remains in $S$ for all future (resp. past) times---then the remaining values of $t_i^\pm$ can be chosen arbitrarily providing they fit the definition. With this canonical set of times, we refer to the intervals $[t_i^-, t_i^+]$ as ``good times'' and $\gamma(t_i^-)$ as ``good returns to $S$''.

\subsection{Anosov Argument} \label{sec:Anosovargument}
 
\begin{definition}\label{def:uRiccatisoln}
    The numbers $T_{ret}, \Delta t >0$ and $K_{neg}<0\leq K_{pos}$ satisfy the \emph{strictly invariant cone condition} if the solution $u^*(t)$ of the Riccati equation $u'(t) = - K^*(t) - u^2(t)$ with
    \[
    K^*(t) = \begin{cases} 
        K_{neg},  & t\in [0, \Delta t),\\
        K_{pos},& t\in [\Delta t, \Delta t + T_{ret}]
   \end{cases}
    \]
    and initial condition $u^*(0)=0$ is defined for all $t \in [0,\Delta t+T_{ret}]$ and satisfies 
        \begin{equation}\label{eqn:strictlyinvconecondition}
        u^*(\Delta t+ T_{ret}) > 0.
        \end{equation}
\end{definition}

In terms of Jacobi coordinates on the tangent space $T(SM)$, the above conditions implies that the first (and 3rd) quadrants get mapped strictly inside themselves. 
 
\begin{theorem}\label{thm:Anosov}
    Let $M$ be a closed surface with a smooth Riemannian metric $g$, and let $S \subset M$ be closed. Suppose there are numbers $T_{ret},\Delta t>0$ and $K_{neg}<0\leq K_{pos}$ satisfying the strictly invariant cone condition such that
    \begin{enumerate}
        \item $S$ has the $(T_{ret},\Delta t)$ strong finite horizon property,
        \item the Gaussian curvature inside $S$ is bounded above by $K_{neg}$, and
        \item the Gaussian curvature on $M$ is bounded above by $K_{pos}$.
    \end{enumerate}
    Then the geodesic flow is Anosov.
\end{theorem}

Note that the synthetic Ricatti solution given by Definition \ref{def:uRiccatisoln} gives a lower bound for Ricatti solutions along geodesics. More precisely, let $\gamma$ be a geodesic that enters $S$ at time $t_i^-$ and then spends at least time $\Delta t$ in $S$ before exiting at time $t_i^+$. Let $u_{\gamma}$ be the Ricatti solution with initial condition $u_{\gamma}(t_i^-) =0$ and with $K(t) = K(\gamma(t))$. Then the Comparision Theorem for Ricatti solutions shows that $u_{\gamma}(t_i^+) \geq u^*(\Delta t)$ and when the geodesic next returns to $S$ in a $\Delta t$ good way at time $t_{i+1}^-$, then $u_{\gamma}(t_{i+1}^-) \geq u^*(\Delta t+T_{ret})$.

The above result would be an immediate consequence of the following formulation of a finite-time uniformly strictly invariant cone fields (\cite{Donnay-Visscher-18} due to Kourganoff \cite{Kourganoff-18}) if the times that geodesics spends in $S$ were bounded. In our situation these times can be unbounded and indeed infinite. However,  since the curvature in $S$ is strictly negative, spending time there helps with generating hyperbolicity. Thus, while the unbounded time adds a technical complication to the proof, it is not a serious problem. 

\begin{theorem}[Kourganoff, \cite{Kourganoff-18}] \label{thm:Kournagoff}
    Let M be a closed surface. Assume that there exist $m > 0$ and $ C > c > 0$   such that for any geodesic $\gamma : \bbR \to M$, there exists an increasing sequence of times $(t_k)_{k\in \bbZ} \in \bbR^\bbZ$ with 
        \begin{equation}\label{Kourganofftimesproperty}
        c \leq t_{k+1} - t_k \leq C,
        \end{equation}
    such that the solution u of the Riccati equation with initial condition $u(t_k) = 0$ is defined on the interval $[t_k, t_{k+1}]$, and $u(t_{k+1}) > m$. Then the geodesic flow $\varphi^t : SM \to  SM$ is Anosov.
\end{theorem}

\bigskip

\noindent {\bf Extension of \ref{thm:Anosov}, \ref{thm:Kournagoff}.} 
    We note that Theorems~\ref{thm:Anosov} and \ref{thm:Kournagoff} will still hold even in cases when $M$ is not closed if Lemma 5.8 of Kourganoff still holds:
        \begin{equation}\label{eqn:suponDphi}
        \sup ||D\varphi^t (p,v)|| < +\infty
        \end{equation}
    for set of all $(t,(p,v)) \in [0,2C] \times SM$. In our case, $C = T_{ret} + 2\Delta t$ and equation (\ref{eqn:suponDphi})  will be satisfied due to the periodic (or close to periodic) nature of our model space.

\bigskip

In the definition of the $(T_{ret},\Delta t)$ strong finite horizon property (\ref{def:strongfhp}), we typically think of $t_i^-$ as a time that $\gamma$ enters $S$ and $t_i^+$ as the subsequent time that $\gamma$ leaves $S$. However, this definition also allows for the case that a geodesic remains in $S$ as $t \to \pm \infty$.

\begin{proof}[Poof of Theorem~\ref{thm:Anosov}]
Consider a geodesic $\gamma$. By above, we can pick the times $t_i^\pm$ with the property $t_i^+ - t_i^- \geq \Delta t$.

If $t_i^+ - t_i^- \geq 2 \Delta t$, then we divide the interval $[t_i^-,t_i^+]$ into $N$ pieces of length $\Delta t$ and a final piece of length between $\Delta t$ and $2\Delta t$. Set $t_i^0 = t_i^-$, $t_i^k = t_i^0 + k \Delta t$ for $1 \leq k \leq N$, and $t_i^{N+1} = t_i^+$. Then we have
\begin{enumerate}
    \item $t_i^{k}-t_i^{k-1} = \Delta t$ for $1 \leq k \leq N$,
    \item $\Delta t \leq t_i^{N+1}-t_i^{N} < 2 \Delta t$, and
    \item $t_{i+1}^0 - t_i^{N+1} \leq T_{ret}$;   so that
    \item $\Delta t \leq t_{i+1}^0-t_i^N < T_{ret}+2\Delta t$.
\end{enumerate}

If $t_i^+ - t_i^- < 2 \Delta t$, then $N=0$ so that $t_i^0 = t_i^N = t_i^-$ and $t_i^{N+1} = t_i^1 = t_i^+$. We define a sequence of times 
    \[
    ...<t_i^0 \underbrace{<t_1^1<...<t_i^N}_{\text{if }N>0} < t_{i+1}^0 < \underbrace{t_{i+1}^1<...<t_{i+1}^N}_{\text{if }N>0} < t_{i+2}^0 ...,
    \]
which has the Kourganoff property~(\ref{Kourganofftimesproperty}) with $c = \Delta t$ and $C = T_{ret}+2\Delta t$.

Let $u^*$ be the Riccati solution as given in Definition~\ref{def:uRiccatisoln}. We will show below that one can choose $m = u^*(\Delta t + T_{ret}) > 0$.  

For $t \in [t_i^k, t_i^{k+1}]$ for $0 \leq k \leq N-1$, $\gamma(t) \in S$. (In the case $N=0$, this case is empty.) Let $u_\gamma(t)$ be the solution of the Riccati equation over this time interval with initial condition $u_\gamma(t_i^k)=0$ and curvature $K(t) =K(\gamma(t))$. Since $K(\gamma(t)) <  K_{neg}$, a standard Comparision Theorem gives that
    $$
    u_\gamma(t_i^{k+1})= u_\gamma(t_i^{k}+ \Delta t) \geq u^*(\Delta t) > 0,
    $$ 

For the interval $[t_i^{N},t_{i+1}^0]$, let $\Delta t_{i} =  t_{i}^{N+1} - t_i^{N}$. For $t\in[0, \Delta t_{i}], \, K(\gamma(t_{i}^N+ t)) \leq K_{neg},$
so by the Comparison Theorem   
    $$
    u_\gamma(t_i^{N+1})=
    u_\gamma(t_i^N + \Delta t_{i} ) \geq u_\gamma(t_i^N+\Delta t) \geq   u^*(\Delta t).  
    $$

For $t \in [\Delta t_{i},t_{i+1}^0-t_i^{N+1}]$, the curvature satifies $K(\gamma(t_{i}^N+ t)) \leq K_{pos}$. Using the above estimate for $ u_\gamma(t_i^{N+1})$, the Comparison Theorem and noting that $t_{i+1}^0 -t_i^{N+1} \leq T_{ret}$ and  $u'<0$ when $K>0$  gives 
    $$
    u_\gamma(t_{i+1}^0) = u_\gamma(t_i^{N}+\Delta t_i + (t_{i+1}^0 -t_i^{N+1} ) )\geq  u^*(\Delta t +(t_{i+1}^0 -t_i^{N+1} ))\geq u^*(\Delta t +T_{ret}) >0 .
    $$

Choosing the Kourganoff constant  
    $$
    m =  \min\{ u^*(\Delta t), u^*(\Delta t +T_{ret})\} = u^*(\Delta t +T_{ret})>0
    $$ 
allows us to satisfy all the conditions of the Kourganoff Theorem which in turn proves Theorem~\ref{thm:Anosov}. 
\end{proof}

Now we derive an easy-to-check condition that implies that condition (\ref{eqn:strictlyinvconecondition}) for the  strictly invariant cone condition holds. Let $u_{pos}$ be the solution of the Riccati equation with constant positive curvature $K(t) \equiv K_{pos}$ and initial condition $u_{pos}(0)=0$. A calculation gives
    \begin{align}\label{eqn:upos}
    u_{pos} (t) = -\sqrt{K_{pos}}\, \tan( \sqrt{K_{pos}} \, t) 
    \end{align}
providing $\sqrt{K_{pos}} \, t< \pi/2$.

Similarly, let $u_{neg}$ be the solution of the Riccati equation with constant negative curvature $K(t) \equiv K_{neg}$ and initial condition $u_{neg}(0)=0$. Then 
    \begin{align}\label{eqn:uneg}
    u_{neg}(t)= \sqrt{-K_{neg}}\tanh(\sqrt{-K_{neg} }\, t).
    \end{align}

\begin{lemma}\label{lemma:uknegukpos}
    The strictly invariant cone condition (\ref{eqn:strictlyinvconecondition}), 
        $$
        u^*(\Delta t +T_{ret}) >0, 
        $$
    will hold if (and only if) 
        $$
        u_{neg}(\Delta t )+ u_{pos}(T_{ret}) > 0. 
        $$ 
    Equivalently, 
        $$
        \sqrt{-K_{neg}}\tanh(\sqrt{-K_{neg} }\, \Delta t)  -\sqrt{K_{pos}}\, \tan( \sqrt{K_{pos}} \, T_{ret}) >0. 
        $$
\end{lemma}

\begin{proof}
The symmetry of the Riccati equation together with time reversal of the geodesic flow (i.e. $\gamma(p, v) (-t) = \gamma(p,-v)(t)$) implies that if $v(t)$ is a solution of the Riccati equation with curvature $K(t) \equiv K_{pos}$ and initial condition $v(0) = - u_{pos}(T_{ret})$ then $v(T_{ret}) = 0$. Thus, if our Riccati solution $u^*(t)$ satisfies $u^*(\Delta t) > v(0)$ then $u^*(\Delta t + T_{ret}) > v(T_{ret}) = 0$. 

The condition 
    $$
    u_{neg}(\Delta t )+ u_{pos}(T_{ret}) > 0
    $$  
is equivalent to $u^*(\Delta t) > -u_{pos}(T_{ret}) = v(0)$. 
\end{proof}

\section{Construction of surface}

\subsection{Disk arrangement on the flat plane with a finite horizon property}\label{sec:circlepacking}

Let $\bbR^2$ be the plane with the Euclidean metric $g_0$, and let $\D_0$ be the collection of closed disks of unit radius given by the hexagonal packing (pictured in Figure~\ref{fig:circlepacking}). Because the disks are tangent and $\bbR^2 \setminus \D_0$ consists of a collection of uniformly bounded regions, it is clear that $\S=\D_0$ has the $T_{ret}$-finite horizon property for some $T_{ret}$. Moreover, since any straight line must enter the \emph{interior} of $\D_0$, there must be an $\S'$ consisting of disks strictly interior to $\S$ and $T_{ret}$ such that $\S' \subset \S$ has the $T_{ret}$ strong finite horizon property. We determine such $\S'$ and $T_{ret}$ below.

\begin{figure}[h]
    \centering
    \includegraphics[width=.3\textwidth]{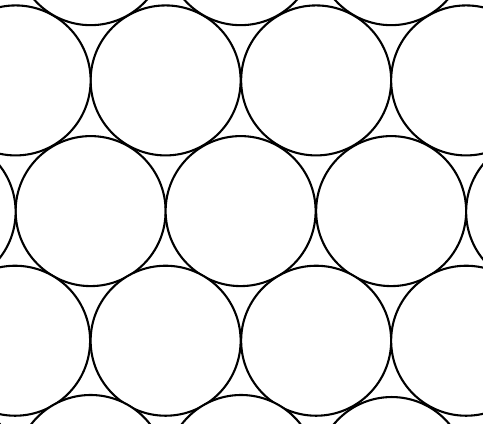}
    \caption{The hexagonal circle packing $\D_0 \subset \bbR^2$.}\label{fig:circlepacking}
\end{figure}

The following is an angle version of strong finite horizon.  

\begin{theorem}\label{thm:diskarrangementangletimebounds} 
    Every straight line geodesic that starts in $\bbR^2 \setminus \text{int}(\D_0)$ will enter $\D_0$ with angle $\theta \geq \pi/6$ in time $t\leq T_{ret}=2$, and this bound on the return time is sharp. 
\end{theorem}

When we talk of the angle $\theta$ of entering a disk, we mean the smaller of the two possible ways of measuring the angle.  
\begin{figure}[th]
    \centering
    \includegraphics[width=5in]{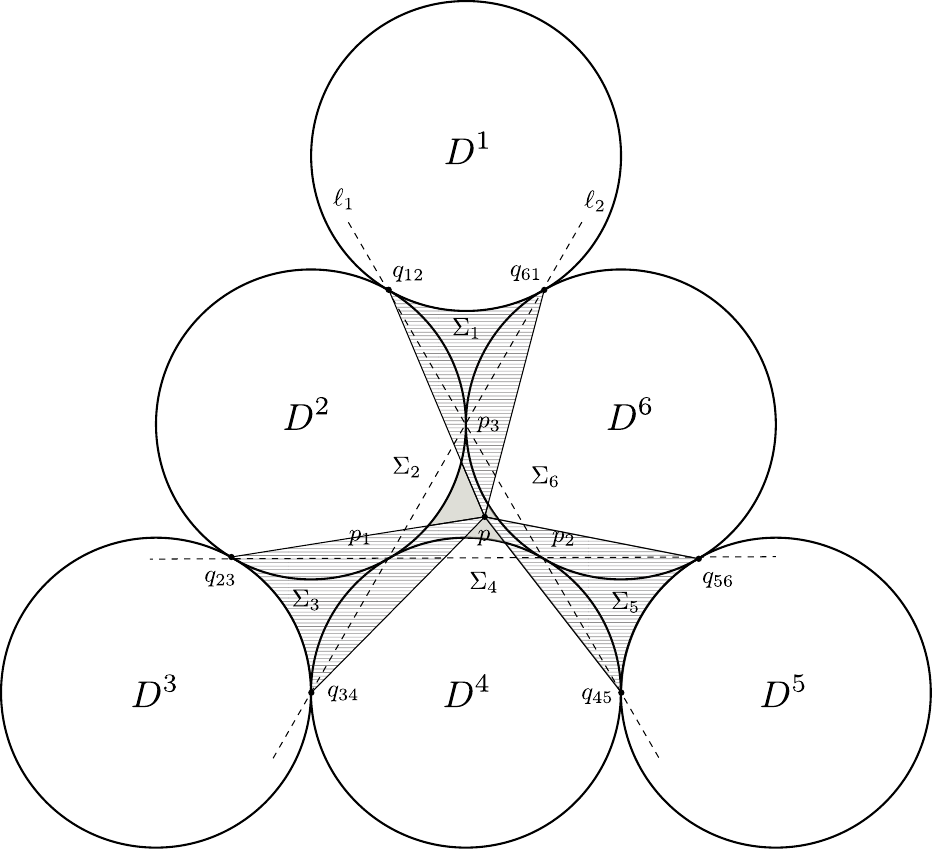} 
    \caption{Any straight line geodesic from the point $p$ enters one of the disks $D^i$, $1 \leq i \leq 6$, at an angle of at least $\pi/6$ within time $T_{ret}=2$.}
    \label{fig:my_label}
\end{figure}

\begin{proof}
Consider an arbitrary point $p \in \bbR^2 \setminus \text{int}(\D_0)$. It is in a region bounded by three disks that we label $D^2$, $D^4$, and $D^6$ in Figure \ref {fig:my_label}.  We consider three additional disks $D^1$, $D^3$, and $D^5$ that are each tangent to two of the inner disks. We will show that any geodesic starting at the point $p$ will enter one of these six disks within time $2$ with angle at least $\pi/6$.

Label the point of tangency between disks $D^i$ and $D^j$ as $q_{ij}$ for $j-i=1$ mod 6. A geodesic ray starting at point $p$ is in one of six sectors bounded by the rays from $p$ to $q_{12}$, $q_{23}$, $q_{34}$, $q_{45}$, $q_{56}$, and $q_{61}$. Let $\Sigma_i(p)$ be the set of unit vectors based at $p$ corresponding to the sector bounded by the rays from $p$ to $q_{i-1,i}$ and from $p$ to $q_{i,i+1}$ (with index arithmetic done mod 6). Note that $\Sigma_i(p)$ is a set of vectors whose rays enter $D_i$ with sizable angle of entry (we will prove below that this is at least $\pi/6$, regardless of which point $p$ we start at).

For a disk $D^i$ and a vector $v\in S_pM$, let $\theta(v,D^i)$ be the angle at which the ray starting at $p$ in the direction $v$ enters $D^i$.

First consider the sectors $\Sigma_1(p)$, $\Sigma_3(p)$, or $\Sigma_5(p)$. Without loss of generality, we make the argument for $\Sigma_1(p)$. Consider the two dotted lines $\ell_1 = \overleftrightarrow{p_2 p_3}$ and $\ell_2 = \overleftrightarrow{p_1 p_3}$ through the points of double tangency as labeled in the picture.  These lines cross $D^1$ at angle $\pi/6$. Because $p$ is in the triangle formed by $p_1$, $p_2$, and $p_3$, by comparison with $\ell_1$ we see that $\theta(\overrightarrow{p q_{61}},D^1) \geq \pi/6$, and similarly by comparison with $\ell_2$ we see that $\theta(\overrightarrow{p q_{12}},D^1) \geq \pi/6$. As the outgoing ray rotates from $\overrightarrow{p q_{61}}$ to $\overrightarrow{p q_{12}}$, the angle of entry into disk $D^1$ increases monotonically from $\theta(\overrightarrow{p q_{61}},D^1) \geq \pi/6$ to a maximum of $\pi/2$, and then decreases monotonically to $\theta(\overrightarrow{p q_{12}},D^1) \geq \pi/6$. This follows from keeping track of the angle of crossing with a fixed line, and then noting that the tangent line along a circle changes in such a way that it further increases (resp. decreases) the angle. Hence, the angle of entry is never less than $\pi/6$.

Now, consider a ray from point $p$ that enters $\Sigma_2(p), \Sigma_4(p)$, or $\Sigma_6(p)$. We will give the argument for $\Sigma_2(p)$. The incoming angle that the ray $\overrightarrow{p q_{12}}$ makes with $D^2$ is equal to the outgoing angle that this ray makes when leaving $D^2$. Due to the tangency of $D^1$ and $D^2$ at this point $q_{12}$, this outgoing angle at $D^2$ is equal to the incoming angle at $D^1 = \theta(\overrightarrow{p q_{12}},D^1)$. Thus $\theta(\overrightarrow{p q_{12}},D^2)=\theta(\overrightarrow{p q_{12}},D^1) \geq \pi/6$. Now the same argument as above shows that the angle of entry with $D^2$ of  rays in the blue region is never less than $\pi/6$.

Now we determine a bound on the time it takes a straight line to enter a disk as above. Without loss of generality, consider a geodesic through $p$ that intersects $D^1$ at a point $p^{in}$ in forward time and intersects $D^4$ at a point $p^{out}$ in backward time. The Euclidean distance $d(p, p^{in})$ between $p$ and $p^{in}$ satisfies
    \[
    d(p, p^{in})\leq d(p^{out}, p^{in}) \leq \max\{  d(p_{1}, q_{12}), d(p_{1}, q_{61})\} = d(p_{1}, q_{61})=2.
    \]

To see that the value $T_{ret}=2$ is sharp, take a sequence $ (p_n, v_n)$ with points $p_n$ on the boundary of  $D^4$   such that   $p_n \to p_{1}$. The vectors $v_n$ are chosen to be in the same direction as the vector $\overrightarrow{p_n p_3}$. The ray determined by $(p_n, v_n)$ first returns to a disk with angle $\theta \geq \pi/6$ when it enters $D^1$. As $p_n \to p_1$, this return time approaches 2. 
\end{proof}

We translate the angle condition for strongly finite horizon into a condition on how far into the disk the trajectories reach.  A straight line geodesic that enters a unit disk with angle $\pi/6$ will reach a minimum radial distance of $\cos (\pi/6)$ from the center of the disk. Measuring inward from the edge of the disk, we say that the trajectory has penetrated a distance
    \begin{align}
    \Delta r = 1- \cos (\pi/6) \label{eqn:DeltarDefn}
    \end{align}
into the disk. 

\begin{definition}\label{defn:D1set}
    Denote by  $\D_1$ be collection of disks with the same centers as $\D_0$ but with radii $1-\Delta r$.
\end{definition}

The disks $\D_1$ have the property that a line entering a $D_0$ at angle $\pi/6$ is tangent to $D_1\subset D_0$ at its point of maximum penetration.

\begin{corollary} \label{thm:D1StrongFiniteHorizon} \hfill
\begin{enumerate}
    \item The sets $\D_1 \subset \D_0$ have the $T_{ret}=2$ strong  finite horizon 
    property. Translated to the optimal time-version statement, $\D_0$ has the $(T_{ret}=2,\Delta t=1)$ strong finite horizon property.
    \item The set $\D_1$ has the $(T_{ret}=3, \Delta t=0)$ strong finite horizon property. 
\end{enumerate}
\end{corollary}

\begin{proof}
A trajectory that enters $\D_0$ with angle $\pi/6$ will take time $\frac{1}{2}$ until it intersects $\D_1$ and then another $\frac{1}{2}$ until it passes out of $\D_0$ thus spending time 1 in $\D_0$. Trajectories that enter with a larger angle will spend more time in $\D_0$.

Every trajectory that starts outside of $\D_1$ will intersect $\D_1$ in time at most 3. Some of those trajectories intersect $\D_1$ tangentially and hence do not spend positive time inside $\D_1$. 
\end{proof}

More generally, it will be useful to replace  $\D_0$ with disks $\D_{\rho}\subset \D_0 $ of smaller than unit radius. We will eventually have tubes of negative curvature sitting above the disks and we will want a strictly negative upper bound $K_{neg}$ on the curvature inside the disk $\D_{\rho}$.

\begin{definition}\label{def:Drho}
    For $0\leq \rho \leq 1$, define $D_{\rho}$ to be the disk of radius $r(\rho)= 1 - \rho \Delta r$ and $\D_\rho$ to be the set of disks $D_\rho$ with the same centers as the hexagonal packing $\D_0$. 
\end{definition} 

The sets $\D_\rho, \, 0 < \rho < 1$ also  have the strong finite horizon property. We will be particularly interested in the case $\rho = \frac{1}{2}$. 

\begin{proposition}\label{prop:DrhoStrongFiniteHorizon}
    The sets $\D_1 \subset \D_{\rho}$ have the  $T_{ret}(\rho)$ strong finite horizon property with optimal value
        $$
        T_{ret}(\rho) = 3 -  2 \sqrt{(1 -\rho \Delta r)^2 - 3/4}.
        $$
    This corresponds to $\D_{\rho}$ having a $(T_{ret}(\rho),\Delta t(\rho))$ strong finite horizon property with optimal value
        $$
        \Delta t(\rho) = 2\sqrt{(1 -\rho \Delta r)^2 - 3/4}.
        $$
    For $\rho = \frac{1}{2}$, to five decimal places and rounding up (respectively down), we get
        \begin{equation}\label{eqn:boundonT12}
        \textstyle
        T_{ret}(\frac 12)= 2.30571 ~~~\text{and}~~~ \Delta t(\frac 12) = 0.69429.
        \end{equation}
\end{proposition}

\begin{proof} 
First consider a trajectory that is tangent to a disk $D_1$. It entered the disk $D_0$ with angle $\theta =\pi/6$. Assume that  this trajectory is horizontal in $(x,y)$ coordinates. Then its $y$-coordinate is $y=1-\Delta r = \pm \cos(\pi/6) = \pm \sqrt{3}/2$. This trajectory will intersect the  disk of radius $r$ at  points $(x,\pm \sqrt{3}/2)$ where
    $$
    x^2 + (\sqrt{3}/2)^2 = r^2. 
    $$
so that 
    $$
    x = \pm \sqrt{r^2-(1-\Delta r)^2} = \pm \sqrt{r^2 - 3/4}.
    $$
For $r = 1 -\rho \Delta r$, the positive root equals 
    $$
    x = \sqrt{(1 -\rho \Delta r)^2 - (1-\Delta r)^2} = \sqrt{(1 -\rho \Delta r)^2 - 3/4}. 
    $$

Since the  tangential trajectory spends time 1 inside of $D_0$, the time it spends from entering $D_0$ to reaching $D_{\rho}$ is $ \frac 12 - x$. Since it takes time $2$ between returns to $D_0$, it will take time 
    $$
    T_{ret}(\rho)=2 + 2( \frac 12 -x)=  3 - 2x 
    $$ 
between returns to $D_{\rho}$.

If a trajectory enters $D_1$ but not tangentially, then it must have crossed $D_0$ with angle $\theta$  satisfying $\theta > \pi/6$. It will then reach $D_{\rho}$ more quickly than the tangential trajectory. Since the time for ``good" returns to $D_0$ is at most 2,  its time between returns to $D_{\rho}$ will be strictly less than $T_{ret}(\rho)$. 

For the value of $\Delta t(\rho)$, note that the trajectory tangent to $\D_1$ spends time 
    $$
    2x = 2\sqrt{(1 -\rho \Delta r)^2 - 3/4}. 
    $$
inside $\D_{\rho}$ and trajectores that enter $\D_1$ non-tangentially will spend more. 
\end{proof}

\subsection{Constructing the model space} \label{sec:tubes}

We create  the model space $\M_0$ by attaching identical  negatively curved tubes to the disks $\D_0$ given in Section~\ref{sec:circlepacking}  on a top and bottom copy of the plane. We denote by $\T_0$ the collection of all these tubes. The model space comes with the metric 
    \begin{equation}\label{eqn:modelspacemetric}
    g_{\M_0} = g_{Eucl} \big|_{\M_0}
    \end{equation} 
that it inherits from being an embedded surface in Euclidean $\bbR^3$.
 
The set of tubes  $\T_0 \subset \M_0$   will have a strong finite horizon property (since they overlay the disks $\D_0$ and we have not changed the metric outside of those disks). 

We introduce the following   notation. 
\begin{definition}\label{def:Tubesrho}
    Let  $\T_{\rho}\subset \T_0 $ be the part of the tubes $\T_0$ that sits above the disks $\D_{\rho}$. 
\end{definition}

\subsubsection{Tube blueprint: a surface of revolution}\label{subsection: tubeconstruction}

In this section, we explicitly construct our negatively curved tube. It will be a surface of revolution, have maximum radius $R=1$, and smoothly connect to the flat plane at the boundaries of the corresponding disks. The tube will be  symmetric with respect to reflection about the plane half-way up.  

First we will make a $C^1$ smooth construction and then we will apply a standard smoothing argument using a partition of unity function to make the tube - plane attachement $C^{\infty}$ smooth.

Assume a tube based on a disk centered at the origin and give $(u,v,w)$ coordinates to $\bbR^3$. Define the profile curve of the tube in the $(u,w)$-plane by using the left half  of the circle of radius $\frac{1}{2}$ centered at the point $(u=1, w = \frac{1}{2})$; see Figure~\ref{fig:angularcoordinates}. Then rotate this profile curve around the $w$ axis to get the tube.

Express the tube in terms of angular coordinates given by $Y: \bbR^2 \to \bbR^3$:
    $$
    Y(\psi, \theta) = (\pr(\psi) \cos \theta, \pr(\psi) \sin \theta, \ph(\psi))
    $$
where 
    $$
    \pr(\psi) = 1-\frac 12 \cos \psi \text{ and } \ph(\psi)= \frac 12 +\frac 12 \sin \psi
    $$
for $\psi \in (-\pi/2, \pi/2)$ and $\theta \in (0, 2\pi)$.

\begin{figure}[h!]
\centering
\begin{tikzpicture}
\begin{axis}[
    width=8cm,
    axis equal image,
    xmin=0, xmax=1.1,
    ymin=0, ymax=1.05,
    axis lines=left,
    xlabel={$u$},
    ylabel={$w$},
    xtick={0,0.5,1},
    ytick={0,0.5,1},
    grid=both,
    grid style={line width=0.2pt, gray!40},
    samples=200,
    domain=-pi/2:pi/2,
    thick
]

\addplot[
    parametric,
    smooth,
    color=red,
]
(
    {1 - 0.5*cos(deg(x))},
    {0.5 + 0.5*sin(deg(x))}
);

\def\psideg{35}

\coordinate (C) at (axis cs:1,0.5);

\coordinate (L) at (axis cs:0.5,0.5);

\coordinate (P) at (
    axis cs:{1 - 0.5*cos(\psideg)},
            {0.5 + 0.5*sin(\psideg)}
);

\addplot[only marks, mark=*] coordinates {(1,0.5)};

\draw[gray, dashed] (C) -- (L);
\draw[blue, thick, -{Latex}] (C) -- (P);

\node[blue, above left] at (P) {$(r(\psi),h(\psi))$};
\node[blue] at (axis cs:0.78,0.57) {$\psi$};

\end{axis}
\end{tikzpicture}
\caption{The tube profile curve determined by the angular coordinate $\psi$.}\label{fig:angularcoordinates}
\end{figure}
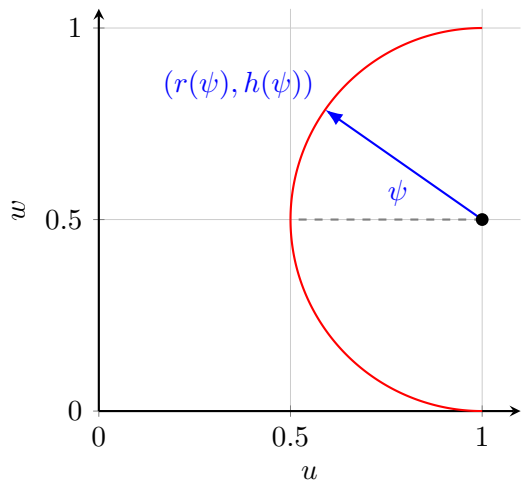

The tube created using circular arcs attaches $C^1$, but not $C^2$, smoothly to the plane. To make the attachment $C^{\infty}$ smooth, use a standard smoothing argument to   modify the tube in a small neighorhood of its boundary. To facilitate the smoothing process, express the tube in polar coordinates: 
    $$
    Y(r,\theta) = (r \cos \theta, r \sin \theta, \pf(r) ).
    $$

Solving the equation of the circle 
    $$
    (r-1)^2+ (w-\frac{1}{2})^2 = \frac{1}{4}, \, 0\leq w \leq 1, 
    $$
for the bottom half of the tube gives  
    \begin{equation}\label{eqn:unsmoothedtubefunction}
    w= \pf(r) = \frac{1}{2}-\sqrt{\frac{1}{4} - (r-1)^2}, \,\frac{1}{2}\leq r \leq 1. 
    \end{equation}

\subsubsection{Tube smoothing}\label{sec:smoothing}

Since the tube must attach to the flat region outside of the tube, we would like the profile function   to go smoothly to zero at $r=1$. We have that $\pf(1)=0, \, \pf'(1)=0$ but $\pf''(1) = 2 \neq 0$. Thus $\pf$ could be extended $C^1$ but not $C^2$ to the flat region. 
 
In what follows, we use a partition of unity function to modify $\pf$ to produce a smooth profile function $f$. This construction is essentially a horizontal shift of $\pf$ by $\delta_h$, a vertical shift by $\delta_v$, and a smoothing over the   interval $r\in [1-\epsilon,1]$ where $\epsilon$ can be made arbitrarily small. 
 
\begin{proposition}\label{prop:smoothprofilefunction}
    Given any $\epsilon>0$, there exists values $\delta_h$ and $\delta_v$ with $0< \delta_h , \delta_v <\epsilon$ and a  $C^{\infty}$ function $f(r)$ such that 
    \begin{enumerate}
        \item $f$ is defined on $r \geq \frac 12 - \delta_h $,
        \item $f(r) = \pf(r + \delta_h ) + \delta_v$ for $\frac 12 - \delta_h  \leq r \leq 1 - \epsilon$,
        \item $f'(r) < 0$, $f''(r)>0$, and $f'''(r)<0$ for $r \in [1-\epsilon,1)$,
        \item the functions $|f'(r)|, |f''(r)|, |\frac{f'(r)}{r}|$ are decreasing for $r\in  [\frac 12 - \delta_h, 1)$,
        \item $f(r) \equiv 0$ for $r \geq 1$.
    \end{enumerate} 
\end{proposition}

\begin{remark}\label{rmk:doteq}
    Since $\delta_h$ and $\delta_v$ can be made arbitrarily small, the differences $f(r)-\pf(r)$, $f'(r)-\pf'(r)$ and $f''(r)-\pf''(r )$ can be made arbitrarily small away from the (arbitrarily small) smoothing region. Hence, for numerical calculations that keep track of only a finite number of decimal places, we can ignore this difference when computing upper bounds by rounding up. We will write, e.g., $f(r) \doteq \pf(r)$ to indicate this numerical equivalence. 
\end{remark}

The bottom half of the smoothed tube is given in polor coordinates by
    $$
    Y (r, \theta) = (r \cos \theta, r \sin \theta, f(r)), \qquad  \frac{1}{2} - \delta_h  \leq r \leq 1.
    $$
The top half of the tube is defined with a similar formula but with
    \begin{equation}\label{def:ftop}
    f^{\text{top}}(r) = 1+2\delta_v-f(r), \qquad \frac{1}{2} - \delta_h  \leq r \leq 1. 
    \end{equation}
so that $\frac 12 + \delta_v \leq  f^{\text{top}}(r) \leq 1+2\delta_v$. The $C^\infty$ smooth tube $T_0$ is then defined by the combination of $f$ and $f^{\text{top}}$. 

In angular coordinates the smoothed  tube can be written as
    $$
    Y (\psi, \theta) = (r(\psi) \cos \theta,  r(\psi) \sin \theta, h(\psi)), \qquad -\pi/2 \leq \psi \leq \pi/2. 
    $$
    
The smooth polar coordinate function  $f$   corresponds to smooth angular coordinate functions $r(\psi)$ and $h(\psi)$ which produce tubes that join smoothly to the flat plane. 

Away from the zone of smoothing, the profile of the tube is still an arc of an circle of radius $\frac{1}{2}$. However the smoothing procedure has shifted the center of the circle from $(u=1, w = \frac12)$ to $(u= 1-\delta_h, w = \frac 12 + \delta_v)$. Hence in this region, the new smooth angular coordinate functions  are translations of the original functions:  
    \begin{equation} \label{eqn:randhsmoothing}
    r(\psi) = (1-\delta_h) -\frac 12 \cos \psi \text{ and } h(\psi)= (\frac 12 + \delta_v)  +\frac 12  \sin \psi
    \end{equation}
providing $-\pi/2 + \tilde \epsilon < \psi < \pi/2 - \tilde \epsilon$ for some small $\tilde \epsilon >0$. 

\begin{remark}\label{rmk:doteqforrandh}
    As is the case for $f$, the differences $r(\psi)-\pr(\psi)$, and $h(\psi)-\ph(\psi)$ can be made arbitrarily small. Hence, $r(\psi) \doteq \pr(\psi)$ and $h(\psi) \doteq \ph(\psi)$. We note that, away from the (arbitrarily small) zone of smoothing, in this case the derivatives are exactly equal: $r'(\psi) = \pr'(\psi)$, and $h'(\psi)=\ph'(\psi)$.
\end{remark}

\begin{definition}
    The model space $\M_0$ is two copies of the plane joined by   the collection $\T_0$ of smooth tubes which are attached to the disks $\D_0$. The model space can be written as $\M_0 = \M^{bot}_0 \cup \M^{top}_0$ where $\M^{bot}_0$ is the subset of $\M_0$ with $0 \leq w\leq \frac 12+\delta_v$ and $\M^{top}_0$ with $ \frac 12+ \delta_v \leq w \leq 1+ 2 \delta_v$.
\end{definition}

\subsubsection{Proof of Proposition~\ref{prop:smoothprofilefunction}  }  \label{sec:proof of Proposition}

\begin{proof}[Proof of Proposition~\ref{prop:smoothprofilefunction}]
Choose $\epsilon >0$. For any $a \in (1-\epsilon, 1)$, we define a partition of unity function $P_a(r)$ to be a $C^{\infty}$ smooth function satisfying
    \begin{align*}
    P_a(r) &=1, \, r\leq a \\
    P_a'(r)& < 0, \,a < r < a+\epsilon \\
    P_a(r)& = 0, \, r \geq \epsilon+a .
    \end{align*}
We use $P_a$ to smoothly connect $\pf''(r)$ to the identically zero function for $r$ near $1$: 
    $$
    \pf_1 ''(r) =  P_a(x) \pf''(r),    \quad  \frac{1}{2}{} < r < \frac{3}{2}, 
    $$
where for $1 < r < \frac 32$, $\pf''$ is defined via the continuation of the circle. The resulting $\pf_1''$ is a smooth function with the properties
    $$
    \pf_1 ''(r) = \begin{cases}
    \pf''(r) & r \in (1/2,a] \\
    \text{decreasing} & r \in (a, a+\epsilon] \\
    0 & r\geq a+ \epsilon >1
    \end{cases}.
    $$
Define $\pf_1'$ for $r> 1/2$ via integration:
    $$
    \pf_1'(r) = \pf'(a) + \int_a^r \pf_1''(\rho)~d\rho.
    $$

Choose a value of $a\in (1-\epsilon, 1)$ for which 
    $$\pf_1'(a+\epsilon) = 0.$$
The existence of such an $a$ follows  by continuity:  for $a=1$, $\pf_1'(a+\epsilon) > 0$, while for $a=1-\epsilon$, $\pf_1'(a+\epsilon) < 0$ as can be seen by comparing with  $\pf',\pf''$. 

We now have a smooth function $\pf_1'$, and values for $\epsilon$ and $a$, with the properties that 
    $$
    \pf_1 '(r) = \begin{cases}
    \pf'(r) & r \in (1/2,a] \\
    < \pf'(r) & r \in (a,a+\epsilon] \\
    0 & r\geq a+ \epsilon >1
    \end{cases}.
    $$
Integrating again gives 
    $$
    \pf_1(r) = \pf(a) + \int_a^r \pf_1'(\rho)~d\rho, 
    $$
resulting in a smooth function $\pf_1$ with the properties
    $$
    \pf_1 (r) = \begin{cases}
    \pf(r) & r \in (1/2,a] \\
    \text{decreasing} & r \in (a,a+\epsilon] \\
    -\delta_v <0 & r\geq a+ \epsilon >1
    \end{cases}.
    $$
where $-\delta_v = \pf(a) + \int_a^{a+\epsilon} \pf_1'(\rho)~d\rho.$

Set $\delta_h  = (a+\epsilon - 1) < \epsilon$. We shift the graph of $\pf_1$ to the left by $\delta_h $ and up by $\delta_v$ to get a final smooth profile curve
    $$
    f(r) = \pf_1(r+\delta_h) +\delta_v.
    $$
The smooth function $f$ is defined for $r \geq 1/2-\delta_h $ and has the properties that 
    $$
    f (r) = \begin{cases}
    1/2 + \delta_v & r = 1/2 - \delta_h \\
    \pf(r+\delta_h)+\delta_v & 1/2-\delta_h < r < a-\delta_h = 1 - \epsilon \\
    \pf_1(r+\delta_h)+\delta_v & r \in [1 - \epsilon, 1] \\
    0 & r \geq 1
    \end{cases}.
    $$    
For $a-\delta_h  < r < 1$, we have $f'<0$ and $f''<\pf''$. Note that because $\delta_h  = (a+\epsilon - 1)$, the interval $(a-\delta_h , 1)$ has length less than $\epsilon$.
 
Note also that
    \begin{align*}
    \delta_v  &\leq \left| \pf(a) + \int_a^{a+\epsilon} \pf_1'(\rho)~d\rho \right| \\
    &\leq \left|  \int_a^{a+\epsilon} \pf_1'(\rho)~d\rho \right|  \\
        & \leq \int_a^{a+\epsilon} \left| \pf_1'(\rho)\right| ~d\rho \\
        & \leq \epsilon \left| \pf'(a) \right| \\
        & \leq \epsilon \left| \pf'(1-\epsilon) \right| \\
        & \leq \epsilon
    \end{align*}
for $\epsilon$ small, such as $\epsilon \leq 0.1$.

Statement (4) now follows from the explicit formula for $f(r), r\in [\frac 12 - \delta_h, 1-\epsilon]$, together with Statement (3). 
\end{proof}

\subsubsection{Tube scaling}\label{sec:tubescaling}

Any set of tubes of negative curvature attached to our finite horizon pattern of disks in the plane will itself satisfy a finite horizon condition since the metric is flat outside of the disks. However, we would not have explicit bounds on the finite horizon parameters. In order to retain good control on where geodesics go, we scale (vertically) the tubes so that the metric away from the center of the tubes is a small perturbation of the flat metric.

To make the outer parts of the tube nearly flat, we vertically scale the tube by $s_0$ via a map 
$$X^{sc}_{s_0}(u,v,w) = (u, v, s_0 w).$$ 
 
The bottom half of the scaled tube in polar coordinates is
    $$
    Y_{s_0} (r, \theta) = (r \cos \theta, r \sin \theta, {s_0} f(r))
    $$
and a similar formula for the top using $f^{\text{top}}$.
In angular coordinates the scaled tube is
    $$
    Y_{s_0} (\psi, \theta) = (r(\psi) \cos \theta,  r(\psi) \sin \theta, s_0 h(\psi)).
    $$
 
For any value of $s_0$, this procedure yields a vertically scaled model surface $\M_{s_0}= X^{sc}_{s_0}(\M_0 )$ which is  non-compact and embedded in $\bbR^3$. Equipping this surface with the metric that it inherits from the Euclinear metric on $\bbR^3$, its  geodesic flow is easily seen to be Anosov.

Pulling this metric back gives a metric $g_{s_0}$ on the model space $\M_0$:  
    \begin{equation}\label{eqn:gs0}
    g_{s_0} = (X^{sc}_{s_0})^* g_{Eucl}|_{\M_{s_0}}.
    \end{equation}
The geodesic flow on $(\M_0, g_{s_0})$ is Anosov (see \cite{Donnay-Visscher-18}).  Hence in the the space of metrics on $\M_{0}$, there is an open neighborhood of $g_{s_0}$ all of  whose metrics are  Anosov (i.e. the geodesic flow they generate is Anosov).

\subsection{Embedding map and pullback metric}\label{sec:embeddingmap}

Now we define an ``embedding map'' $X^{emb}:\bbR^3 \to \bbR^3$ which when applied to the scaled model space $\M_{s_0} \subset \bbR^3$ will yield an immersed--and for some parameters embedded--surface.

In  the standard Euclidean coordinates, the map  $X^{emb}$ is given by
    $$
    X^{emb}(u, v, w) = (X_1(u, v, w), X_2(u, v, w), X_3(u, v, w))
    $$
with 
    \begin{align}
    X_1(u, v, w)
        & = \left( R_1 +( R_2+w) \cos\left( \frac{v}{R_2}\right) \right) \cos\left( \frac{  u }{R_1}\right) \notag \\
    X_2(u, v, w) 
        &= \left(R_1 +( R_2+w) \cos\left( \frac{  v}{R_2}\right) \right)\sin\left( \frac{  u }{R_1}\right) \label{eqn:embeddingmap}\\
    X_3(u, v, w) 
        &=  ( R_2+w) \sin\left( \frac{  v}{R_2}\right) \notag
    \end{align}
for parameters $R_1$ and $R_2$. This takes a plane $w=\text{constant}$ and maps it to an immersed torus, with major radius $R_1$ and minor radius $R_2+w$. Note that as long as $R_2+w < R_1$, this yields an embedded torus.

For reasons that will become clear below, we define a one parameter family of radii via the following explicit   functions: 
    \begin{equation}\label{eqn:r1r2functions}
    R_1(s_1) = \frac 1{\pi s_1^2}, \quad R_2(s_1)=\frac{1}{\pi s_1}  .
    \end{equation}
We denote the embedding map with this choice of radii by $X^{emb}_{s_1}$.

Our final mapping $ X_{(s_0, s_1)}$ from the model space $\M_0$ into $\bbR^3$ is defined via the composition of the vertical scaling followed by the embedding map. Setting $s=(s_0, s_1)$, we have 
    $$
    X_s =X^{emb}_{s_1}\circ X^{sc}_{s_0}. 
    $$

\begin{goal*}
    The goal of our paper is to show that there exist $s=(s_0, s_1)$ values for which the set $X_s(\M_0)$ is an embedded closed surface whose geodesic flow is Anosov, and that such $(s_0, s_1)$ values are computable.
\end{goal*}

The matrix $DX_s$ is non-singular, so the map $X_s$ is an immersion of the model space into Euclidean $\bbR^3$. Hence we can pull back the Euclidean metric $g_{Eucl}$ on $\bbR^3$ by the map $X_s$ to get a metric $g_s$ on the model space $\M_0$ defined by 
    \begin{equation}\label{eqn:pullbackmetric}
    g_s  = X_s^*\, (g_{Eucl} \big|_{X_s(\M_0)}).
    \end{equation}

We study  the geodesic flow on the Riemannian surface $(\M_0, g_s)$ and determine values of $s$ for which the flow is Anosov and for which the set $X_s(\M_0)$ is an embedded surface.

In later parts of the paper, we will need to make use of $DX^{emb}_{s_1}$, the Jacobian matrix of partial derivatives for the map $X^{emb}_{s_1}$, which is given by 
    $$
    DX^{emb}_{s_1}= \left(
        \begin{array}{ccc}
         -\frac{\left(R_1+ (R_2+w) \cos \left(\frac{v}{R_2}\right)\right) \sin \left(\frac{u}{R_1}\right)} {R_1}
         & -\frac{(R_2+w) \sin \left(\frac{v}{R_2}\right) \cos \left(\frac{u}{R_1}\right)}{R_2} 
         & \cos \left(\frac{v}{R_2}\right) \cos \left(\frac{u}{R_1}\right) \\
         \frac{\left(R_1+ (R_2+w) \cos \left(\frac{v}{R_2}\right) \right) \cos \left(\frac{u}{R_1}\right)}{R_1} 
         & -\frac{(R_2+w) \sin \left(\frac{v}{R_2}\right) \sin \left(\frac{u}{R_1}\right)}{R_2} 
         & \cos \left(\frac{v}{R_2}\right) \sin \left(\frac{u}{R_1}\right) \\
         0
         & \frac{(R_2+w) \cos \left(\frac{v}{R_2}\right)}{R_2} 
         & \sin \left(\frac{v}{R_2}\right) \\
        \end{array}
        \right)
    $$
with the radii $R_i=R_i(s_1)$.

\begin{proposition} \label{lemma:metricdecomposition}
    The pullback metric $g_s$ on $\M_{_0}$ can be decomposed: 
        $$
        g_{(s_0,s_1)}   = g_{s_0}  + \Delta g_{(s_0,s_1)}, 
        $$  
    where $g_{{s_0}}$ is given by equation~(\ref{eqn:gs0}). 
    The term
        $$
        \Delta g_{(s_0,s_1)}\to 0 \, \quad\text{ as }\quad  s_1\to 0, 
        $$ 
    and this convergence is uniform for $s_0\in[0,1]$. 
\end{proposition}

This proposition provides the essential ingredient of the proof in \cite{Donnay-Visscher-18}: that for $s_1$ sufficiently small, the pull-back metric $g_s$ must be in the open neighborhood of Anosov metrics around the Anosov model space metric $g_{s_0}$ and so the surface $X_{s}(\M_{0})$ will have an Anosov geodesic flow. For certain of these $s$ values, the surface will also be embedded.    In the present paper, we essentially quantify the size of this open neighborhood so that we can say how small the parameter $s_1$ needs to be in order to have an Anosov geodesic flow. This analysis will involve getting bounds on the size of $\Delta g_{(s_0,s_1)}$. The proof uses the following lemma. 
 
\begin{lemma}\label{lemma:DXdecomp}
    The  differential $DX=DX^{emb}_{s_1}$ can be written as
        $$
        DX = Rot + \Delta DX 
        $$
    where $Rot$ (thought of as a rotation matrix) is an isometry given by 
        $$
        Rot = 
        \left(
        \begin{array}{ccc}
         -\sin \left(\frac{u}{R_1}\right) & -\cos \left(\frac{u}{R_1}\right) \sin \left(\frac{v}{R_2}\right) & \cos \left(\frac{u}{R_1}\right) \cos \left(\frac{v}{R_2}\right) \\
         \cos \left(\frac{u}{R_1}\right) & -\sin \left(\frac{u}{R_1}\right) \sin \left(\frac{v}{R_2}\right) & \sin \left(\frac{u}{R_1}\right) \cos \left(\frac{v}{R_2}\right) \\
         0 & \cos \left(\frac{v}{R_2}\right) & \sin \left(\frac{v}{R_2}\right) \\
        \end{array}
        \right)
        $$
    and $\Delta DX$ is a matrix whose entries are given by 
        $$
        \Delta DX = 
        \left(
        \begin{array}{ccc}
         -\frac{(w+R_2)\sin \left(\frac{u}{R_1}\right) \cos \left(\frac{v}{R_2}\right) }{R_1} & -\frac{\cos \left(\frac{u}{R_1}\right) \sin \left(\frac{v}{R_2}\right) w}{R_2} & 0 \\
         \frac{(w+R_2)\cos \left(\frac{u}{R_1}\right) \cos \left(\frac{v}{R_2}\right) }{R_1} & -\frac{\sin \left(\frac{u}{R_1}\right) \sin \left(\frac{v}{R_2}\right) w}{R_2} & 0 \\
         0 & \frac{\cos \left(\frac{v}{R_2}\right) w}{R_2} & 0 \\
        \end{array}
        \right)
        $$
\end{lemma}

\begin{corollary}\label{cor:approachingrotation}
    Let the radii $R_1, R_2$ be functions of a parameter $s_1$ and let $w\in R$ be in a bounded region. If $ R_1(s_1), R_2(s_1)\to \infty$ and $R_2(s_1)/R_1(s_1) \to 0$ as $s_1 \to 0$, then $\Delta DX \to 0$ uniformly.
\end{corollary}

\begin{proof}[Proof of Proposition~\ref{lemma:metricdecomposition}]
Let $\zeta = DX^{sc *}_{s_0} (\zeta_0), \eta= DX^{sc *}_{s_0} (\eta_0) $. Using the notation $DX = DX^{emb}_{s_1}$,
    \begin{align*}
    g_{(s_0,s_1)}(\zeta_0,\eta_0) &=  g_{Eucl}(\zeta, \eta)=
     g_{Eucl}(DX \zeta_0, DX \eta_0) \\
     &= 
                    g_{Eucl}(Rot\, \zeta_0, Rot \,\eta_0) + g_{Eucl}(Rot \,\zeta_0, \Delta DX \eta_0) \\ & \hspace{.8in} + g_{Eucl}(\Delta DX \zeta_0, Rot \,\eta_0) + g_{Eucl}(\Delta DX \zeta_0, \Delta DX \eta_0) \\
                    & = g_{s_0} (\zeta_0,\eta_0) +\Delta g_{(s_0,s_1)} (\zeta_0,\eta_0),
    \end{align*} 
where we use that $Rot$ is an isometry and have $\Delta g_{(s_0,s_1)}$ denote the sum of the three terms which contain $\Delta DX$. Since the maps $Rot$ are uniformly bounded  and, by Corollary~\ref{cor:approachingrotation},  $\Delta DX \to 0$  as $s_1 \to 0$ uniformly  for $s_0\in [0,1]$ , we conclude that $\Delta g_{(s_0,s_1)}\to 0$ unformly. 
\end{proof}

\subsection{Embedded surface and its genus}\label{sec:genuscalculation}
 
For any values of $R_1$ and $R_2$, the map $X$ is well-defined on $\bbR^3$, and the region $[0, 2 \pi R_1] \times [0, 2 \pi R_2] \times \{w_0\} \subset \bbR^3$ is mapped once around an embedded  torus for any $w_0 > 0$ and $R_2+w_0 < R_1$. In particular, for $0<R_2<R_1-1$, the region $[0, 2 \pi R_1] \times [0, 2 \pi R_2] \times [0,1]$ will be mapped to a thickened torus in $\bbR^3$. However, this map may not be periodic when restricted to $\M_0$ due to the arrangement of the connecting tubes (see Section \ref{sec:circlepacking}). To produce an embedding, we need a match between the periodicity of the map and the periodicity of the fundamental region of the model space.
   
For a fixed value of $(R_1, R_2)$, the set $u\in [0, 2\pi R_1), v\in [0, 2 \pi R_2)$ will get mapped one-to-one onto the torus. The fundamental region of the model space hexagonal packing is the rectangle $u\in [0, 2], v\in[0, 2\sqrt3]$. Thus for the two periodicities to match, we need
    $$
    2m = 2 \pi R_1(s_1), \quad 2\sqrt3 n = 2 \pi R_2(s_1)
    $$
for some integers $m, n$ (the number of copies of the $[0,2] \times [0,2\sqrt 3]$ rectangle in the horizontal and vertical directions that will form the fundental region for the map $X_s$). The parametrization 
    \begin{equation*}
    R_1(s_1) = \frac 1{\pi s_1^2}, \quad R_2(s_1)=\frac{1}{\pi s_1}       
    \end{equation*}
gives the condition
    $$
    m= \frac{1}{s_1^2}, \quad \sqrt3 n = \frac{1}{s_1}.
    $$
Solving gives
    $$
    s_1 = \frac{1}{\sqrt3 n}
    $$
and then
    $$
    m = 3n^2.   
    $$

\begin{lemma}\label{lemma:genusformula}
    If $s_1 = \frac{1}{\sqrt3 n}$ and $m = 3n^2$, the resulting embedded surface has genus
        $$g =  6 n^3 +1$$
\end{lemma}

\begin{proof}
In our construction, we have two flat tori connected by tubes. The genus of this surface is given by
    $$g = \rm{number\,  of\,  tubes} +1.$$
In each fundamental region there are two tubes so if there are $ m \times n$ copies of the fundamental region being mapped 1-1 onto the torus, then the genus of the embedded surface they generate is 
    $$g = 2mn + 1 = 6 n^3 +1.$$  
\end{proof}

\section{Coordinates on the embedded surface}\label{sec:coordinates}
 
In this section we introduce and examine three different coordinate systems which we will use to analyze the geometry of the surface: rectangular, angular and polar. We invite the reader to skim or skip this section initially and return to it as the material is needed in later sections. 

Coordinates on (a subset of) the  model space $\M_0$ are given by a map 
    $$
    Y_0 : \mathscr{S} \subset \bbR^2 \to \M_{0} \subset \bbR^3.
    $$
Coordinates on (a subset of) the scaled model space $\M_{s_0} = X^{sc}_{s_0}(\M_0)$ are given by a map 
    $$
    Y_{s_0} = X^{sc}_{s_0} \circ Y_0: \mathscr{S} \subset \bbR^2 \to \M_{s_0} \subset \bbR^3;
    $$
Coordinates on (a subset of) the final space  $\M^{emb}_{s} = X_{s}(\M_0)$ are given by a map 
    $$
    Y_{(s_0, s_1)} = X^{emb}_{s_1}\circ X^{sc}_{s_0} \circ Y_0: \mathscr{S} \subset \bbR^2 \to \M^{emb}_{s} \subset \bbR^3;
    $$
\begin{figure}
    \centering
    \begin{tikzcd}[row sep=large, column sep=large]
    \text{model space} & \M_0 \subset \bbR^3 \ar[r,"X^{sc}_{s_0}"] \ar[rr,bend left=40,"X_{(s_0,s_1)}"] & \M_{s_0} \subset \bbR^3 \ar[r,"X^{emb}_{s_1}"] & \bbR^3 \\
    \text{coordinates} & \mathcal{U} \subset \bbR^2 \ar[u,"Y_0"'] \ar[ur,bend right=5,"Y_{s_0}"] \ar[urr,bend right=10,"Y_{(s_0,s_1)}"']
    \end{tikzcd}
    \caption{Coordinate maps for various spaces.}
\end{figure}
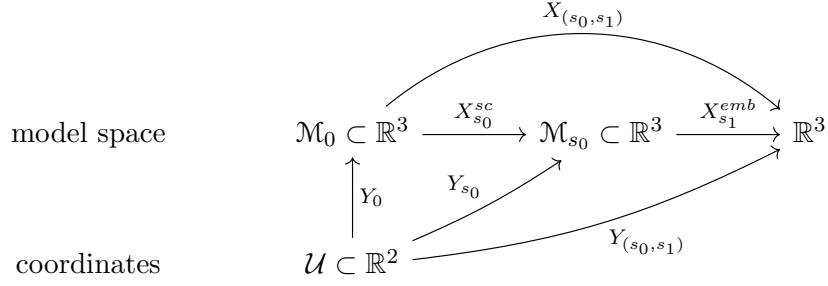
We will use $(\sigma, \tau) \in U$ as notation for generic coordinates and write
    $$
    Y(\sigma, \tau) = ( u = Y_1(\sigma,\tau), v = Y_2(\sigma,\tau), w = Y_3(\sigma,\tau)).
    $$

We will express our various metrics in terms of such coordinates. To analyze the metric $g_s$ on $\M_0$, we decompose it into the sum of two terms: 
    $$g_s = g_{s_0} + \Delta g,$$
where $g_{s_0}$ is the pullback of the Euclidean metric on the scaled model space $\M_{s_0}$ and $\Delta g = g_s - g_{s_0}$ is the remainder. The $g_{s_0}$ term can easily be  calculated explicitly. In  Proposition~\ref{lemma:metricdecomposition}, we have  shown that $\Delta g_{(s_0, s_1)}\to 0$ as $s_1\to 0$. Our challenge is to find a tractable way to get a bound on how quickly this converges.

\subsection{Decomposition of metric in coordinates}
 
In what follows, we will calculate the first and second fundamental forms of the metric $g_s$ in $(\sigma, \tau)$ coordinates. Let $N_\M = Y_\sigma \times Y_\tau$ be  the normal to the scaled model space and $N = X_\sigma \times X_\tau$ be the normal to the embedded/immersed surface $\M^{emb}_s$.

To simplify notation we set
    $$
    X = X^{emb}_{s_1}, Y = Y_{s_0}, \M = \M_{s_0}  \hbox{ and } s=(s_0,s_1). 
    $$
We omit the $s$ subscript for the first and second fundamental form terms so for example, we set $E = E_s$. 

\begin{lemma}\label{lemma:fundformdiffterms}
    The fundamental form terms for the metric $g_s$ on $\M_0$ are equal to the value of the corresponding term for the scaled model space metric (denoted with subscript $\M$) and a difference term. The scaled model space term only depends on $s_0$. The difference terms depend both on $s_0$ and $s_1$. The first Fundamental Form terms are
    \begin{align*}
        E& = \langle DX.Y_\sigma, DX.Y_\sigma\rangle = E_\M + \Delta E\, \text{with } E_\M = \langle Y_\sigma, Y_\sigma\rangle\\
        F&= \langle DX.Y_\sigma, DX.Y_\tau\rangle = F_\M + \Delta F\, \text{with } F_\M = \langle Y_\sigma, Y_\tau\rangle\\
        G&= \langle DX.Y_\tau, DX.Y_\tau\rangle = G_\M + \Delta G\, \text{with } G_\M = \langle Y_\tau, Y_\tau\rangle\\ 
    \end{align*}
The second Fundamental Form terms are
    \begin{align*}
        e &= \langle X_{\sigma\sigma}, N\rangle = e_\M + \Delta e \text{ with } e_\M = \langle Y_{\sigma\sigma},N_\M\rangle\\ 
        f &= \langle X_{\sigma \tau}, N\rangle = f_\M + \Delta f \text{ with } f_\M = \langle Y_{\sigma \tau},N_\M\rangle\\ 
        g &= \langle X_{\tau\tau}, N\rangle = g_\M + \Delta g \text{ with } g_\M = \langle Y_{\tau\tau},N_\M\rangle\\ 
    \end{align*}
\end{lemma}

We can derive explicit formulas for the $\Delta$ terms based on the decomposition \newline $DX = DX_{s_1}^{emb} = Rot + \Delta DX $ as described in Lemma~\ref{lemma:DXdecomp}.
We illustrate the $\Delta$ formula  in the case of $E$.

\begin{lemma}
    $E = E_\M + \Delta E$ where  
    $$\Delta E = 2 \, \langle Rot. Y_{\sigma},\Delta DX. Y_{\sigma} \rangle + \langle \Delta DX. Y_{\sigma}, \Delta DX.Y_{\sigma} \rangle $$
\end{lemma}

\begin{proof}
    \begin{align*}
    E &= \langle DX. Y_{\sigma},DX. Y_{\sigma}\rangle\\
      &=\langle Rot. Y_{\sigma},Rot. Y_{\sigma}\rangle + 2 \langle Rot. Y_{\sigma},\Delta DX. Y_{\sigma}\rangle + \langle \Delta DX. Y_{\sigma}, \Delta DX. Y_{\sigma}\rangle
    \end{align*}
By the properties of isometry, $\langle Rot. Y_{\sigma},Rot. Y_{\sigma}\rangle= \langle   Y_{\sigma},  Y_{\sigma}\rangle = E_\M.$
\end{proof}   

We will want to compare the length of vectors in the Euclidean metric $g_0$ with their length in the $g_s$ metric. 

Let $\lambda_s: T\M_0 \to \bbR$ be given by
    \begin{equation}\label{eqn:lambdadef}
    \lambda_s(v) = \frac{\norm{v}_s}{\norm{v}_0}.
    \end{equation}
where $\norm{v}_0$ is the Euclidean norm of the vector and $\norm{v}_s$ is its $g_s$ length. Let  $ Q_s$ be the matrix of first fundamental form of the metric $g_s$, 
    $$Q_s =  \begin{pmatrix}
        E_s & F_s \\         
        F_s & G_s  
    \end{pmatrix}. 
    $$
    
\begin{lemma}
    The ratio $\lambda_s(v)$ is bounded by 
    \begin{equation}
    \sqrt{\lambda_{\min}(a,b)}   \leq \lambda_s(v) \leq   \sqrt{\lambda_{\max}(a,b)},
    \end{equation}
where $\lambda_{\min}(a,b) \leq \lambda_{\max}(a,b)$ are the minimum and maximum eigenvalues of the symmetric matrix $Q_s(a,b)$ at the point $(a,b) \in M$.
\end{lemma}
 
\begin{proof}
The matrix $Q_s$ determines the length of a vector in the $g_s$ metric by
    $$
    \norm{v}_s = \sqrt{v^T Q_s v}.
    $$ 
Then
    $$
    \sup \lambda_s(v) = \sup \frac{\norm{v}_s}{\norm{v}_0} = \sup_{\norm{v}_0=1} \sqrt{v^T Q_s v}
    =: \sqrt{\norm{Q_s}_{op}},   
    $$
where $\norm{\cdot}_{op}$ is the operator norm. With a similar lower bound calculation, we get:
    \begin{equation} 
    \frac{1}{\sqrt{\norm{Q_s^{-1}}_{op}}} \leq \lambda_s(v) \leq \sqrt{\norm{Q_s}_{op}}.
    \end{equation}
Thus 
    \begin{equation}\label{eqn:Qoperatorbound}
    \sqrt{\lambda_{\min}(a,b)} = \frac{1}{\sqrt{\norm{Q_s^{-1}}_{op}}} \leq \lambda_s(v) \leq \sqrt{\norm{Q_s}_{op}} = \sqrt{\lambda_{\max}(a,b)}.
    \end{equation}
\end{proof}

In Lemma~\ref{lemma:lengthratio2} we will use rectangular coordinates to get uniform bounds  $\lambda_s^{lb}\leq \lambda_{\min}(a,b)$ and $\lambda_s^{ub} \geq \lambda_{\max}(a,b)$ for these quantities.

\subsection{Rectangular coordinates}\label{sec:rectangularcoordinates}

Rectangular coordinates can be defined on the bottom (top) half of the (scaled) model space. When using rectangular coordinates, we will always restrict ourselves to the region $\Omega$ defined as follows. 

\begin{definition}  
    Define $\Omega \subset \bbR^2$ by 
        $$
        \Omega = (\interior \D_1)^c=\bbR^2 \setminus \interior(\D_1),
        $$
    where $\D_1$ are the disks of radius $1-\Delta r$ as given in Definition~\ref{defn:D1set}.
\end{definition}

On $\M^{\text{bot}}_{s_0}$, the bottom half of the scaled model space,  we define a coordinate system $Y_{s_0}^{bot}: \Omega \to \M^{\text{bot}}_{s_0}$ given by
    \begin{equation}\label{eqn:globalcoordinatesystem}
    Y_{s_0}^{bot}(a,b) = (a, b, \F^{bot}({s_0};a,b))
    \end{equation}
where 
    $$
    \F^{bot}(s_0;a,b) = 
        \begin{cases}
            0 & \text{for } (a,b) \not\in \D_0 \\
            s_0 f(r) & \text{for } (a,b) \in \D_0
        \end{cases}
    $$
with $f$ the smooth function given by Proposition~\ref{prop:smoothprofilefunction} and $r$ is the distance from $(a,b)$ to the center of the disk in which it is located. For the upper component $\M^{\text{top}}_{s_0}$, we use a coordinate map $Y_{s_0}^{top}$ defined with the function
    $$
    \F^{top}(s_0;a,b) \dot= 
        \begin{cases}
            s_0 & \text{for } (a,b) \not\in \D_0 \\
            s_0 (1-f(r)) & \text{for } (a,b) \in \D_0
        \end{cases}
    $$
(see equation~(\ref{def:ftop})).

The following lemma bounds rectangular coordinate partial derivatives in terms of the radial derivatives of the tube profile function. The following corollary gives global bounds of these partial derivatives for points in $\Omega$.
  
\begin{lemma}\label{lemma:bigFbounds}
    For any $(a,b) \in \Omega \cap \D_0$ and $\F = \F^{bot}$ or $\F^{top}$, let $r$ be the corresponding radial distance of the point $(a,b)$ from the center of the disk. 
    Let $\alpha, \beta \in \{a, b\}$.  
    \begin{enumerate}
        \item $|\F_{\alpha}(s_0;a,b)| \leq s_0 |f'(r)|$  
        \item$|\F_{\alpha \beta}(s_0;a,b)| \leq s_0(|f''(r)|+|\frac{f'(r)}{r}|)$.   
    \end{enumerate}
\end{lemma}

\begin{proof}
\begin{enumerate}
    \item The chain rule gives for $r=\sqrt{a^2 + b^2}$ that
        $$
        f_a = \frac{\partial f}{\partial r}\frac{\partial r}{\partial a} = f'(r) \frac{a }{\sqrt{a^2+b^2}}
        $$
    so that $|\F_a(s_0;a,b)| \leq s_0 |f_a(r)| \leq s_0 |f'(r)|$ (for both top and bot). Similarly for $|\F_b|$.

    \item For the second partials, 
        $$
        f_{aa} = f''(r) \frac{ a^2}{a^2 + b^2} + f'(r) \frac{ b^2}{(a^2 + b^2)^{3/2}} 
        $$
    so 
        $$
        |f_{aa}|\leq |f''(r)| + |f'(r) |\frac{1}{r}. 
        $$
    Similarly for $f_{bb}$. For the mixed second partial,
        $$
        f_{ab} = f''(r) \frac{ a b }{a^2 + b^2} - f'(r) \frac{ ab}{(a^2 + b^2)^{3/2}} 
        $$
    and hence again
        $$
        |f_{ab}|\leq |f''(r)| + |f'(r) |\frac{1}{r}. 
        $$
    Thus $|\F_{\alpha \beta}(s_0;a,b)| \leq s_0|f_{\alpha \beta}(r)| \leq s_0(|f''(r)|+|\frac{f'(r)}{r}|)$.
\end{enumerate}
\end{proof}
 
Combining the above with the monotonicity property (4) of Proposition~\ref{prop:smoothprofilefunction}, we get the following.   

\begin{corollary} \label{cor:f'maxf''max}
    Let $r_1 = 1-\Delta r$. For any $(a,b) \in \Omega \cap \D_0$ and $\F = \F^{bot}$ or $\F^{top}$,
    \begin{enumerate}
        \item $|\F(s_0;a,b)|\leq  s_0 f_{\max}(\Omega)$
        \item  $|\F_{\alpha}(s_0;a,b)| \leq    s_0 f_{\max}'(\Omega)$  
        \item $|\F_{\alpha \beta}(s_0;a,b)| \leq     s_0 f_{\max}''(\Omega)$
    \end{enumerate}
\end{corollary}
where 
    $$
    f_{\max}(\Omega) \doteq 1, \quad 
    f_{\max}'(\Omega) =|f'(r_1)| \leq  0.27812, \, \quad
    f_{\max}''(\Omega) =(|f''(r_1)|+|\frac{f'(r_1)}{r_1}|) \leq 2.5577.
    $$
In our Mathematica calculations, we use these values as   upper bounds for $f', f''$ on $\Omega$ .

\subsubsection{Bounds on First Fundamental Form}

By direct  calculation, we determine the values of the First Fundamental Form in rectangular coordinates. The following formulas and bounds hold for both the bottom and top sheet with $f$ as described above.  

\begin{lemma}\label{lemma:DeltaEGF}
    \begin{align*}
    E_\M &= 1+ s_0 f_a(a,b)^2\\
    \Delta E &= 2 \left( \frac{\cos \left(\frac{b}{R_2(s_1)}\right) (s_0f(a,b))+R_2(s_1))}{R_1(s_1)}\right) + \left( \frac{\cos \left(\frac{b}{R_2(s_1)}\right) (s_0f(a,b))+R_2(s_1))}{R_1(s_1)}\right)^2 \\
    F_\M &=  s_0^2  f_a(a,b)\, f_b(a,b)\\
    \Delta F &= 0\\
    G_\M &= 1+ s_0 f_b(a,b)^2\\
    \Delta G &= 2 \left(\frac{s_0 f(a,b)}{R_2(s_1)}\right) + \left(\frac{s_0 f(a,b)}{R_2(s_1)}\right)^2
    \end{align*}
\end{lemma}

In order to bound higher order terms in the following lemmas, we need to set an upper bound on $s_0$ and $s_1$. The bound $0.1$ is strategically chosen to quickly collapse higher order terms while not providing a substantive restriction to the $s_0$ and $s_1$ values we eventually work with.

We note that is easy in the following case to get a serviceable analytic upper bound (of $2s_1 + 2s_1^2$). However, we will use the better numeric upper bounds going forward and these methods are increasingly important for more complicated formulas.

We use the symbol $:=$ to indicate a definition of the notation to the right. 

\begin{lemma}\label{lemma:DeltaEGbound}
    For points in $\Omega$ and values $s_0, s_1 <0.1$
        $$
        \max\{ |\Delta E(s_0,s_1)|, |\Delta G(s_0,s_1)| \}  < 2s_1 + 1.693s_1^2 :=\Delta^{ub}(s_0,s_1).
        $$
\end{lemma}

\begin{proof}
We start with Lemma~\ref{lemma:DeltaEGF} with $R_1(s_1) = \frac{1}{\pi s_1^2}$, $R_2(s_1)=\frac{1}{\pi s_1}$. Then we use Mathematica with the bounds $|\cos(x)| \leq 1$, $f_{max} \dot=1$, $s_0,s_1 < 0.1$ and rational upper bounds on real coefficients to get:
    \begin{align*}
        |\Delta E|
            &= 2 s_1 |\cos(b \pi s_1)|+s_1^2 \cos(b \pi s_1)^2+2 \pi s_0 s_1^2 |\cos(b \pi s_1)| f \\
                & \qquad \qquad \qquad +2 \pi s_0 s_1^3 \cos(b \pi s_1)^2 f+\pi^2 s_0^2 s_1^4 \cos(b \pi s_1)^2 f^2 \\
            &\leq 2 s_1 +s_1^2 +2 \pi s_0 s_1^2 f+2 \pi s_0 s_1^3 f+\pi^2 s_0^2 s_1^4 f^2 \\
            &\leq 2s_1 + s_1^2 + 0.693 s_1^2 \\
            &= 2s_1 + 1.693 s_1^2.
    \end{align*}
Similarly, 
    \begin{align*}
         |\Delta G|
            &\leq   2s_1 s_0 \pi f+ s_1^2 s_0^2 \pi^2 f^2 \\
            &< 6.29 s_0 s_1 + .099 s_1^2.
    \end{align*}
\end{proof}

\begin{lemma}\label{lemma:DetLbQ}
For points in $\Omega$ and values $s_0, s_1 <0.1$
    $$\det  
    \begin{bmatrix}
    E & F\\
    F & G
    \end{bmatrix} 
      \geq 1 -  2s_1  - 0.155 s_0^2s_1 - 1.946s_1^2:=\det Q_s^{lb}
    $$
\end{lemma}

\begin{proof}
Writing $E = E_\M + \Delta E, G = G_\M + \Delta G $ and $F = F_M$ gives
    \begin{align*}
    \det
    \begin{bmatrix}
      E & F\\
      F & G
    \end{bmatrix}
        &= (E_\M G_\M- F_\M^2)+ E_\M \Delta G + \Delta E G_\M  + \Delta E \Delta G.\\
        & = 1 + s_0^2(f_a^2 + f_b^2) + E_\M \Delta G + \Delta E G_\M  + \Delta E \Delta G.\\
        & \geq 1 + \Delta E G_\M  + \Delta E \Delta G, \, \hbox{since } \Delta G \geq 0 \\
        & \geq 1 -  2s_1  - 0.155 s_0^2s_1 - 1.946s_1^2.
    \end{align*}
The last estimate comes from using Mathematica to compute numerical lower bounds for how negative $\Delta E$ can be and  upper bounds for $G_\M$ and $\Delta G \geq 0$, before combining via the formula above and bounding higher order terms.
\end{proof}

We will need bounds on the partial derivatives of the first fundamental form: $\Delta E_a$, $\Delta E_b$, $\Delta G_a$, $\Delta G_b$.

\begin{lemma}\label{lemma:Delta'max}
    For points in $\Omega$ and values $s_0, s_1 < 0.1$, 
    \begin{equation*}
       \max \{ |\Delta E_a |, |\Delta E_b |, |\Delta G_a |, |\Delta G_b| \} < 1.748 s_0 s_1 + 7.342 s_1^2
       =:\Delta_{\alpha}^{ub} (s_0, s_1)
    \end{equation*} 
\end{lemma}

\begin{proof}
Taking the partial derivatives of $\Delta E$ and $\Delta G$ from Lemma~\ref{lemma:DeltaEGF} and following the same computational bounding technique from Lemma~\ref{lemma:DeltaEGbound} in Mathematica yields:
        $$
        |\Delta E_a | \leq 0.193 s_1^2, \qquad \quad |\Delta E_b| \leq 7.342 s_1^2 
        $$
        $$
        |\Delta G_a|, |\Delta G_b| \leq 1.748 s_0 s_1 + 0.055 s_1^2
        $$
Thus the maximum is bounded by $1.748 s_0 s_1 + 7.342 s_1^2$.
\end{proof}

\subsubsection{Bounds on the ratio of lengths of a vector}

Here we determine upper and lower bounds for the quantity
    $$
    \lambda_s(v) = \frac{\norm{v}_s}{\norm{v}_0}, 
    $$
which was introduced in equation~(\ref{eqn:lambdadef}). 

\begin{lemma}\label{lemma:lengthratio2}
Let
    $$\lambda_s^{lb} = \sqrt{ 1 - 2s_1 - 6.284 s_0 s_1^2 }$$
and
    $$\lambda_s^{ub} = \sqrt{1 +  2s_1 + 0.155 s_0^2 + 1.693 s_1^2 }.$$
Then for any $v$ with footprint in $\Omega$ and values $s_0, s_1 < 0.1$,
    \begin{equation}\label{ineq:lambdalbub}
    \lambda_s^{lb} \leq \lambda_s(v) \leq \lambda_s^{ub}.
    \end{equation}
\end{lemma} 

\begin{proof}
Let $w_a=DX.Y_a, w_b=DX.Y_b$ be the basis vectors in the $(a,b)$-coordinate system at the same base point as $v$.

By scaling $v$, we can assume that $\norm{v}_0 = 1$. Expressing  $v = c_a w_a + c_b w_b$ gives that 
    $$
    ||v||_s  = \sqrt{c_a^2 E + 2 c_a c_b F + c_b^2 G}. 
    $$
Using the First Fundamental Form 
    \begin{align*}
      E_{(s_0, s_1)} &=  \left(1 + \frac{\cos \left(\frac{b}{R_2(s_1)}\right) (s_0f(a,b))+R_2(s_1))}{R_1(s_1)}\right)^2+ s_0^2f_a(a,b)^2\\
      &\\
      F_{(s_0, s_1)}&= s_0^2 \, f_a(a,b) \, f_b(a,b)\\
      &\\
      G_{(s_0, s_1)} &=  \left(1+ \frac{s_0f(a,b)}{R_2(s_1)}\right)^2 + s_0^2 f_b(a, b)^2, 
    \end{align*}
we get a lower bound by ignoring all terms that are clearly positive: 
    $$
    s_0^2f_a^2 c_a^2 + s_0^2 \, f_a \, f_b c_a c_b + s_0^2 f_b^2 c_b^2 = (s_0 f_a c_a + s_0 f_b c_b)^2 \geq 0
    $$
and
    $$
    \left(1+ \frac{s_0f(a,b)}{R_2(s_1)}\right)^2 c_b^2 \geq c_b^2,
    $$
so
    $$
    ||v||_s  \geq 
      \sqrt{  c_a^2 + c_b^2+ \frac{2 c_a^2 R_2 \cos[\frac{b}{R_2}]}{R_1} + \frac{2 c_a^2 s_0 \cos[\frac{b}{R_2}]f(a,b)}{R_1}
       }.
    $$
Since $c_a^2+c_b^2 = 1$ and $c_a^2\leq 1$ and using the values of $R_1(s) = \frac{1}{\pi s_1^2}$ and $R_2(s)= \frac{1}{\pi s_1}$ gives the result:
    \begin{align*}
    ||v||_s  
        &\geq \sqrt{ 1 - \frac{2 R_2}{R_1} - \frac{2 s_0 f(a,b)}{R_1}} 
        \geq \sqrt{ 1 - 2s_1 - 2\pi  s_0 s_1^2 f }
        \geq \sqrt{ 1 - 2s_1 - 6.284 s_0 s_1^2 }
    \end{align*}
when $s_0 , s_1 < 0.1$ and $2\pi f < 6.284$.
 
Using Lemma~\ref{lemma:DeltaEGbound} and that $c_a^2+c_b^2 = 1$ and $(c_a+ c_b)^2 \leq 2$,
    \begin{align*}
    ||v||^2_s &=  c_a^2 E + 2 c_a c_b F + c_b^2 G   \\
        &=  c_a^2 (1 + \Delta E)  + c_b^2 (1+ \Delta G) +(s_0 f_a(a,b)c_a + s_0f_b(a,b) c_b)^2 \\
        &\leq 1 + c_a^2 \Delta E  + c_b^2 \Delta G +s_0^2 f_{\max}'^2(c_a + c_b)^2\\
        &\leq 1 + (c_a^2+c_b^2)  \max\{ |\Delta E|, |\Delta G|\} +2 s_0^2 f_{\max}'^2\\
        &\leq 1 +  \Delta^{ub} +2 s_0^2 f_{\max}'^2\\
        &\leq 1 + 2s_1 + 1.693 s_1^2+  .155 s_0^2 
    \end{align*}
since $\Delta^{ub} = 2s_1 + 1.693 s_1^2$ by Lemma~\ref{lemma:DeltaEGbound} and $2 f_{\max}'^2 <0.155$ by Corollary~\ref{cor:f'maxf''max}. 
\end{proof}

\subsection{Angular coordinates}\label{sec:angularcoordinates}

These coordinates will be used for bounding the negative curvature inside the tubes $\T_{1/4}$ (see Section~\ref{sec:curvboundsinA14}) which require bounds on the first and second fundamental forms for the metric $g_s$.

Using  the angular coordinates from Section \ref{sec:tubescaling},  the vertically scaled tubes  are given by: 
    $$Y_{s_0}(\psi, \theta) = (r(\psi) \cos \theta, r(\psi) \sin \theta, s_0 h(\psi))$$
where 
    $$r(\psi) \doteq 1-\frac 12 \cos \psi \text{ and } h(\psi) \doteq \frac 12 +\frac 12 \sin \psi, \quad \text{for}~ \psi \dotin (-\pi/2, \pi/2) \text{ and }\theta \in (0, 2\pi). $$ 

Recall that the $\doteq$ sign means that these formulas are not exact, but (arbitrarily!) close enough that we can use them for numerical calculations. See equation~(\ref{eqn:randhsmoothing}) and Remark~\ref{rmk:doteqforrandh} for a discussion of the smoothing of these functions. Similarly, the $\dotin$ refers to the fact that $\psi$ needs to avoid the arbitrarily small transition zone close to $\psi = \pm \pi/2$ where the functions have been smoothed (see Section~\ref{sec:smoothing}).

Denote by $T_{\rho}$ the part of the tube that sits above the disk $D_{\rho}$ (Definition \ref{def:Drho}), the disk of radius $r(\rho)=1 -\rho \Delta r$ for $0 \leq \rho \leq 1$ where $\Delta r = 1 - \cos (\pi/6)$ (see equation (\ref{eqn:DeltarDefn})). In what follows, we simplify notation and set $Y = Y_{s_0}, \M = \M_{s_0}$. 

\begin{definition}\label{def:psirho}
    Let  $\psi_\rho$ be the angular coordinate such that $r(\psi_\rho) = 1-\rho\Delta r =r(\rho)$.
\end{definition} 
Then we can write
    $$
    T_{\rho} = \{ Y(\psi, \theta) ~|~ \psi\in [-\psi_{\rho}, \psi_{\rho}],~ \theta\in [0,2\pi) \}.
    $$

\begin{lemma}\label{lemma:EFGbdsangcoordmodelsp}
    The first fundamental form for the scaled model space in angular coordinates is given by
        \begin{align*}
       E_\M &= \langle Y_\psi, Y_\psi \rangle = r'(\psi)^2 + s_0^2 \, h'(\psi)^2 \\
       F_\M&=\langle Y_\psi, Y_\theta \rangle =0\\
       G_\M &=\langle Y_\theta, Y_\theta \rangle =  r(\psi)^2
        \end{align*}
    and the second fundamental form by
        \begin{align*}
        e_\M &= \langle Y_{\psi\psi}, N_\M \rangle 
        =-\frac 14 s_0 r(\psi) \\
        f_\M &= \langle Y_{\psi\theta}, N_\M \rangle=0, \\
        g_\M &=\langle Y_{\theta \theta}, N_\M \rangle =s_0 r^2(\psi) h'(\psi) 
        \end{align*} 
    where $N_\M =  Y_\psi \times Y_\theta$ is the normal to the scaled model space surface.
\end{lemma}

\begin{proof}
The formulas for $E_\M$, $F_\M$, and $G_\M$ follow directly from  
    $$
    Y_\psi = (r'(\psi) \cos \theta, r'(\psi) \sin \theta, s_0 \, h'(\psi)), \qquad Y_{\theta}= (-r(\psi)  \sin \theta, r(\psi) \cos \theta , 0).
    $$
     
For the second fundamental form, note that
    $$
    Y_{\psi\psi} = (r''(\psi) \cos \theta, r''(\psi) \sin \theta, s_0 h''(\psi)), \quad Y_{\psi\theta} =(-r'(\psi) \sin \theta, r'(\psi) \cos \theta, 0)
    $$
    $$
    Y_{\theta \theta}= (-r(\psi) \cos \theta, -r(\psi) \sin \theta, 0),
    $$
and
    $$
    N_\M = (-s_0 r(\psi) h'(\psi) \cos \theta, -s_0  r(\psi) h'(\psi) \sin \theta, r(\psi)r'(\psi) ).
    $$
Then 
    \begin{align*}
     e_\M &= \langle Y_{\psi\psi}, N_\M \rangle =  s_0 r(\psi) \left( r'(\psi) h''(\psi) - h'(\psi) r''(\psi) \right )  
    =-\frac 14 s_0 r(\psi)
    \end{align*}
since from equation (\ref{eqn:randhsmoothing}),
    $$
    r'(\psi) = \frac 12 \sin \psi, \qquad r''(\psi) = \frac 12 \cos \psi
    $$
    $$
    h'(\psi)= \frac 12  \cos \psi, \quad \text{and} \quad h''(\psi) = -\frac 12 \sin \psi
    $$
Similar calculations give $f_\M$ and $g_\M$.
\end{proof}

\begin{lemma}\label{lemma:rhbounds}
Maxima for the absolute value of the $r(\psi)$ and $h(\psi)$ functions and their derivatives over the tube $T_{\rho}$ (i.e. for $\psi \in [-\psi_\rho, \psi_\rho]$) are given by
    \begin{center}
    \begin{align*}
        \max_{|\psi|\leq \psi_\rho} r(\psi)  &= r(\psi_\rho) \doteq 1 -\frac 12 \cos \psi_\rho = 1 - \rho \Delta r \\
        \max_{|\psi|\leq \psi_\rho} r'(\psi)  &= r'(\psi_\rho) = \frac{\sin \psi_\rho}{2} =\frac{\sqrt{1- 4 \rho^2 \Delta r^2}}{2}\\
        \max_{|\psi|\leq \psi_\rho} r''(\psi)  & = r''(0) = \frac 12 \\
        \max_{|\psi|\leq \psi_\rho} h(\psi)  &= h(\psi_\rho) \doteq \frac12 +\frac 12 \sin \psi_\rho = \frac12 +\frac {\sqrt{1- 4 \rho^2 \Delta r^2}}{2}\\
        \max_{|\psi|\leq \psi_\rho} h'(\psi)  &= h'(0) = \frac 12\\
        \max_{|\psi|\leq \psi_\rho} h''(\psi)  &= h''(\psi_\rho) = \frac{\sin \psi_\rho}{2}=\frac{\sqrt{1- 4 \rho^2 \Delta r^2}}{2} 
    \end{align*}
    \end{center}
\end{lemma}
 
\begin{proof}
Since $ r(\psi) \doteq 1 -\frac 12 \cos \psi$ and $h(\psi) \doteq \frac 12 +\frac 12 \sin \psi$, the results on the maximum values follow from noting that, for example, $$\max_{|\psi|\leq \psi_\rho} \left(1 -\frac 12 \cos \psi\right) = 1 -\frac 12 \cos \psi_\rho.$$

To express these maxmium values in terms of $\rho \, \Delta r$, note that  
    $$ 
    \frac 12 \cos \psi_\rho = 1 - r(\rho) = \rho \Delta r
    $$ 
so 
    $$
    \cos \psi_\rho = 2 \rho \Delta r
    $$
and
    $$
    \sin \psi_\rho = \sqrt{1- 4 \rho^2 \Delta r^2} 
    $$
\end{proof}

\begin{lemma}  \label{lemma:ModelEFGupperboundsangularcoord}
Upper bounds for the absolute value of the $E_\M, G_\M, e_\M, g_\M$ terms over the tube $T_{\rho}$ are given by 
    \begin{align*}
    \max_{|\psi| \leq \psi_\rho} |E_\M(\psi)|&=  E_\M(\psi_\rho) = \frac 14 ( 1 - 4\rho^2 \Delta r^2 ( 1 - s_0^2)) :=  E_{\M,abs}^{ub} (\rho)
    \\[.1in]
    \max_{|\psi| \leq \psi_\rho} |G_\M(\psi)|&= G_\M(\psi_\rho) \doteq r(\psi_\rho)^2 = (1 - \rho\Delta r) ^2:=  G_{\M,abs}^{ub} (\rho) 
    \\[.1in]
   \max_{|\psi| \leq \psi_\rho} |e_\M(\psi)| & = |e_\M (\psi_\rho)| = \frac 14 s_0 (1 - \rho \Delta r):=  e_{\M,abs}^{ub}(\rho) 
   \\[.1in]
    \max_{|\psi| \leq \psi_\rho}  |g_\M(\psi)|  & \leq 0.15 s_0:=g_{\M,abs}^{ub}(\rho)
    \end{align*}
In addition, 
    \[  
    e_\M g_\M(\psi_\rho) \doteq -\frac {s_0^2}{4}   (1-  \frac{\cos  \psi_\rho}{2})^3 \, \frac {\cos \psi_\rho}{2} = -\frac {s_0^2}{4} (1- \rho \Delta r)^3 \,  \rho \Delta r:= e_\M g_\M(\rho)
    \]
\end{lemma}

\begin{proof}
$E_\M(\psi) = \frac 14 \left( \sin^2 \psi + s_0^2 \cos^2 \psi \right)$ is monotone increasing in $\psi$, thus 
    \begin{align*}
    \max_{|\psi| \leq \psi_\rho} E_\M(\psi)&= E_\M(\psi_\rho)  
    =\frac 14 \left( \sin^2 \psi_\rho + s_0^2 \cos^2 \psi_\rho \right)\\
    &=\frac 14 ( 1 - 4 \rho \Delta r(1 - s_0^2)).
    \end{align*}

For $G_\M(\psi) = r(\psi)^2$ and $e_\M(\psi) = -\frac 14 s_0 r(\psi)$, the results follow by the monotonicity of  $r(\psi)$. 
    
A straightforward calculation gives that the maximum of 
    $$
    g_\M = s_0 (1 -\frac  12 \cos \psi)^2 \frac{\cos \psi}{2}.
    $$
occurs at $\psi= 0.841069$ with value $g_\M(0.841069) = 0.148148 s_0 $. 

We have 
    \begin{align*}
    e_\M g_\M(\psi_\rho) &= -\frac{1}{4} s_0^2\, r^3(\psi_\rho) h'(\psi_\rho) \doteq -\frac {s_0^2}{4}   (1-  \frac{\cos  \psi_\rho}{2})^3 \, \frac {\cos \psi_\rho}{2}\\
    &=-\frac {s_0^2}{4} (1- \rho\Delta r)^3 \, \rho \Delta r. 
    \end{align*}
\end{proof}

Now we consider the first and second fundamental forms for the embedded surface in angular coordinates and their differences with the scaled model space values \newline (i.e., $\Delta E = E - E_\M$). 

\begin{lemma}\label{lemma:DeltaEFGupperboundsangularcoord}
    Assuming $s_0, s_1 < 0.1$, for $\rho = \frac 14$ a computation gives
    \begin{align*}
    \max_{\psi\in [-\psi_\rho, \psi_\rho]} |\Delta E(\psi, \theta)| &\leq   (0.498 + 1.563 s_0)s_1+ .446 s_1^2:= \Delta E_{abs}^{ub}(\rho=\tfrac 14) \\
    \max_{\psi\in [-\psi_\rho, \psi_\rho]} |\Delta F(\psi, \theta)| &\leq (0.483 + 1.514 s_0) s_1 + .432 s_1^2 := \Delta F_{abs}^{ub}(\rho=\tfrac14)  \\
    \max_{\psi\in [-\psi_\rho, \psi_\rho]} |\Delta G(\psi, \theta)| &\leq   (1.869 + 5.863 s_0)s_1 + 1.672 s_1^2:= \Delta G_{abs}^{ub}(\rho=\tfrac14) 
    \end{align*}
and 
    \begin{align*}
    \max_{\psi\in [-\psi_\rho, \psi_\rho]} |\Delta e(\psi, \theta)| &\leq      (0.380  + 0.484 s_0 + 
    3.034  s_0^2) s_1 + 1.165 s_1^2 := \Delta e_{abs}^{ub}(\rho = \tfrac14) \\
     \max_{\psi\in [-\psi_\rho, \psi_\rho]} |\Delta g(\psi, \theta)| &\leq    (1.416 + 0.468 s_0 + 1.466 s_0^2) s_1 + 3.039 s_1^2 := \Delta g_{abs}^{ub}(\rho = \tfrac14)
    \end{align*}
For  $\Delta eg = e_\M \Delta g + g_\M \Delta e + \Delta e \Delta g$,  
    $$
    \max_{\psi\in [-\psi_\rho, \psi_\rho]} |\Delta eg(\psi, \theta)| \leq  
    (0.399 s_0 + 0.185  s_0^2 + 0.804 s_0^3) s_1 + 1.114 s_1^2
    :=\Delta eg_{abs}^{ub}(\rho = \tfrac14).
    $$
\end{lemma}

\begin{proof}
We use Mathematica to carry out the calculations and to take bounds. The case of $\Delta E$ illustrates our method. 

Using Mathematica to compute $\Delta E = E - E_\M$ yields many terms; we start managing these with a Simpify command. To compute an upper bound $\Delta E_{abs}^{ub}$, we make all the coefficients positive, bound all sine and cosine terms by 1, and take the maximum of terms involving $r(\psi), h(\psi)$ and their derivatives for $\psi\in [-\psi_\rho, \psi_\rho]$ over the region $\T_{1/4}$ using Lemma~\ref{lemma:rhbounds}.

For the second fundamental form terms, the calculation is  longer but follows the same approach. Here we note that sometimes Mathematica will be unable to make certain desirable simplifications under certain orders of substitution and Simplify and Expand commands. For instance, in order to cancel out a constant term $(s_1)^0$ in the expression for $\Delta e$, it helps to first expand out the expression for $e = \langle X_{\psi \psi}, N \rangle$.

To bound $\Delta eg$,  we use the relation
    $$
    \max_{\psi\in [-\psi_\rho, \psi_\rho]} |\Delta eg(\psi, \theta)| \leq   e_{\M,abs}^{ub} (\rho)\Delta g_{abs}^{ub}(\rho) + g_{\M,abs}^{ub}(\rho) \Delta e_{abs}^{ub}(\rho) + \Delta e_{abs}^{ub}(\rho) \, \Delta g_{abs}^{ub}(\rho). 
    $$
\end{proof}

In Section~\ref{sec:MinTimeVariousRegions}, we will get an improved bound for negative curvature by using Theorem~\ref{thm:knegformula} but with bounds evaluated at $\rho = \frac 12$ rather than at $\rho=\frac14$

\begin{lemma} \label{lemma:DeltaEFGupperboundsangularcoordrho1}
    Assuming $s_0, s_1 < 0.1$, for $\rho = \frac 12$ a computation gives
        \begin{align*}
        \max_{\psi\in [-\psi_\rho, \psi_\rho]} |\Delta E(\psi, \theta)| &\leq   (0.492 + 1.536 s_0)s_1+ .439 s_1^2:= \Delta E_{abs}^{ub}(\rho=\tfrac 12) \\
        \max_{\psi\in [-\psi_\rho, \psi_\rho]} |\Delta F(\psi, \theta)| &\leq (0.463 + 1.446 s_0) s_1 + .414 s_1^2 := \Delta F_{abs}^{ub}(\rho=\tfrac12)  \\
        \max_{\psi\in [-\psi_\rho, \psi_\rho]} |\Delta G(\psi, \theta)| &\leq   (1.742 + 5.446 s_0)s_1 + 1.556 s_1^2:= \Delta G_{abs}^{ub}(\rho=\tfrac12)
        \end{align*}
    and 
        \begin{align*}
        \max_{\psi\in [-\psi_\rho, \psi_\rho]} |\Delta e(\psi, \theta)| &\leq      (0.367  + 0.467 s_0 + 
        2.922  s_0^2) s_1 + 1.123 s_1^2 := \Delta e_{abs}^{ub}(\rho = \tfrac12) \\
        \max_{\psi\in [-\psi_\rho, \psi_\rho]} |\Delta g(\psi, \theta)| &\leq    (1.265 + 0.436 s_0 + 1.362 s_0^2) s_1 + 2.714 s_1^2 := \Delta g_{abs}^{ub}(\rho = \tfrac12)\\
        \max_{\psi\in [-\psi_\rho, \psi_\rho]} |\Delta eg(\psi, \theta)| &\leq  
        (0.350 s_0 + 0.171  s_0^2 + 0.751 s_0^3) s_1 + 0.964 s_1^2
        :=\Delta eg_{abs}^{ub}(\rho = \tfrac12).
        \end{align*}
\end{lemma}

\subsection{Polar coordinates} \label{sec:metricinpolarcoordinates}

Using the angular coordinates from the previous section, we will obtain a bound on the negative curvature on the inner part of tubes (see Section~\ref{sec:curvboundsinA14}).  Using rectangular coordinates, we will compute the curvature outside of the tubes (see Section~\ref{sec:curvboundsoutsideT0}) and determine an upper bound on the positive curvature there. To bound the curvature over the outer part of the tube $\A_{1/4} = \T_0 \setminus \T_{1/4}$, we will use polar coordinates  
(see Section~\ref{sec:curvboundsinsideT0minusT14}). We note that $\A_{1/4}$ consists of two connected components (from the top and bottom of the tube), which are isometric in the model space but behave slightly differently under the embedding map. Polar coordinates $Y = Y_{s_0} $ for the bottom half of the  scaled tubes are given by 
    $$
    Y(r, \theta) = (u= r \cos \theta, v=r \sin \theta, w=s_0\,  f(r)) 
    $$
with $f$ defined by equation (\ref{eqn:unsmoothedtubefunction}) and Proposition \ref{prop:smoothprofilefunction}. For the top half of the tubes one uses $f^{top}$ given by equation~(\ref{def:ftop}).

We determine the Fundamental Forms with these coordinates and then obtain bounds. For the tubes in the scaled model space, the metric is symmetric between the top and bottom. 
  
\begin{lemma} 
    The Fundamental Forms for the scaled model space in polar coordinates are
    \begin{align*}
    E_\M =  1 + s_0^2 f'(r)^2, \qquad F_\M &= 0 , \qquad G_\M = r^2 \\
    e_\M =  r s_0 f''(r), \qquad f_\M &= 0 , \qquad g_\M = r^2 s_0 f'(r).
    \end{align*}
\end{lemma}

\begin{proof}
These follow by direct calculation using that  
    $$
    Y_r = ( \cos    \theta,  \sin \theta, s_0 f'(r)) \quad \text{and} \quad
    Y_{\theta} = ( -r \sin    \theta, r \cos \theta, 0).
    $$
\end{proof} 

In proving Proposition~\ref{prop:Kpolarub}, which gives us an upper bound for the potentially positive curvature in $\A_{1/4}$,  we will want lower bounds for $E_\M, G_\M$ and upper bounds for $e_\M, g_\M$ over $\A_{1/4}$. Recall that $r(\tfrac14)$ denotes the radius $r(\rho=\frac 14) = 1 - \frac 14 \Delta r$ (see equation (\ref{eqn:DeltarDefn})).
  
\begin{lemma} \label{lemma:modelEFGupperboundspolarcoord}
    For $Y(r, \theta) \in \A_{1/4}$, 
        \begin{align*}
        r &\leq 1 := r^{max}\\
        f(r) &\leq 1+\epsilon: = f^{max} ~~\text{for some arbitrarily small }\epsilon\\
        |f'(r)|& \leq |f'(r(\tfrac14))|: = f'_{max}(\A_{1/4})\\
        |f''(r)|& \leq |f''(r(\tfrac14))| := f''_{max}(\A_{1/4})\\
        &\\
        E_\M &=  1 + s_0^2 f'(r)^2 \geq 1 := E_\M^{lb}\\
        G_\M &= r^2 \geq r(\tfrac14)^2 :=G_\M^{lb}\\
        &\\
        |e_\M| &=  r s_0 |f''(r)|\leq s_0 |f''(r(\tfrac14))| := e_{\M, abs} ^{ub}\\ 
        |g_\M| &= r^2 s_0 |f'(r)|\leq s_0 |f'(r(\tfrac14))| :=  g_{\M, abs}^{ub}.\\ 
        \end{align*}
\end{lemma}

To determine the $\Delta$ terms and their bounds, we proceed in the same manner as  for the angular coordinates in Section~\ref{sec:angularcoordinates}. We calculate the embedding terms and take the difference with the scaled model space terms. We then take absolute values, bound any sine and cosine terms by 1, and  bound $f(r), f'(r), f''(r)$ by their maximum values over the region. This leads to the values
    $$
    \Delta E_{abs}^{ub}, \Delta G_{abs}^{ub},\Delta F_{abs}^{ub}, \Delta e_{abs}^{ub},\Delta g_{abs}^{ub}.
    $$

Since
    $$
    \Delta eg = e_\M \Delta g + g_\M \Delta e + \Delta e \Delta g, 
    $$
we get 
    $$
    \Delta eg \leq  e_{\M, abs}^{ub} \Delta g_{abs}^{ub} + g_{\M, abs}^{ub} \Delta e_{abs}^{ub} + \Delta e_{abs}^{ub} \Delta g_{abs}^{ub}:= \Delta eg_{abs}^{ub}
    $$

\begin{lemma}\label{lemma:deltaEFGupperboundspolarcoord}
    Assuming that $s_0, s_1 < 0.1$,   bounds   for each of the $\Delta$ terms in the first fundamental form over the set  $\A_{1/4}$  are given by 
        \begin{align*}
         \max_{r\in [r(\tfrac14), 1]} |\Delta E (r, \theta)| & \leq   (2 +  6.284 s_0) s_1 + 1.791 s_1^2:= \Delta E_{abs}^{ub} \\\
          \max_{r\in [r(\tfrac14), 1]} |\Delta F (r, \theta)| & \leq  (1 +3.142 s_0) s_1 + .896 s_1^2 := \Delta F_{abs}^{ub} \\
           \max_{r\in [r(\tfrac14), 1]} |\Delta G (r, \theta)| &\leq (2 +  6.284 s_0) s_1 +  1.791 s_1^2:=   \Delta G_{abs}^{ub} 
        \end{align*}
  
    Bounds for  the $\Delta$ terms in the second fundamental form are given by 
        \begin{align*}
          \max_{r\in [r(\tfrac14), 1]}|\Delta e (r, \theta)|&\leq (3.142 + 2.014 s_0 + 6.355 s_0^2) s_1 + 6.297 s_1^2:= \Delta e_{abs}^{ub}   \\
          \max_{r\in [r(\tfrac14), 1]} |\Delta g (r, \theta)|&\leq (3.142 + 0.068 s_0 + 0.211 s_0^2) s_1 +6.195 s_1^2 := \Delta g_{abs}^{ub}  \\
          \max_{r\in [r(\tfrac14), 1]} |\Delta eg (r, \theta)|&\leq (6.538 s_0 + 0.271 s_0^2 + 0.852 s_0^3) s_1 + 16.507 s_1^2 := \Delta eg_{abs}^{ub}.
        \end{align*}
\end{lemma}

\section{Strong finite horizon property and time bounds}
\label{sec:contdep}

For scaling parameter $s_0$ and  embedding parameter $s_1$  sufficiently close to zero, the pullback metric $g_s, \, s=(s_0, s_1)$,  can be viewed as a uniformly small perturbation of the flat metric $g_0$ outside of the tubes $\T_1$. We  will  use information about the strong finite-horizon property for flat plane with disks (Proposition~\ref{prop:DrhoStrongFiniteHorizon}) to determine  a   strong finite-horizon condition for the metric $g_s$. 

To determine the size of the perturbation,  we examine geodesics in metrics $g_0$ and $g_s$ that start at the same point and in the same direction and see how  close they stay for a certain finite amount of time. We call this process geodesic control. All comparisons are measured in the common rectangular coordinate system. We quantify all aspects of this statement below. Note that a general technique for this situation would be to use a Gr\"onwall inequality, which would lead to an exponential estimate. Because we are perturbing from the flat metric, which is integrable and parabolic, we get a polynomial (in fact quadratic) bound; see Theorem~\ref{thm:deltaabounds}. 

We carry out the analysis using rectangular coordinates (see Section \ref{sec:rectangularcoordinates}) on the set $\Omega \subset \bbR^2$ where $\Omega = (\mbox{int}(\D_1))^{c} = \overline{\D_1^{c}};$ i.e., $\Omega$ is the  closed set exterior to the disks $\D_1$. We have two coordinate chart maps: one mapping $\Omega$ into the bottom half of $\M_{s_0}$ and one into the top half (see equation (\ref{eqn:globalcoordinatesystem})). 

We will examine the distance between the $g_s-$ geodesic $\gamma_s(p, v_s, t)$  and its partner $g_0-$ geodesic $ \gamma_0(p, v, t)$. These geodesics both start at point $p\in \Omega$ and point in the same direction relative to the rectangular coordinate system: the vector $v$ has unit length in the $g_0$ metric, and $v_s$ is the vector $v$ scaled to be unit length in the $g_s$ metric. The distance between these geodesics, $dist(\gamma_s(p, v_s, t), \gamma_0(p, v, t))$, is measured in the Euclidean metric on the rectangular coordinate system. The following theorem says that if we can control, for a certain amount of time $T_{gc}$, the  distance between the geodesics, then we can transfer a strong finite-horizon condition (set-version and time version) for the $g_0$ metric to a strong finite-horizon condition for the $g_s$ metric.  

For fixed value of $s$ and for $p \in \Omega$ and $v \in S_p \Omega$, let 
    \begin{equation}\label{eqn:Ttube}
    T_{tube}(p,v)>0
    \end{equation}
be the first time either $\gamma_{0}(p, v, t)$ or $\gamma_{s}(p, v_s, t)$ leaves $\Omega$ (so either $\gamma_0$ enters $\D_1$ or, on $\M_{s_0}$,  the $\gamma_s$    geodesic enters  $\T_1$). We note that from the finite horizon condition for $\D_1$ on the flat plane (Corollary~\ref{thm:D1StrongFiniteHorizon}) that $T_{tube}(p,v)\leq 3$.
    
The functions of $s_0, s_1$ below are defined in Sections \ref{sec:estimatingC1} and \ref{sec:boundC3}.
    \begin{align}
    C_1(s_0, s_1)&= \frac{\sqrt{1 +  2s_1 + 0.155s_0^2 + 1.693s_1^2 } - 1 }{\sqrt{ 1 - 2s_1 - 6.284 s_0 s_1^2 }} \notag \\ 
    C_2(s_0, s_1)&= \frac{1}{\sqrt{ 1 - 2s_1 - 6.284 s_0 s_1^2 }} \label{eqn:Ci}\\
    C_3(s_0, s_1) &= 3\sqrt{2}\; \left(\frac{0.712 s_0^2 + .874 s_0 s_1 + 3.671 s_1^2  }{ 1 - 2s_1 - 6.284 s_0 s_1^2} \right).
    \notag
    \end{align} 

\begin{theorem}\label{thm:Cibounds->geodcontrol->finitehorizon} 
    Consider a metric $g_s$ and set $T_{gc}=3$. Then for the following statements, I $\Rightarrow$ II $\Rightarrow$ III.
        \begin{enumerate}
        \item[I.] For $s=(s_0,s_1)$ with $s_0, s_1 < 0.1$,
            \begin{equation}\label{eqn:contdep}
            C_1(s) \, T_{gc}  +  C_2^2(s)\, C_3(s) \, \frac{T_{gc}^2}{2}\,   \leq \frac{\Delta r}{4\sqrt{2}},
            \end{equation}
        \item[II.] for all $p \in \Omega$, $v \in S_p \Omega$, and $0 \leq t \leq \min\{ T_{tube}(p,v), T_{gc}\}$,
            \begin{equation}\label{eqn:Deltarover4}
            dist(\gamma_{s}(p, v_s, t), \gamma_{0}(p, v, t)) \leq \frac{\Delta r}{4}  
            \end{equation} 
        \item[III.] on $\M_{s_0}$ the pair of sets $\T_{3/4} \subset \T_{1/4}$ is $T_{ret}$ strongly finite horizon in the $g_{s}$ metric with 
            $$
            \textstyle T_{ret} = 2.30571.
            $$
        \end{enumerate}
\end{theorem}

Theorem~\ref{thm:Cibounds->geodcontrol->finitehorizon} is proved in Section~\ref{sec:ThmgeodcontimpliesfhpProof}. Tools to prove I $\Rightarrow$ II are developed in the next few sections. Note that except for some technical details dealt with in that section,  II $\Rightarrow$ III follows by noting that in the flat case, $\D_1 \subset \D_{1/2}$ has the $T_{ret}(\frac 12) = 2.30571$ strongly finite horizon property (see Proposition~\ref{prop:DrhoStrongFiniteHorizon}) and $\gamma_s$ is within $\frac{\Delta r}{4}$ of $\gamma_0$.  

In Section~\ref{sec:TimeInDiskD14}, we derive a time-version of the finite-horizon property of Statement III: i.e., a $\Delta t$ for which $\T_{1/4}$ has the $(T_{ret}=2.30571,\Delta t)$ finite horizon property. We think of $T_{gc}$ as the length of time for which geodesic control holds, i.e. the $g_s, g_0$ geodesics stay close (condition II). In Section ~\ref{sect:improvedTgc}, we show that the results of this theorem hold  under the weaker condition of  $T_{gc}= 2.5$.

\subsection{Geodesics in rectangular coordinates}\label{sec:geodesicsincoords}

In what follows we use rectangular coordinates on the scaled model space $\M_{s_0}$: 
    $$
    Y^{bot}(a, b) = (a,b,\F^{bot}(s_0;a,b)) = (a, b, s_0 f(a,b))
    $$
and
    $$
    Y^{top}(a, b) = (a,b,\F^{top}(s_0;a,b)) \dot= (a, b, s_0 (1- f(a,b)))
    $$
defined on the set $\Omega$ (see equation~(\ref{eqn:globalcoordinatesystem})). 

The geodesic equations for a metric $g$ in a coordinate system $(a,b)$ are expressed in terms of Christoffel symbols $\Gamma^i_{jk}(a,b)$. A path $\gamma(t) = (a(t),b(t))$ is a geodesic in the $g$ metric if the coordinate functions $a$ and $b$ satisfy
    \[
    a''+ \Gamma^1_{11} a'^2 +  \Gamma^1_{12} a' b' +  \Gamma^1_{22} b'^2=0 ,
    \]
    \[
    b''+ \Gamma^2_{11} a'^2 +  \Gamma^2_{12} a' b' +  \Gamma^2_{22} b'^2=0 ,
    \]
where we have suppressed the dependence of $\Gamma^i_{jk}$ on $a$ and $b$.

For the flat metric $g_0$, the first fundamential form in the $(a,b)$ coordinates is given by
    $$
    E_0(a,b) =1, \, F_0(a,b) = 0, G_0(a,b)=1.
    $$
The Christoffel symbols are defined in terms of derivatives of the first fundamental form and hence are identically zero for $g_0$:  $\Gamma^i_{jk}=0$ for all $i,j,k$. If we denote by $(a_0(t), b_0(t))$ the solutions of the geodesic equations for $g_0$, then 
    $$
    a_0'' =0 , \quad b_0''=0.
    $$
and  
    $$
    a_0(t) = a_0(0) + a_0'(0) t, \quad 
    b_0(t) = b_0(0) + b_0'(0) t.
    $$

Let $g_s$ be a continuous $s=(s_0,s_1)$ parameterized family of $C^\infty$ metrics defined on $\Omega \subseteq \bbR^2$ including the Euclidean metric $g_0$.

Let $\gamma_s(t) = (a_{s} (t), b_{s} (t))$ denote the solutions of the geodesic equations for the metric $g_{s}$. Then 
    \[
    a_{s}''=-( \Gamma^1_{11} a_{s}'^2 +  \Gamma^1_{12} a_{s}' b_{s}' +  \Gamma^1_{22} b_{s}'^2)  ,
    \]
    \[
    b_{s}''=-( \Gamma^2_{11} a_{s}'^2 +  \Gamma^2_{12} a_{s}' b_{s}' +  \Gamma^2_{22} b_{s}'^2) 
    \]
where the Christoffel symbols $\Gamma^i_{jk}= \Gamma^i_{jk}(a,b;s)$ are functions of parameter $s$ as well as the point $(a,b)$ on the surface.  

We want to examine the evolution of geodesics under the $g_s$ and $g_0$   metrics that start with the ``same" initial conditions. However, 
a unit vector in one metric may no longer be a unit vector in another metric. Let $\lambda_s: T\Omega \to \bbR$ be a scaling function defined by 
    \begin{equation}
    \lambda_s(v) = \frac{\norm{v}_s}{\norm{v}_0}.
    \end{equation}

Then we can ``match'' a $\gamma_s$ geodesic with a $\gamma_0$ geodesic by choosing initial conditions such that
    $$
    \gamma_s(0) = \gamma_0(0)
    \quad 
    \text{and}
    \quad
    \gamma_s'(0) =  \frac1{\lambda_s(\gamma_0'(0))} \,   \gamma_0'(0)= \frac{\gamma_0'(0)}{\norm{\gamma_0'}_s}. 
    $$
where $\norm{\gamma'(0)}_0 = 1$. Then the vector $\gamma_{s}'(0)$ has unit length in the $g_s$ metric and, in coordinates,
    $$
    (a_{s}(0) = a_0(0), b_{s}(0) = b_0(0))
    \quad 
    \text{and}
    \quad
    ( a_{s}'(0) = \frac1{\lambda_s} \,   a_0'(0), b_{s}'(0) = \frac1{\lambda_s} \,  b_0'(0)).
    $$
Thus the geodesics in the two different metrics start at the same base point and point in the same direction with respect to $(a,b)$ coordinates.

We wish to control how far apart this trajectories can move as a function of $t$. We define
    $$ \Uplambda a_s(t) = a_s(t) - a_{0}(t), \quad 
    \Uplambda b_s(t) = b_s(t) - b_{0}(t).
    $$
Then
    \begin{equation} \label{eqn:deltaa''}
    \begin{aligned}
    \Uplambda a_s'' &= a_s'' - a_0''= ( \Gamma^1_{11} a_s'^2 +  \Gamma^1_{12} a_s' b_s' +  \Gamma^1_{22} b_s'^2) \\
    \Uplambda b_s'' &= b_s'' - b_0''= ( \Gamma^2_{11} a_s'^2 +  \Gamma^2_{12} a_s' b_s' +  \Gamma^2_{22} b_s'^2)
    \end{aligned}
    \end{equation}
with initial conditions
    $$
   \Uplambda a_s(0) = \Uplambda b_s(0) = 0,
    $$
    $$
    \Uplambda a_s'(0)= (\frac1{\lambda_s}-1) a_0' , \quad \text{and} \quad
    \Uplambda b_s'(0)= (\frac1{\lambda_s}-1) b_0' .
    $$
    
For $g_s$ not too far away from $g_0$ on $\Omega$ in $C^1$, the objects below are all bounded and depend only  on $s$. 
For $s=(s_0,s_1)$, let $C_1(s)$, $C_2(s)$, $C_3(s)$ be any functions that satisfy the  following:
    \begin{equation} \label{eqn:threebounds}
    \begin{aligned}
    \max\{ |\Uplambda a_s'(0)|, |\Uplambda b_s'(0)|\} &\leq C_1(s)  \\
    \max\{|a_{s}'(t)|, |b_{s}'(t)|\}  &\leq C_2(s) \\
    \max_{k=1,2}  |\Gamma^k_{11}  +  \Gamma^k_{12} +  \Gamma^k_{22}| &\leq C_3(s) . 
    \end{aligned}
    \end{equation}
These maximums are taken over all $\gamma_s$ defined above and  values of $t$ such that $\gamma_s(\tau) \in \Omega$ for all $0 \leq \tau \leq t$. We will eventually make a specific (but non-optimal) choice of the bounding functions $C_i(s)$ (see Sections \ref{sec:estimatingC1} and \ref{sec:boundC3}.).
  
Note that for points in $\Omega$,  as $s\to 0$ the metric $g_{s} \to g_0$ so that by continuity 
    \begin{equation*}
    \begin{aligned}
    \max\{ |\Uplambda a_s'(0)|, |\Uplambda b_s'(0)|\} &\to 0  \\
    \max\{|a_{s}'(t)|, |b_{s}'(t)|\}  &\to 1 \\
    \max_{k=1,2}  |\Gamma^k_{11}  +  \Gamma^k_{12} +  \Gamma^k_{22}| &\to 0 . 
    \end{aligned}
    \end{equation*}
Of course finding better (smaller) bounds allows us to have better control on geodesics, so we will be looking for bounds such that $C_1(s) \to 0$, $C_2(s) \to 1$, and $C_3(s) \to 0$.

\subsection{Bounding separation of geodesics under one parameter family of metrics }

Using the Gr\"onwall inequality we could  give bounds for $\Uplambda a_s(t), \Uplambda b_s(t)$ which are  exponential in $t$. By perturbing off the flat metric we can achieve a quadratic bound.

\begin{theorem}\label{thm:deltaabounds} 
    Let $\gamma_0(t), \gamma_s(t)$ be geodesic with matching initial conditions as described above. Let $T$ be a time such that both $\gamma_0(t), \gamma_s(t) \in \Omega$ for $0 \leq t \leq T$. Then for  any $C_1(s)$, $C_2(s)$, and $C_3(s)$ that satisfy the inequalities~(\ref{eqn:threebounds})
        \begin{align}\label{eqn:contdepformula}
        |\Uplambda a_s(t)|    &  \leq  C_1\, T +  C_2^2\, C_3\ \frac {T^2}2 ~~~ \text{and}\\
        |\Uplambda b_s(t)|       & \leq  C_1\, T +  C_2^2\, C_3\ \frac {T^2}2 .
        \end{align}
\end{theorem}

In the following, we suppress the dependency of $\Uplambda a$ and $\Uplambda b$ on $s$.
\begin{proof}
Integrating~(\ref{eqn:deltaa''}) gives  
    \begin{align*}
    \Uplambda a'(t)= \Uplambda a'(0)  + \int_0^t \Uplambda a''(t)\, dt
    \end{align*}
so that 
    \begin{align*}
    |\Uplambda a'| &\leq |\Uplambda a'(0) | + \int_0^t |\Uplambda a''(t)|\, dt\\
      &\leq |\Uplambda a'(0) | + \int_0^t | \Gamma^1_{11} a_s'^2 +  \Gamma^1_{12} a_s' b_s' +  \Gamma^1_{22} b_s'^2| dt\\
      &\leq C_1 + C_2^2\, C_3 \,t,
    \end{align*}
and similarly for $|\Uplambda b'(t) |$. 
Then integrating once more gives
    \begin{align*}
    \Uplambda a(t)= \Uplambda a(0)  + \int_0^t \Uplambda a'(t)\, dt
    \end{align*}
so
    \begin{align*}
    |\Uplambda a(t)| &\leq |\Uplambda a(0) | + \int_0^t |\Uplambda a'(t)|\, dt\\
      &\leq |\Uplambda a(0) | + \int_0^t (C_1 + C_2^2\, C_3\,t \,) dt\\
      &\leq |\Uplambda a(0) | + C_1\, t + C_2^2 \, C_3 \frac {t^2}2,
    \end{align*}
and similarly for $|\Uplambda b(t) |$. Since we have that $|t| \leq T$ and $\Uplambda a(0) = \Uplambda b(0) =0$, it follows that
    \begin{align*}
    |\Uplambda a_s(t)|       &\leq  C_1\, T + C_2^2 \, C_3\ \frac {T^2}2 \,\\
     |\Uplambda b_s(t)|       &\leq  C_1\, T + C_2^2 \, C_3 \ \frac {T^2}2.
    \end{align*}
\end{proof}

\subsubsection{Bounds $C_1$ and $C_2$} \label{sec:estimatingC1}

Quantifying the above geodesic control argument comes down to finding  upper bounds $C_1(s), C_2(s), C_3(s)$. 

Let $\lambda_s^{lb}$ and $\lambda_s^{ub}$ be any values that provide lower and upper bounds for the scaling function $\lambda_s( v)$ (equation~(\ref{eqn:lambdadef}))  over all vectors $v$ with basepoint $p\in \Omega$  
    $$ 
    \lambda_s^{lb} \leq \lambda_s(v) \leq \lambda_s^{ub} 
    $$
Using  Lemma  \ref{lemma:lengthratio2}, we will eventually make a specific (but non-optimal) choice of the values $\lambda_s^{lb}, \lambda_s^{ub}$.

\begin{lemma} 
    For any  lower and upper bound $\lambda_s^{ub}, \lambda_s^{ub}$, we have that  
        $$\max\{ |\Uplambda a_s'(0)|, |\Uplambda b_s'(0)|\} \leq \frac {1}{{\lambda_s^{lb}}} \max {\{\lambda_s^{ub}- 1, \,1 - \lambda_s^{lb} \}}
        $$
    and
        $$ \max \{|a_s'(t)|, |b_s'(t)|\} \leq \frac{1   }{\lambda_s^{lb}}$$ 
    whenever $\gamma_s(t) \in \Omega$.
\end{lemma}

\begin{proof}
We have 
    $$
    \Uplambda a_s'(0) = (\frac{1}{\lambda_s}-1) a_0' = \frac{1 - \lambda_s}{\lambda_s} a_0'.
    $$ 
Hence using that $|a_0'|\leq 1$, 
     $$
    |\Uplambda a_s'(0)| 
        = |\frac{1 - \lambda_s}{\lambda_s}| |a_0'| 
        \leq \frac {1}{{\lambda_s^{lb}}} \max {\{\lambda_s^{ub}- 1, \,1 - \lambda_s^{lb} \}}
    $$
The same estimate holds for  $|\Uplambda b_s'(0)|$. 
    
Now let $v_{s}(t) = (a_{s}'\, (t), b_{s}'\,(t))$ be a unit tangent vector to the $\gamma_s$ geodesic so 
$$v_{s}(t) = \frac{v_0}{\lambda_s(v_0)}  $$
for  a vector $v_0$ with unit Euclidean length. Hence its length in the Euclidean metric is 
    $$
    ||v_s(t)||_0 = \frac 1{\lambda_s(v_0)} \leq \frac{1   }{\lambda_s^{lb}}. 
    $$ 
\end{proof}

Using the results of Lemma \ref{lemma:lengthratio2} and assuming that  $s_0, s_1 < 0.1$, we have explicit formulas for a particular pair of lower and upper bounds for $\lambda_s$: 
    \begin{equation}\label{eqn:lambdalbs}
    \lambda_s^{lb} = \sqrt{ 1 - 2s_1 - 6.284 s_0 s_1^2 }
    \end{equation} 
and
    \begin{equation}\label{eqn:lambdaubs}
    \lambda_s^{ub} = \sqrt{1 +  2s_1 + 0.155s_0^2 + 1.693s_1^2 }.
    \end{equation}

\begin{corollary} \label{cor:C1C2formulas} 
    Define 
        $$
        C_1(s_0, s_1):= \frac{\sqrt{1 +  2s_1 + 0.155s_0^2 + 1.693s_1^2 } - 1 }{\sqrt{ 1 - 2s_1 - 6.284 s_0 s_1^2 }},\quad C_2(s_0, s_1):= \frac{1}{\sqrt{ 1 - 2s_1 - 6.284 s_0 s_1^2 }}.
        $$
    For $s_0, s_1 < .1$, 
        $$\max\{ |\Uplambda a_s'(0)|, |\Uplambda b_s'(0)|\} \leq C_1(s_0, s_1)
        $$
    and
        $$ \max \{|a_s'(t)|, |b_s'(t)|\} \leq C_2(s_0, s_1)$$ 
    whenever $\gamma_s(t) \in \Omega$.
\end{corollary}

We observe that the functions $C_i$ are both monotone increasing in $s_0$ and $s_1$.

\begin{proof}
For $s_0 < 0.1$, 
    \begin{align*}
    \frac { \max {\{\lambda_s^{ub}- 1, \,1 - \lambda_s^{lb} \}}}{\lambda_s^{lb}}&= 
    \frac{ \lambda_s^{ub} -1}{\lambda_s^{lb}}
    = \frac{\sqrt{1 +  2s_1 + 0.155s_0^2 + 1.693s_1^2 } - 1 }{\sqrt{ 1 - 2s_1 - 6.284 s_0 s_1^2 }} 
    \end{align*}    
\end{proof}

\subsubsection{Bound $C_3$} \label{sec:boundC3}

The Christoffel symbols for the $g_s$ metric in $(a,b)$-coordinates can be determined by solving equations of the form
    \begin{equation}\label{Gammaeqation}
    Q_{s} \Gamma_{ij} = {\vb}_{ij}
    \end{equation}
where 
    $$
    Q_s =  \begin{pmatrix}
        E_s & F_s \\         
        F_s & G_s  
    \end{pmatrix}.
    $$
is the first fundamental form of the $g_s$ metric. Note that the terms $\Gamma_{ij}, {\vb}_{ij}$ also depend on $s$ and the coordinates $(a,b)$.
 
We have three equations for differing values of $i$ and $j$:
    \begin{align*}
    \Gamma_{11} =  \begin{pmatrix}
    \Gamma^1_{11} \\ 
      \Gamma^2_{11}   \\ 
    \end{pmatrix} 
        \, &\qquad
    {\vb}_{11} = 
    \begin{pmatrix}
    \frac 12 E_a \\ 
    F_a - \frac12 E_b  
    \end{pmatrix} 
     \\
    \Gamma_{12} =  \begin{pmatrix}
    \Gamma^1_{12} \\ 
      \Gamma^2_{12}   \\ 
    \end{pmatrix} 
        \,& \qquad
    {\vb}_{12} = 
    \begin{pmatrix}
    \frac 12 E_b \\ 
     \frac12 G_a   
    \end{pmatrix}
    \\
    \Gamma_{22} =  \begin{pmatrix}
    \Gamma^1_{22} \\ 
      \Gamma^2_{22}   \\ 
    \end{pmatrix} 
        \,& \qquad
    \vb_{22} = 
    \begin{pmatrix}
    F_b - \frac 12 G_a \\ 
    \frac12 G_b   \\ 
    \end{pmatrix}
    \end{align*}
Taking inverses gives
    $$
    \Gamma_{ij} = Q_s^{-1} \vb_{ij}
    $$
where
    $$
    Q_s^{-1} = \frac{1}{\det Q_s} \begin{pmatrix}
        G_s & -F_s \\         
        -F_s & E_s  
    \end{pmatrix}
    $$
 
We wish to bound
    $$
    \max_{i,j} \norm{\Gamma_{ij}}_\text{max},
    $$
over all points in $\Omega$, where $\norm{\cdot}_\text{max}$ is the maximum entry norm. This will then allow us to bound the sum of $\Gamma_{ij}$
    $$\max_{k=1,2}  |\Gamma^k_{11}  +  \Gamma^k_{12} +  \Gamma^k_{22}| \leq 3 \max_{i,j} \norm{\Gamma_{ij}}_\text{max}
    $$
in order to obtain the bound $C_3$.

\begin{lemma} \label{lemma:Gammaijmax}
    Over points in $\Omega$,
        $$\norm{\Gamma_{ij}}_\text{max} \leq \sqrt{2} \frac{|| \vb_{ij}||_{\max}}{(\lambda_s^{lb})^2}.
        $$
    where 
        $$ || \vb_{ij}||_{\max} = \max\{||\vb_{11}||_{\max}, ||\vb_{12}||_{\max},  ||\vb_{22}||_{\max} \}$$
    and $\lambda_s^{lb} $ is a lower bound for $\lambda_s(v)$. 
\end{lemma}

\begin{proof}
First we bound the maximum value of a single $\Gamma_{ij}$ over points in $\Omega$:
    \begin{align*}
    ||\Gamma_{ij}||_{\max} &\leq  ||\Gamma_{ij}||_{Eucl}
    \leq
       || Q_s^{-1}||_{op}\, || \vb_{ij}||_{Eucl}\\
    &\leq \sqrt{2} || Q_s^{-1}||_{op}\, || \vb_{ij}||_{\max}
    \leq \sqrt{2} \frac{|| \vb_{ij}||_{\max}}{(\lambda_s^{lb})^2}.
    \end{align*}
by inequality~(\ref{eqn:Qoperatorbound}).
\end{proof}

\begin{lemma} \label{lemma:Bijmax}
    Over points in $\Omega$ and assuming $s_0, s_1 < .1$
        $$
        || \vb_{ij}||_{\max} = \max\{||\vb_{11}||_{\max}, ||\vb_{12}||_{\max},  ||\vb_{22}||_{\max} \} \leq   0.712 s_0^2 + .874 s_0 s_1 + 3.671 s_1^2 .
        $$
\end{lemma}

\begin{proof}
We have 
    $$
    E= 1 + s_0^2 f_a^2+ \Delta E, \quad F = s_0^2 f_a\, f_b, \quad G= 1 + s_0^2\, f_b^2 + \Delta G.
    $$
so 
    \begin{align*}
    E_a = 2 s_0^2\, f_a \, f_{aa} + \Delta E_b,\quad & E_b = 2 s_0^2\, f_a \, f_{ab} + \Delta E_b,\\
    F_a =  s_0^2(  f_{aa}\, f_b + f_a f_{ba}),\quad &  \, 
    F_b = s_0^2 (f_{ab}f_b + f_a f_{bb}), \\
    G_a = 2s_0^2\,  f_b \, f_{ab} + \Delta G_a, \quad & G_b = 2s_0^2\,  f_b \, f_{bb} + \Delta G_b.
    \end{align*}
Then 
    $$
    F_a -\frac 12 E_b = s_0^2 f_{aa}f_b - \frac{1}{2}\Delta E_b
    $$
and
    $$
    F_b -\frac 12 G_a = s_0^2 f_{a}f_{bb} - \frac{1}{2}\Delta G_a
    $$

Thus, 
    \begin{align*}
    &\max \{ \frac 12 E_a, F_a -\frac 12 E_b, \frac 12 E_b, \frac 12 G_a, F_b - \frac 12 G_a, \frac 12 G_b \}  \\
    &\leq s_0^2 f_{\max}'f''_{\max}+\frac 12 
    \max \{ \Delta E_a, \Delta E_b, \Delta G_a, \Delta G_b\} = s_0^2 f_{\max}'f''_{\max}+\frac 12 
    \Delta_{\alpha}^{ub}\\
    &< 0.712 s_0^2 + .874 s_0 s_1 + 3.671 s_1^2, 
    \end{align*}
where we estimate $ f_{\max}'\, f''_{\max}$  by Collorary \ref{cor:f'maxf''max} and $\Delta_{\alpha}^{ub}(s_0,s_1) = 1.748 s_0 s_1 + 7.342 s_1^2$ from Lemma~\ref{lemma:Delta'max}. 
\end{proof}

\begin{corollary} \label{cor:C3formula} 
    Define 
        $$
        C_3(s_0, s_1) = 3\sqrt{2}\; \left(\frac{0.712 s_0^2 + .874 s_0 s_1 + 3.671 s_1^2  }{ 1 - 2s_1 - 6.284 s_0 s_1^2} \right).
        $$
    Then for $s_0, s_1 < .1$ and all points in $\Omega$, 
        $$\max_{k=1,2}  |\Gamma^k_{11}  +  \Gamma^k_{12} +  \Gamma^k_{22}| \leq C_3(s_0, s_1).$$ 
\end{corollary}

\begin{proof}
    $$
    \max_{k=1,2}  |\Gamma^k_{11}  +  \Gamma^k_{12} +  \Gamma^k_{22}| \leq 3 ||\Gamma_{ij}||_{\max} 
    \leq 3\sqrt{2} \frac{|| \vb_{ij}||_{\max}}{(\lambda^{lb})^2}.
    $$
Now use Lemma \ref{lemma:Bijmax} and the formula for $\lambda_s^{lb}$ from Lemma \ref{lemma:lengthratio2}.
\end{proof}

In Section~\ref{sec:ImprovedC3formula}, we show how to get an improved bound for $C_3$ using more extensive numerical calculations.

\subsection{Proof of Theorem \ref{thm:Cibounds->geodcontrol->finitehorizon}} \label{sec:ThmgeodcontimpliesfhpProof}

\begin{proof}
We work in rectangular coordinates. The cases below are the same  whether the coordinate map $Y: \Omega \rightarrow \M_{s_0}$ gives the bottom or top half of $\M_{s_0} \setminus \T_1$.

I $\Rightarrow$ II: 
Let  $p \in \Omega$, $v \in S_p \Omega$ (so that $\norm{v}_0 =1$), and $v_s = \lambda_s(v) v$. For $0 \leq t \leq \min\{ T_{tube}(p,v), T_{gc}\}$, both $\gamma_{s}(p, v_s, t)$ and $\gamma_{0}(p, v, t)$ remain in $\Omega$. Hence 
    \begin{align*} 
     dist(\gamma_{s}(p, v_s, t), \gamma_{0}(p, v, t))&= \sqrt{(\Uplambda a_s(t))^2 + ( \Uplambda b_s(t))^2} \\
     & \leq 
     \sqrt{2}\max \{ \Uplambda a_s(t), \Uplambda b_s(t)\} \\
     \text{ which by Theorem \ref{thm:deltaabounds}} & \text{ and Corollaries \ref{cor:C1C2formulas} and \ref{cor:C3formula}} \\
     & \leq \sqrt{2} \left( C_1(s) t + C_2^2(s) C_3(s) \frac{t^2}{2} \right) \\\text{ for  }
    \, t\leq T_{gc} = 3 \,  &
     \\
     & \leq \sqrt{2} \left( 3 C_1(s)  + \frac 92 C_2^2(s) C_3(s) \right) \\
     & \leq \sqrt{2} \left (\frac{ \Delta r}{4\sqrt{2}} \right )
  = \frac{\Delta r}{4}
    \end{align*}

\noindent by the assumption of part I. This proves the first implication of Theorem \ref{thm:Cibounds->geodcontrol->finitehorizon}.  

II $\Rightarrow$ III:
Let $p \in \Omega$ and $v \in S_p \Omega$. The geodesics will remain in one of the two disconnected components $\bbR^2\setminus \D_1$ of $\Omega$ for the time under consideration. We proceed by cases on $T_{tube}(p,v)$. 

By  Proposition~\ref{prop:DrhoStrongFiniteHorizon},  $\gamma_0$ will leave $\Omega$ in time at most 3. Suppose $\gamma_0$ is the geodesic that leaves $\Omega$ first.  By Proposition~\ref{prop:DrhoStrongFiniteHorizon}, $\gamma_0$ will enter $\D_{1/2}$ in a $\D_1-$ good way (Definition \ref{def:strongfhpgeom})  in time at most $T_{ret}(\frac 12)$. By equation (\ref{eqn:Deltarover4}) when $\gamma_0 \in \D_1$ (resp. $D_{1/2}$), $\gamma_s \in \D_{3/4}$ (resp. $D_{1/4}$). Hence $T_{ret}(\frac 12)$ from equation~(\ref{eqn:boundonT12}) serves as a bound on the time between good returns to $\D_{1/4}$ (i.e. will immediately go on to enter $\D_{3/4}$) for $\gamma_s$.  

Now suppose $\gamma_s$ is the geodesic that leaves $\Omega$ first (i.e., enters $\D_1$)  at time $t=T_{tube}$. The geodesic has  had to first enter $D_{1/4}$ and then enter $D_{3/4}$. Let $t^* < T_{tube}$ be the largest time for which $\gamma_s(t^*) \in \partial \D_{1/4}$ and $\gamma_s(t) \in \D_{1/4} $ for $t\in [t^*, T_{tube}]$. Thus $\gamma_s$ enters $D_{1/4}$ in a $D_{3/4}-$ good way at time $t^*$. We claim that $t^* \leq T_{ret}(\frac 12) $. By Condition II,  we have that $\gamma_0(t^*) \in \D_0 \setminus int(D_{1/2})$. Since $T_{ret}(\frac 12)$ is an upper bound for  the time at which $\gamma_0$ enters $\D_{1/2}$ in a $\D_1-$ good way, one has that  $t^* \leq T_{ret}(\frac 12)$. Hence $\gamma_s$ enters $D_{1/4}$ in a $\D_{3/4}-$ good way in time at most $T_{ret}(\frac 12)$. 
\end{proof}

\begin{lemma} \label{lemma:monotonicityinGCeqn}
    We consider values of $s_0, s_1 <0.1$, $T_{gc}=3$, and values of $C_i$ given by equation~(\ref{eqn:Ci}).
    \begin{enumerate}
        \item[(a)] For $s_1=0$, there is a unique solution $s_0 = \hat{s}_0 \approx 0.04106$ for which 
        \begin{equation}
            C_1(s_0,0) \, T_{gc}  +  C_2^2(s_0,0)\, C_3(s_0,0) \, \frac{T_{gc}^2}{2}\,   = \frac{\Delta r}{4\sqrt{2}}, \notag
        \end{equation}
        \item[(b)] For any fixed $s_0<\hat{s}_0$, there is a unique $s_1>0$ for which \begin{equation}\label{eqn:contdepequality}
            C_1(s_0,s_1) \, T_{gc}  +  C_2^2(s_0,s_1)\, C_3(s_0,s_1) \, \frac{T_{gc}^2}{2}\,   = \frac{\Delta r}{4\sqrt{2}},
        \end{equation}
        \item[(c)] As $s_0$ increases, the value of $s_1$ for which equation~(\ref{eqn:contdepequality}) holds decreases.
        \item[(d)] If 
        \begin{equation}\label{eqn:contdep2}
            C_1(s_0,s_1) \, T_{gc}  +  C_2^2(s_0,s_1)\, C_3(s_0,s_1) \, \frac{T_{gc}^2}{2}\,   \leq \frac{\Delta r}{4\sqrt{2}},
        \end{equation}
        holds for $(s_0^*, s_1^*)$,  then it holds for any $(s_0, s_1)$ with $0 \leq s_0 \leq s_0^* $ and $0 < s_1 \leq s_1^*$.
    \end{enumerate}
\end{lemma}

\begin{proof}
First, note that each function $C_i(s_0,s_1)$ from equation~(\ref{eqn:Ci}) is an increasing function of $s_0$ and an increasing function of $s_1$. For (a), note that $C_1(0,0)=C_3(0,0)=0$ and $C_1(s_0,0), C_3(s_0,0)$ go to infinity as $s_0 \to \infty$; the result follows by monotonicity. For (b), note that from part (a), for $s_0<\hat{s}_0$, 
    $$
     C_1(s_0,0) \, T_{gc}  +  C_2^2(s_0,0)\, C_3(s_0,0) \, \frac{T_{gc}^2}{2}\, < \frac{\Delta r}{4\sqrt{2}}.
    $$ 
The $C_i$ go to infinity as $\lambda_s^{lb} \to  0$, so  there is a unique $s_1=s_1(s_0)$ for which equality holds. Statements (c) and (d) follow from the monotonicity of the functions $C_i(s_0,s_1)$.
\end{proof}

\subsubsection{Numerical calculation of the geodesic control curve}

We define the geodesic control curve to be the boundary of the region in the $(s_0,s_1)$ plane for which 
    $$
    C_1(s_0,s_1) \, T_{gc}  +  C_2^2(s_0,s_1)\, C_3(s_0,s_1) \, \frac{T_{gc}^2}{2}\,   \leq \frac{\Delta r}{4\sqrt{2}}
    $$
holds. 

In order to numerically find this curve, we set $s_1 = m s_0$ and use Mathematica to solve the resulting equation for $s_0$. With $T_{gc}=3$ and functions $C_i$ given by equation~(\ref{eqn:Ci}), we get the results given in Table \ref{table:geodcontrolvalues}. Using these values, we sketch the region of geodesic control in the $(s_0, s_1)$ plane (see Figure \ref{fig:geodesiccontrolregion}).
\begin{table}[h!]
    \centering
    \begin{tabular}{c | c | c}
        $m$ & $s_0$ & $s_1$  \\ \hline
        $\infty$ & 0 & 0.0067417 \\
        1 & 0.0064204 & 0.0064204 \\
        1/2 & 0.011874 & 0.0059371 \\
        1/4 & 0.019675 & 0.0049188 \\
        1/10 & 0.029929 & 0.0029929 \\
        1/100 & 0.040058 & 0.00040058 \\
        1/1000 & 0.04127 & 0.00004127 \\
        0 & 0.041406 & 0
    \end{tabular}
    \caption{Boundary values of $s_0$ and $s_1$ for which we have $\Delta r / (4 \sqrt 2)$ geodesic control.}\label{table:geodcontrolvalues}
\end{table}

Inside the region of geodesic control, the strong finite horizon property holds. In Sections  \ref{sec:curvaturbounds} and \ref{sec:embeddedsurfacesAnosovpart1}, we will show that the strictly invariant cone conditions holds on a subset of this region which will imply the Anosov property. 

\begin{figure}[h!]
    \includegraphics[width=\textwidth]{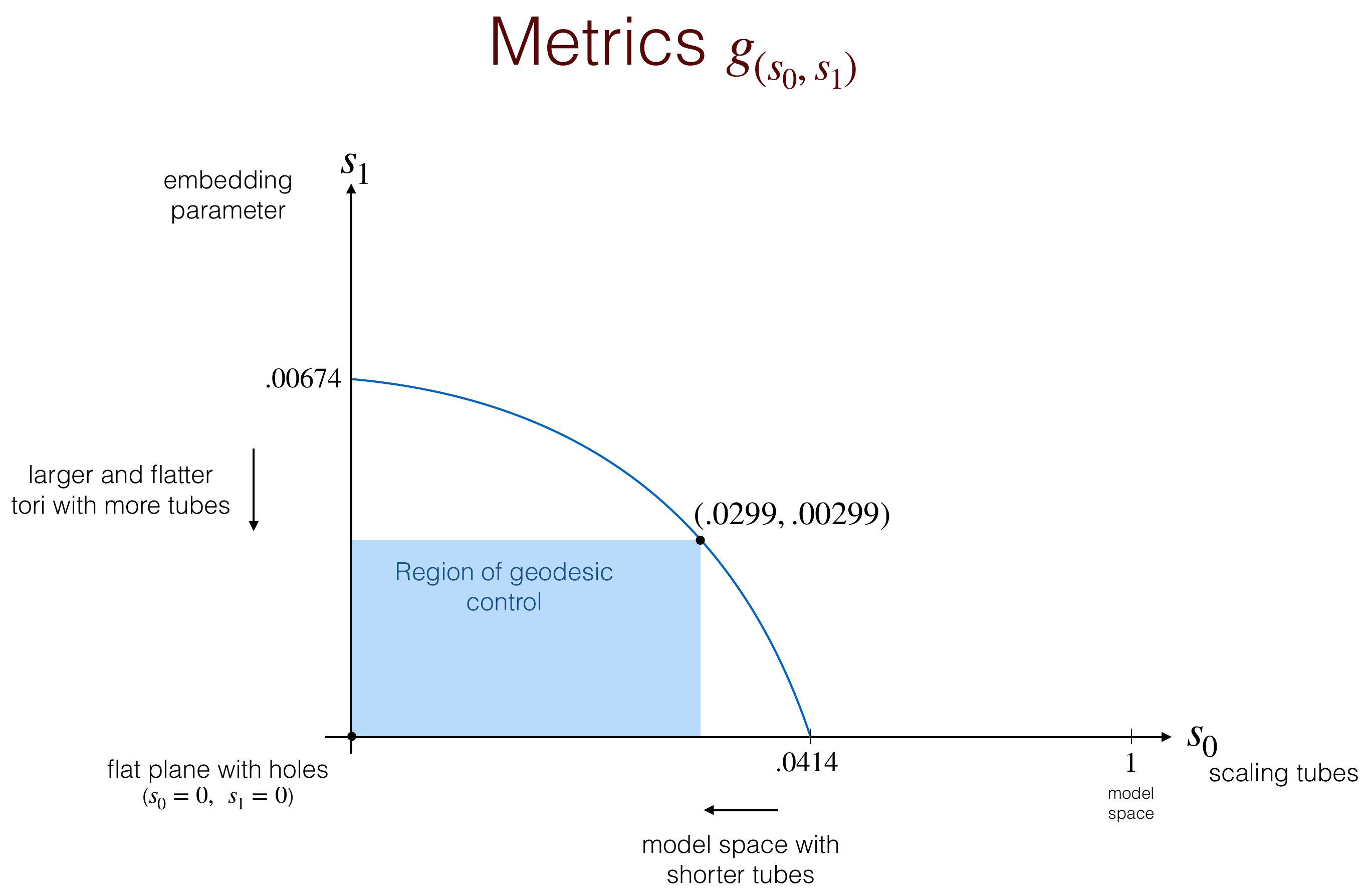}
    \caption{Region of the $(s_0,s_1)$-parameter space for which the metrics $g_{(s_0,s_1)}$ have geodesic control. Coordinates given above come from setting $T_{gc}=3$ and using functions $C_i$ given by Equation~(\ref{eqn:Ci}).}
    \label{fig:geodesiccontrolregion}
\end{figure}

\subsection{Time bounds for metrics with strong finite horizon property}\label{sec:TimeInDiskD14}

The following theorem gives a time-based strong finite horizon property, an alternative to Statement III of Theorem~\ref{thm:Cibounds->geodcontrol->finitehorizon}.

\begin{theorem} \label{thm:Deltattheorem}
    If a metric $g_s$ satisfies Condition I of Theorem~\ref{thm:Cibounds->geodcontrol->finitehorizon}, then
    $\T_{1/4}$ has the $(T_{ret},\Delta t)$ strongly finite horizon property in the metric $g_s$ with 
    $$T_{ret} =  2.30571$$
    and
    $$\Delta t = 0.678015 \,\lambda^{lb}_s = 0.678015 \, \sqrt{1-2s_1-6.284 s_0 s_1^2}.$$
\end{theorem}

\begin{proof} 
Since we will be comparing $\gamma_s$ with $\gamma_0$, everything will be done in the coordinate space, and we will work with the disks $D_\rho$ that get mapped to the tubes $T_\rho$ by the coordinate mapping.
First we consider a geodesic $\gamma_s$ that tangentially intersects $\D_{3/4}$ and show it spends time at least $0.678015 \,\lambda^{lb}_s$ inside $\D_{1/4}$, where $\lambda^{lb}_s$ is a lower bound on the ratio of the length of a vector in the $g_s$ metric to that in the Euclidean metric (see Lemma \ref{lemma:lengthratio2}). Then we show that any other geodesic that intersects $\D_{3/4}$ must also spend more than $0.678015 \,\lambda^{lb}_s$ in $\D_{1/4}$. Picking the explicit bound of $\lambda^{lb}_s$ given by equation~(\ref{eqn:lambdalbs}) yields the formula with $s_0, s_1$.
\begin{enumerate}
    \item  Let $\gamma_s$ be a geodesic that  tangentially intersects $\D_{3/4}$.
    Let $\gamma_s(0) \in \partial \D_{3/4}$ and $\gamma_s'(0)$ be tangent to $\D_{3/4}$. 
    Let $t_{1/4}^\pm$ be the times at which $\gamma_s(t_{1/4}^\pm) \in \partial \D_{1/4}$ and $\gamma_s(t) \in \D_{1/4}$ for $t_{1/4}^- \leq t \leq t_{1/4}^+$. 

    Using rectangular coordinates $(a,b)\subset \bbR^2$, we analyze $\gamma_s$ by comparing it to the  corresponding straight line geodesic $\gamma_0$.   By applying an appropriate translation and rotation, we can assume that $\D_{1/4}$ is centered at the origin, the vector $\gamma_s'(0)$ is horizontal (i.e., its $b$-component is 0), and $\gamma_s(0) = (0,b)$ with $b = 1 - \frac 34 \Delta r >0$.
    
    First we see that in time $\frac 12$ the straight line $\gamma_0$ leaves $\D_0$. Since $\gamma_s(\frac 12)$ must be within $\Delta r/4$ of $\gamma_0$, it must be outside of $\D_{1/4}$.  Thus $|t_{1/4}^\pm| < 1/2$. 
    
    This allows us to get tighter geodesic control. We repeat the geodesic control   argument but for the shorter period of time $T_{gc}=\frac 12$. For $(s_0, s_1)$ satisfying Statement I of Theorem~\ref{thm:Cibounds->geodcontrol->finitehorizon} and $0 \leq t \leq \frac 12$,
        \begin{align}
        \max\{ |\Delta a(t,s)|, |\Delta b(t,s)|\}  
            &\leq  C_1 t + \frac 12  C_2^2 C_3 t^2 \notag \\
            &\leq \frac 12  C_1 + \frac 18  C_2^2 C_3 \notag \\
            &\leq \frac 16 ( 3 C_1 + \frac 92  C_2^2 C_3 ) \notag \\
            & \leq \frac 16 ( \frac{\Delta r}{ 4 \sqrt{2}} )
            = \frac{\Delta r}{ 24 \sqrt{2}}. \label{eqn:DeltaabBounds}
        \end{align}
    Defining
        $$
        \E_{1/4}= \frac{\Delta r}{ 24} , 
        $$
    we see that for any time $- \frac 12 \leq t \leq \frac 12$ the distance between $\gamma_s(t)$ and the corresponding point on the straight line $\gamma_0(t)$ is at most $  \E_{1/4}$.  Hence the $b$-coordinate of $\gamma_s(t)$ for $-\frac 12 \leq t \leq \frac 12$ is bounded by
        $$ 
        |b(\gamma_s(t))| \leq  1 - \frac{3}{4}\Delta r +\E_{1/4}:= b_{max}.
        $$
     
    This allows us to estimate where $\gamma_s$ enters and exits $\D_{1/4}$.
    Let $a^\pm_{1/4}$ be the $a$-coordinate values such that $(a^\pm_{1/4}, b_{max}) \in \partial \D_{1/4}$. 

    Let $\ell(p,q)$ denote the Euclidean line from $p$ to $q$, and let $length_s(c)$ be the length of a curve $c$ measured in the metric $g_s$. The amount of time $\gamma_s(t)$ spends in $\D_{1/4}$ is the $length_s(\gamma_s(t))$ for $t_{1/4}^- \leq t \leq t_{1/4}^+$. Then
    \begin{align*}
        length_s(\gamma_s(t)) 
        & \geq \lambda^{lb}_s * length_0(\gamma_s(t)) \\
        & \geq\lambda^{lb}_s * length_0(\ell(\gamma_s(t_{1/4}^-),\gamma_s(t_{1/4}^+)) \\
        & \geq\lambda^{lb}_s * length_0(\ell((a_{1/4}^-,b_{max}),(a_{1/4}^+,b_{max})))\\
        & = \lambda^{lb}_s * (a_{1/4}^+ - a_{1/4}^-). 
    \end{align*}
    
    The above argument is based on the fact that since the geodesic $\gamma_s$ shadows the above straight line, the angle it sweeps out around the circle $\partial \D_{1/4}$ between the  points $\gamma_s(t_{1/4}^-)$ and $\gamma_s(t_{1/4}^+)$ must be less than $\pi$. Hence, the Euclidean length of the curve $\gamma_s(t)$ for $t \in [t_{1/4}^-,t_{1/4}^+]$ is greater than or equal to the length of the straight line connecting the two endpoints which in turn is bounded below by $a_{1/4}^+-a_{1/4}^-$. 

    To compute $(a_{1/4}^+ - a_{1/4}^-)$, note that
        $$
        (a^\pm_{1/4})^2 + (1-3/4 \Delta r+\E_{1/4})^2 = (1- \Delta r/4)^2, 
        $$
    so
        $$
        a^\pm_{1/4} = \pm \sqrt{(r_{1/4})^2- (r_{3/4}+\E_{1/4})^2}.
        $$
 
    This gives
        $$\Delta t \geq length_s(\gamma_s(t)) \geq \lambda_s^{lb}  * 2 \sqrt{(r_{1/4})^2- (r_{3/4}+\frac{\Delta r}{24})^2} = 0.678015 \, \lambda^{lb}_s.$$ 

    \item Now consider trajectories that cross $\D_{1/4}$ and then enter into the interior of $\D_{3/4}$. Let $p=\gamma_s(t_{1/4}^-)$ be  the basepoint at which the above mentioned tangential geodesic intersects $\D_{1/4}$. Consider the family of geodesics $\gamma_s(t, \alpha)$ that start at point $p$ with angle $\alpha$ so that $\gamma_s(0, \alpha)=p$. Let angle $\alpha_0$ correspond to the tangential trajectory. Consider $\alpha \in (\alpha_0, \pi/2]$. There are a few cases to consider.
 
        \begin{enumerate}
        \item Suppose the  geodesic $\gamma_s(t, \alpha)$ stays on the same sheet while in $\D_{3/4}$, and let $t_{1/4}^+(\alpha)$ be the time at which $\gamma_s(t, \alpha)$ leaves $\D_{1/4}$. If $\gamma_s$ does not leave $\D_{1/4}$, the theorem is trivially satisfied (with $t_{1/4}^+ = \infty$). Otherwise, denote by $\Delta \theta (\alpha)$ the angle that the geodesic rotates from entering to leaving $\D_{1/4}$, as measured relative to the center of the disk that the tube lies above. 

        So long as the geodesics $\gamma_s(t, \alpha)$ remain in negative curvature, there are no conjugate points in this region and hence for $\alpha \geq \alpha_0$, the rotation function  $\Delta \theta (\alpha)$ is a monontone increasing function up until the first $\alpha$ at which $\Delta \theta (\alpha)$ is undefined. This $\alpha $ value corresponds to the geodesic that becomes asympototic to the closed geodesic around the neck of the tube.     

            \begin{enumerate}
            \item Suppose that $\Delta \theta (\alpha)\leq \pi$. We argue as in case (1). The length of the geodesic from entering to leaving $\D_{1/4}$ is bounded below by the Euclidean length of the curve times $\lambda^{lb}_s$. And the Euclidean length of the curve is  greater than the Euclidean length of the straight line connecting the points $\gamma_s(0, \alpha)$ and $\gamma_s(t_{1/4}^+(\alpha), \alpha)$. This distance in turn is greater than the length of the straight line connecting the points $\gamma_s(0, \alpha_0)$ and $\gamma_s(t_{1/4}^+(\alpha_0) ,\alpha_0)$, which equals $\lambda^{lb}_s*(a_{1/4}^+-a_{1/4}^-)$.

            \item  Suppose that $\Delta \theta (\alpha)> \pi$. Then, following a similar bounding argument as in part (1), the distance traveled in the $g_s$ metric is greater than or equal to $\lambda^{lb}_s$ times the Euclidean length of the curve. The Euclidean metric in polar coordinates is $ds^2 = r^2 + r^2 d\theta^2$. In the tube $\D_0$, the radius $r\geq 1/2$. Thus the Euclidean length of the curve is $\geq \frac 12 \pi$ and the $g_s$ length is $\geq 1.57 \, \lambda^{lb}_s$. 
     
            \end{enumerate}   
        \item Suppose $\gamma_s(t, \alpha)$ passes from the lower sheet to the upper sheet before exiting $\D_{1/4}$. Then it must reach the inner most part  of the tube where $r=1/2$. Using a similar argument as in part (1), this length $> 2 \lambda^{lb}_s (r_{1/4} - \frac 12) > .93 \, \lambda^{lb}_s $.  
        \end{enumerate}
    \end{enumerate}
\end{proof}

The following generalizes the previous result for any outer tube of radius $1/4 \leq \rho \leq 1$. It follows from the same argument as above.

\begin{corollary}\label{cor:deltatinTrho} 
    Let $\gamma_s$ enter the tube $\T_{\rho}$, $\frac14 \leq \rho \leq 1$, in a $\T_{3/4}$-good way. 
    The time spent in $\T_\rho$ is bounded below by 
        $$\Delta t_s(\rho) = \lambda^{lb}_s * 2 \sqrt{(r_{\rho})^2- (r_{3/4}+\frac{\Delta r}{24})^2} =\sqrt{ 1 - 2s_1 - 6.284 s_0 s_1^2 } * \sqrt{(r_{\rho})^2- (r_{3/4}+\frac{\Delta r}{24})^2}. $$
\end{corollary}

\section{Curvature bounds} \label{sec:curvaturbounds}

\subsection{Curvature bound $K_{neg}$ inside $\T_{\rho}$ in angular coordinates}\label{sec:curvboundsinA14}

We determine an upper bound $K_{neg}(\rho)< 0$ for the curvature of the $g_s$ metric inside the tube $\T_{\rho}$ using the angular coordinates of Section~\ref{sec:angularcoordinates}: 
    $$Y(\psi, \theta) = (r(\psi) \cos \theta, r(\psi) \sin \theta, s_0 h(\psi))$$
where 
    $$r(\psi) \doteq 1-\frac 12 \cos \psi \text{ and } h(\psi)\doteq \frac 12 +\frac 12 \sin \psi$$
for $\psi \dotin (-\pi/2, \pi/2)$ and $\theta \in (0, 2\pi)$ and the dot indicating a numerical precision coming from the smoothing process (see Section~\ref{sec:smoothing}). Recall that 
    $$\T_{\rho} = \{ Y(\psi, \theta) ~|~ \psi\in [-\psi_{\rho}, \psi_{\rho}] \},$$ 
where $\psi_\rho$ is defined implicitly in Definition~\ref{def:psirho} by the formula $$r(\psi_\rho) = 1 - \rho \Delta r.$$ 

The curvature of a metric can be expressed in coordinates using the first and second fundamental forms
    $$
    K = \frac {eg-f^2}  {(EG-F^2)^2}.
    $$

In what follows, we denote the scaled model space by $\M=\M_{s_0}$. On the scaled model space, we use the angular coordinates of Section~\ref{sec:angularcoordinates} to determine that the curvature $K_\M$  inside the tube is a monotone increasing negative function of $|\psi|$ given by  
    $$
    K_\M(\psi) =\frac{e_\M\,g_\M}{(E_\M G_\M)^2}
    = \frac{-\frac 14 s_0^2   r^3(\psi) h'(\psi) }{(\frac 14 r(\psi)^2 \left( \sin^2 \psi + s_0^2 \cos^2 \psi \right))^2}
    = \frac{-4 s_0^2  h'(\psi) }{r(\psi)\left( \sin^2 \psi + s_0^2 \cos^2 \psi \right)^2}. 
    $$
Note that the numerator $e_\M\,g_\M = -\frac 14 s_0^2   r^3(\psi) h'(\psi) $ is not monotone, but the curvature is still montone due to partial cancelation with the denominator. 

In estimating the curvature for the embedded surface, however, we will not have the same cancellation, so we consider the numerator and denominator of the curvature formula separately. We view these terms as perturbations of the numerator and denominator for the model space case, and we find a weaker form of monotonicity for the numerator $e_\M\,g_\M = -\frac 14 s_0^2   r^3(\psi) h'(\psi)$ which will be enough for our purposes. 

\begin{lemma}\label{lemma:egmodelspacemonotonicitylight}
    The function $e_\M g_\M$ has the following propoerties:
    \begin{enumerate} 
        \item 
            $
            \max_{ |\psi| \leq \psi_{\rho}} e_\M g_\M(\psi) = \max \{ e_\M g_\M (0), e_\M g_\M (\psi_\rho)\} <0. 
            $
        \item There is a unique $\psi^* > 0$ for which $e_\M\,g_\M(\psi^*) = e_\M\, g_\M(0)= -\frac {s_0^2}{64}$.
        \item For $\psi_\rho \geq \psi^*$, 
            \[
            \max_{ |\psi| \leq \psi_{\rho}} e_\M g_\M(\psi) = e_\M g_\M (\psi_\rho)= -\frac {s_0^2}{4} (1- \rho \Delta r)^3 \,  \rho \Delta r <0
            \]
        \item $\psi_{1/4} > \psi_{1/2} > \psi^*$. 
    \end{enumerate}
\end{lemma}

\begin{proof}
\begin{enumerate}
    \item The function 
        $$e_\M\,g_\M(\psi)=-\frac 14 s_0^2   r^3(\psi) h'(\psi)
        = -\frac {s_0^2}{4}   (1-  \frac{\cos  \psi_\rho}{2})^3 \, \frac {\cos \psi_\rho}{2} = -\frac {s_0^2}{4} (1- \rho \Delta r)^3 \,  \rho \Delta r
        $$
    (by Lemmas~\ref{lemma:EFGbdsangcoordmodelsp} and \ref{lemma:rhbounds}) is negative for $|\psi|<\pi/2$ and is monotonically decreasing for small $\psi$. It reaches a minimum at some value $\psi_0$ and then is monotonically increasing, which proves the claim. 
    \item Since $e_\M\,g_\M(\pi/2) = 0$, the existence of a unique $\psi^*$ follows from the Intermediate Value Theorem and monotonicity properties.
    \item Follows from monotonicity properties.
    \item A direct calculation shows $e_\M g_\M(\psi_{1/2}) > e_\M g_\M(0)$.
\end{enumerate}
\end{proof}

Now we derive the formula with which we can get negative upper bounds on the curvature of the embedded space. 

\begin{theorem}\label{thm:knegformula} 
    Let $\rho>0$ such that $\psi_\rho \geq \psi^*$ and  fix $s_0<0.1$ . 
    \begin{enumerate}
        \item Then for sufficiently small $s_1$,  the curvature $K(\psi, \theta)$ for the embedded metric $g_s$ on the tube $\T_\rho$ is negative and bounded above by 
            \begin{align*}
            \max_{Y(\psi, \theta) \in \T_\rho} K(\psi, \theta) & \leq \frac{e_\M\, g_\M(\psi_\rho) +  \Delta eg_{abs}^{ub}(\rho)}
            {\{(E_{\M,abs}^{ub}(\rho)+ \Delta E_{abs}^{ub}(\rho))(G_{\M,abs}^{ub}(\rho) + \Delta G_{abs}^{ub}(\rho))  \}^2}:= K_{neg}(\rho) <0
            \end{align*}
        where
            $$
            \Delta eg_{abs}^{ub}(\rho): =  e_{\M,abs}^{ub} (\rho)\Delta g_{abs}^{ub}(\rho) + g_{\M,abs}^{ub}(\rho) \Delta e_{abs}^{ub}(\rho) + \Delta e_{abs}^{ub}(\rho) \, \Delta g_{abs}^{ub}(\rho).
            $$
        All the bounding terms in the above formula have been defined in Lemmas~\ref{lemma:ModelEFGupperboundsangularcoord} and \ref{lemma:DeltaEFGupperboundsangularcoord}.  

        \item For a fixed $s_0$, as $s_1$ decreases the value of  $K_{neg}(\rho)$ decreases. 
    \end{enumerate}
\end{theorem}

In Section \ref{sec:numericallycomputegenus} we describe how to explicitly calculate the $s_1$ value. 
    
\begin{proof}
We express the first and second fundamental form for the embedded surface as  perturbations of the scaled model space case. 

For the angular coordinates, $F_\M = f_\M =0$ so that
    $$
    K =\frac{(e_\M + \Delta e)(g_\M + \Delta g) - \Delta f^2} {\{(E_\M+ \Delta E)(G_\M +\Delta G) - \Delta F^2\}^2}
    $$
We estimate the numerator and denonimator over $\T_\rho$. For the 
numerator, setting
    $$
    \Delta eg: =  e_{\M} \Delta g+ g_\M \Delta e + \Delta e \, \Delta g,
    $$ 
we get 
    \begin{align}
    (eg -f^2)(\psi, \theta) &= (e_\M\, g_\M + \Delta eg -\Delta f^2) (\psi, \theta)\notag \\
    &\leq  (e_\M\, g_\M + \Delta eg)(\psi, \theta) \notag \\
    &\leq e_\M \,g_\M (\psi_\rho) + \Delta e g_{abs}^{ub}(\rho) \, \text{ by Lemma~\ref{lemma:egmodelspacemonotonicitylight}  and Lemma~\ref{lemma:DeltaEFGupperboundsangularcoord}} \notag \\
    &< 0 \label{eqn:Knum<0}
    \end{align}
for $s_1$ sufficiently small since $\Delta e g_{abs}^{ub}(\rho) \to 0 $ as $s_1\to 0$ and $e_\M \,g_\M (\psi_\rho) < 0$. In our numerical calculations, this is one of the conditions we will impose on $s_1$.

For the denominator, we have
    \begin{equation}\label{eqn:Kdenom>0}
    (E_\M+ \Delta E)(G_\M +\Delta G) - \Delta F^2 > 0 
    \end{equation}
for  $s_1$ sufficiently small since the $\Delta E, \Delta F$ and $\Delta G$ terms all go to zero uniformly in $s_1$ by Lemma~\ref{lemma:DeltaEFGupperboundsangularcoord}. This is another condition imposed upon $s_1$ in our numerical calculations.

Hence 
    \begin{align*}
    |(E_\M+ \Delta E)(G_\M +\Delta G) - \Delta F^2| & \leq |(E_\M+ \Delta E)(G_\M +\Delta G) | \\
    & \leq (E_{\M,abs}^{ub}(\rho)+ \Delta E_{abs}^{ub}(\rho))(G_{\M,abs}^{ub}(\rho) + \Delta G_{abs}^{ub}(\rho) ).
    \end{align*}
 
Combining numerator and denominator, and choosing $s_1$ is sufficiently small that (\ref{eqn:Knum<0}) and (\ref{eqn:Kdenom>0}) hold, finishes the proof of part (1). For (2), note that the numerator is negative and decreases as $s_1$ decreases while the denominator is positive and decreases. 
\end{proof}

For the case of $\rho = \frac 14$, we use the bounds from Lemmas~\ref{lemma:ModelEFGupperboundsangularcoord}, \ref{lemma:DeltaEFGupperboundsangularcoord} and \ref{lemma:egmodelspacemonotonicitylight} to get an upper bound  for $K_{neg}(1/4)$: 
    $$
    K_{neg} < \frac{ -0.0075599 s_0^2 + 0.42466 s_0 s_1 + 1.1138 s_1^2}{0.054052 + 0.00048732 s_0^2 + 0.91349 s_1}.  
    $$
This is negative providing that $s_1 < 0.0170409 s_0$.

\subsection{Curvature bound $K_{pos}$ outside of $\T_{1/4}$}\label{sec:curvboundsoutsideT14}

We split the analysis of the curvature outside $\T_{1/4}$ into two cases: (i) the flat region of the model space corresponding to the complement of all of the tubes $\T_0$ and (ii) the region $\A_{1/4}$, consisting of two disjoint annuli for each tube,  outside of $\T_{1/4}$ but inside $\T_0$.

\subsubsection{Curvature bound $K_{tor}^{ub}$ outside of $\T_0$} \label{sec:curvboundsoutsideT0}
    
\begin{lemma}
    The curvature in the pullback metric $g_s$ on the scaled model surface $\M_{s_0}$ outside of $\T_0$ is bounded above by
        \[
        K_{tor}^{ub} = \frac{\pi^2 s_1^3}{1-s_1-\pi s_0 s_1^2}.
        \]
\end{lemma}

\begin{proof}
We use rectangular coordinates $(u, v, w)$ on $\M_{s_0} \setminus \T_0$ with $w\in\{0, s_0\}$ giving either the bottom plane  ($w=0$) or top plane ($w\doteq s_0)$  (see Section~\ref{sec:rectangularcoordinates}).

For a fixed value of $w$, a direct calculation gives the curvature of the corresponding embedded torus as
    $$  
    K_{tor}(u, v, w)= \frac{\cos \left(\frac{v}{R_2}\right)}{(R_2+w) \left(R_1+(R_2+w) \cos \left(\frac{v}{R_2}\right)\right)}
    $$
Hence the maximum of the curvature in the $g_s$ metric on the surface $\M_{s_0} \setminus \T_0$  is bounded above by
    $$
    K_{tor}^{ub} =
    \frac{1}{(R_2) (R_1-(R_2+s_0) ) }.
    $$
Then substituting
    $$
   R_1(s_1) = \frac 1{\pi s_1^2}, \quad R_2(s_1)=\frac{1}{\pi s_1}
    $$
gives the result.
\end{proof}

\subsubsection{Curvature bound $K_{\A_{1/4}}^{ub}$ in $\A_{1/4} = \T_0 \setminus \T_{1/4}$} \label{sec:curvboundsinsideT0minusT14}

To estimate the curvature in the  region $\A_{1/4}$ we use polar coordinates. 

\begin{proposition}\label{prop:Kpolarub} 
    For $s_0 < .1$ and  $s_1$ sufficiently small, the curvature $K$ in the  region $\A_{1/4}$ is bounded above by 
        \begin{align*}
        K  & 
        \leq \frac{ \Delta eg_{abs}^{ub}}{\{(E_\M^{lb}- \Delta E_{abs}^{ub})(G_\M^{lb} - \Delta G_{abs}^{ub})-(\Delta F_{abs}^{ub})^2  )\}^2}
        := K_{\A_{1/4}}^{ub}
        \end{align*}
    where these terms are defined in Section~\ref{sec:metricinpolarcoordinates}.
\end{proposition}

\begin{proof}
We use the estimates from Section~\ref{sec:metricinpolarcoordinates}. 
    \begin{align*}
      K &= \frac{eg-f^2}{(EG-F^2)^2}
      \leq \frac{eg}{(EG-F^2)^2}
      = \frac{e_\M\, g_\M + \Delta eg}{(EG-\Delta F^2)^2} \leq \frac{\Delta eg}{(EG-\Delta F^2)^2}  
    \end{align*}
since $e_\M\, g_\M < 0$ on the region $\A_{1/4}$. Note that 
    $$
    \Delta eg \leq \Delta eg_{abs}^{ub}
    $$ 
and
    \begin{equation}\label{eqn:Kposgroupcheck}
    \begin{gathered}
    E \geq E_\M^{lb} - \Delta E_{abs}^{ub} >0 \\ 
    F  = \Delta F \leq  \Delta F_{abs}^{ub}  \\ 
    G \geq G_\M^{lb} - \Delta G_{abs}^{ub} >0 \\ 
    (E_\M^{lb}- \Delta E_{abs}^{ub})(G_\M^{lb} - \Delta G_{abs}^{ub})-(\Delta F_{abs}^{ub})^2   >0 ~
    \end{gathered}
    \end{equation}
for $s_1$ sufficiently small, where these terms are defined and computed in Lemmas~\ref{lemma:modelEFGupperboundspolarcoord} and \ref{lemma:deltaEFGupperboundspolarcoord}. These give three more conditions imposed on $s_1$ in our numerical calculations. Then
    \begin{align*}
    K  & 
       \leq \frac{ \Delta eg_{abs}^{ub}}{\{(E_\M^{lb}- \Delta E_{abs}^{ub})(G_\M^{lb} - \Delta G_{abs}^{ub})-(\Delta F_{abs}^{ub})^2  )\}^2}:= K_{\A_{1/4}}^{ub}
    \end{align*}
\end{proof}

Combining the estimates for curvature in these two regions lets us get an upper bound for the curvature outside of $\T_{1/4}$:

\begin{proposition}\label{prop:kposbound} 
    For $s_0 < .1$ and  $s_1$ sufficiently small, 
        \begin{enumerate}
        \item  the curvature $K$ in the  region $\M_{s_0} \setminus \T_{1/4}$  is bounded above by
            $$
            \max_{\M_{s_0} \setminus \T_{1/4}} K \leq \max\{ K_{tor}^{ub},
            K_{\A_{1/4}}^{ub}\}:= K_{pos}
            $$
        \item as $s_1$ decreases, $K_{pos}$ decreases. 
        \end{enumerate}
\end{proposition}

Result (2) follows since each of the functions $K_{tor}^{ub}$ and $K_{\A_{1/4}}^{ub}$ are monotonic is $s_1$. 

Since  the leading order term of $K_{tor}$ is $s_1^3$,  for $s_1$ small enough the maximum of the two terms will be given by $K^{ub}(\A_{1/4})$. Estimating the terms in $K_{\A_{1/4}}^{ub}$ using the polar coordinates of Section~\ref{sec:metricinpolarcoordinates} gives
    $$
    K_{\A_{1/4}}^{ub} \leq  \frac{6.57265 s_0 s_1+16.507 s_1^2}{0.87258 - 10.7863 s_1}. 
    $$

\section{ Embedded surfaces with Anosov geodesic flow} \label{sec:embeddedsurfacesAnosovpart1}

In Section~\ref{sec:MainThmProof}, we reprove  the existence of embedded surfaces with Anosov geodesic flow using our two-parameter family of metrics $g_{(s_0, s_1)}$.  We   use Mathematica to  numerically estimate parameter values for which we can prove the metric gives an embedded surface with Anosov geodesic flow (Section~\ref{sec:numericallycomputegenus}) \cite{Donnay-Visscher-mathematica}.  Employing  an algorithm that optimizes (minimizes) the genus of these Anosov embedded examples---i.e., finds the largest $s_1$ value for which our methods show $g_{(s_0,s_1)}$ has an Anosov geodesic flow.

\subsection{Main Theorem}\label{sec:MainThmProof}

The following theorem provides a parametrized version of the Donnay-Pugh Theorem: that there are compact embedded surfaces whose geodesic flow is Anosov. This approach  provides a method for computing the genus of such a surface. We use Mathematica to carry out numerical computations of parameter values  (see  Section~\ref{sec:numericallycomputegenus}) and use a root finding algorithm to optimize the genus produced by this method.

\begin{theorem}\label{thm:AnosovCondition} 
    Consider the model space $\M_0$ equipped with a pull-back metric $g_{(s_0, s_1)}$. 
        \begin{enumerate}
        \item Given $m>0$ there exists $s_0^*>0$ such that the sets $\T_{3/4} \subset \T_{1/4}$ have the $(T_{ret},\Delta t)$ strongly finite horizon property in the  $g_{(s_0^*,m s_0^*)}$ metric, with values of $T_{ret}$ and $\Delta t$ given in Theorem~\ref{thm:Deltattheorem}. Moreover, any $(s_0,s_1)$ with $s_0 \leq s_0^*$ and $s_1 \leq ms_0^*$ also has the same strongly finite horizon property.
    
        \item Then there exists an $s_1^* \leq m s_0^*$ such that $g_{(s_0^*,s_1)}$ has an Anosov geodesic flow for any $0\leq s_1 \leq s_1^*$.

        \item Additionally, there exists an $s_1^{e} \leq s_1^*$ for which the metric $g_{(s_0^*,s_1^e)}$ comes from a compact embedded surface. From above, this metric has an Anosov geodesic flow.
        \end{enumerate}
\end{theorem}

\begin{proof}
\begin{enumerate} 
    \item Let $m>0$. Setting $s_1 = ms_0$, and recalling that $T_{gc}=3$ and $\Delta r = 1-\cos(\pi/6)$ are constants, the left hand side of the inequality
    \begin{equation}
    C_1(s) T_{gc} + C_2^2(s) C_3(s) \frac {T_{gc}^2}2 \leq \frac{\Delta r}{4\sqrt{2}}
    \end{equation}
    (Condition I from Theorem~\ref{thm:Cibounds->geodcontrol->finitehorizon}) becomes a function of just one variable $s_0$. 
    The functions $C_i(s_0, ms_0), i=1,2,3$ (see Corollary \ref{cor:C1C2formulas} and Corollary~\ref{cor:C3formula}), are all strictly monotone increasing functions of $s_0$ with $$
    C_1(0,0)=0 \;\text{and}\; C_2^2(0,0)C_3(0,0)=0.$$ 
    
    As $s_0$ increases from $0$,  $\lambda_s^{lb} = (1-2ms_0-6.284m^2s_0^3) \to  0$, so the the denominator of $C_i$ approaches 0 and  the $C_i$ go to infinity. Thus, there is a unique $s_0^*>0$ before this singularity for which equality holds.
    
    Then, Lemma \ref{lemma:monotonicityinGCeqn}(d) proves that the sets $\T_{3/4} \subset \T_{1/4}$ have the $(T_{ret},\Delta t)$ strong finite horizon condition in the $g_{(s_0,s_1)}$ metric for any $s_0 \leq s_0^*$ and $s_1 \leq ms_0^*$.

    \item To identify an $s_1^*$ so that $g_{(s_0^*,s_1^*)}$ has an Anosov geodesic flow, Theorem \ref{thm:Anosov} and Lemma~\ref{lemma:uknegukpos} show we need only check that $s_0^*, s_1^*$ produce values of $K_{neg}$, $K_{pos}$, $\Delta t$, and $T_{ret}$ so that
        \begin{equation}\label{eqn:Riccaticondition}
        u_{neg}(\Delta t) + u_{pos} (T_{ret})  > 0,
        \end{equation}
    where 
        $$
        u_{pos} (T_{ret}) = -\sqrt{K_{pos}}\, \tan( \sqrt{K_{pos}} \, T_{ret}) 
        $$
    and 
        $$
        u_{neg}(\Delta t)= \sqrt{-K_{neg}}\tanh(\sqrt{-K_{neg} }\, \Delta t).
        $$

    These terms have explicit formulas determined in previous Theorems and are continous functions of $s_0$ and $s_1$: $K_{neg}(s_0,s_1)$ given by Theorem~\ref{thm:knegformula}, $K_{pos}(s_0,s_1)$ by Proposition~\ref{prop:kposbound}, and $\Delta t(s_0, s_1)$ and $T_{ret}$ by Theorem~\ref{thm:Deltattheorem}. 

    As $s_1$ decreases, $K_{neg}(s_0^*, s_1)<0$ and  $K_{pos}(s_0^*, s_1)>0$ both decrease and $\Delta t(s_0^*, s_1)$ increases. Thus $u_{neg}(\Delta t)>0$ and $ u_{pos} (T_{ret})<0$ increase, and hence $u_{neg}(\Delta t) + u_{pos} (T_{ret})$ increases. When $s_1 =0$, we have $K_{neg}(s_0, 0) < 0$, $K_{pos}(s_0, 0) = 0$ so that  $  u_{neg}(\Delta t) + u_{pos} (T_{ret})  >0 $. Thus there exists $s_1^*>0$ for which  $  u_{neg}(\Delta t) + u_{pos}(T_{ret})  >0 $ and hence $g_{(s_0^*, s_1)}$ is Anosov for all $0\leq s_1 \leq s_1^*$. 

    \item Choose any positive integer $n$ such that 
        \begin{equation}\label{eqn:embeddingcondition}
        s_1^e = \frac{1}{\sqrt3 n}<s_1^*
        \end{equation}
    By Section~\ref{sec:genuscalculation}, the corresponding metric $g_{(s_0^*,s_1^e)}$ is the pull-back of a metric on embedded surface with genus $6 n^3 +1$, and by part (2), it generates an Anosov geodesic flow.
\end{enumerate}
\end{proof}

\subsection{Numerical methods and minimizing the genus of the embedded surface}\label{sec:numericallycomputegenus} 

We use numerical methods and take advantage of the monotonicity of various functions to find boundary values of inequalities from the above proof. 

For a given $m$, we set $s=(s_0, m s_0)$ and use Mathematica to numerically solve 
    $$C_1(s) T_{gc} + C_2^2(s) C_3(s) \frac {T_{gc}^2}2 = \frac{\Delta r}{4\sqrt{2}}$$
denoting the solution $s_0^*$. The curve $(s_0^*, s_1 = m s_0^*)$ serves as the boundary of a closed region in the $(s_0,s_1)$ plane for which we have geodesic control, shown in blue below. Note that simply truncating  the decimal solution at any precision will produce an $s_0^*$ value for which inequality holds in the above equation. 
 
Among the metrics that satisfy geodesic control, we need to find those for which we can show the corresponding geodesic flow is Anosov. If the   Riccati condition~(\ref{eqn:Riccaticondition}) holds for $s=(s_0^*, m s_0^*)$  then we set $s_1^* = ms_0^*$. If not, we use a bisection method on the $s_1$ interval $[0, m s_0^*]$ to determine an $s_1^*$ for which (\ref{eqn:Riccaticondition}) holds  and that is within a specified tolerance of the boundary of the $u_{neg}(\Delta t) + u_{pos} (T_{ret})>0$ region. At this stage, we verify that $s_1^*$ satisfies the conditions in equations (\ref{eqn:Knum<0}) and (\ref{eqn:Kdenom>0}) in the definition of $K_{neg}$ and equations (\ref{eqn:Kposgroupcheck}) in the definition of $K_{\A_{1/4}}^{ub}$. This gives an approximately largest value of $s_1^*$ for which $g_{(s_0^*,s_1^*)}$ has Anosov geodesic flow.

Selecting the smallest positive integer $n$ that satisfies condition~(\ref{eqn:embeddingcondition}), we get an embedded surface from the metric $g_{(s_0^*, s_1^e)}$ with Anosov geodesic flow that has genus $6n^3+1$. If decreasing $n$ by one yields an $s_1$ value for which the $u$ condition (\ref{eqn:Riccaticondition}) is less than $0$, we know that we have the minimal genus for this value of $m$ for which our methods work.
 
This process, with tolerance set at $10^{-10}$, produces the following data:
    \begin{table}[h!]
    \renewcommand{\arraystretch}{1.2}
    \centering
    \begin{tabular}{c | c | c | c | c | c | c }
        $m$ & $s_0^*$ & $s_1= ms_0^*$ & $s_1^*$ & $s_1^e$ & $n$ & $g$ ($*10^{11}$)  \\ \hline \hline
        .1 & $2.9929*10^{-2}$ & $2.9929*10^{-3}$ & $1.2527*10^{-4}$ & $1.2526*10^{-4}$ & 4609 & $5.87451$ \\
        .01 & $4.0058*10^{-2}$ & $4.0058*10^{-4}$ & $1.6724*10^{-4}$ & $1.6720*10^{-4}$ & 3453 & $2.47025$ \\
        .006 & $4.0592*10^{-2}$ & $2.4355*10^{-4}$ & $1.6945*10^{-4}$ & $1.6941*10^{-4}$ & 3408 & $2.37493$ \\
        .005 & $4.0727*10^{-2}$ & $2.0363*10^{-4}$ & $1.7000*10^{-4}$ & $1.7000*10^{-4}$ & $3396$ & $2.34993$ \\
        .004 & $4.0862*10^{-2}$ & $1.6344*10^{-4}$ & $1.6344*10^{-4}$ & $1.6341*10^{-4}$ & 3533 & $2.64595$ \\
        .003 & $4.0997*10^{-2}$ & $1.2299*10^{-4}$ & $1.2299*10^{-4}$ & $1.2297*10^{-4}$ & 4695 & $6.20952$ \\
    \end{tabular}
    \caption{Results of the AnosovLoop algorithm for specified $m$ values. Note that for decreasing values of $m$, the genus $g$ decreases and then increases. A minimum will be found between $m=0.006$ and $m=0.004$.}
    \end{table}

Now we search for the optimal result across values of $m$. As $m$ decreases and $s_0^*(m)$ increases, we find numerically that $s_1^*$ increases until the point at which it coincides with the boundary of the geodesic control region. The intuition behind this result is that as $s_0^*$ increases, the outer parts of the tubes are flattened less and hence $K_{neg}$ increases. This in turn allows larger values for $K_{pos}$ and hence larger values of $s_1$. These values of $s_1^*$ produce a region in which we can prove the geodesic flow is Anosov (the purple region in Figure~\ref{fig:Anosovmetrics}).

    \begin{figure}
    \includegraphics[width=\textwidth]{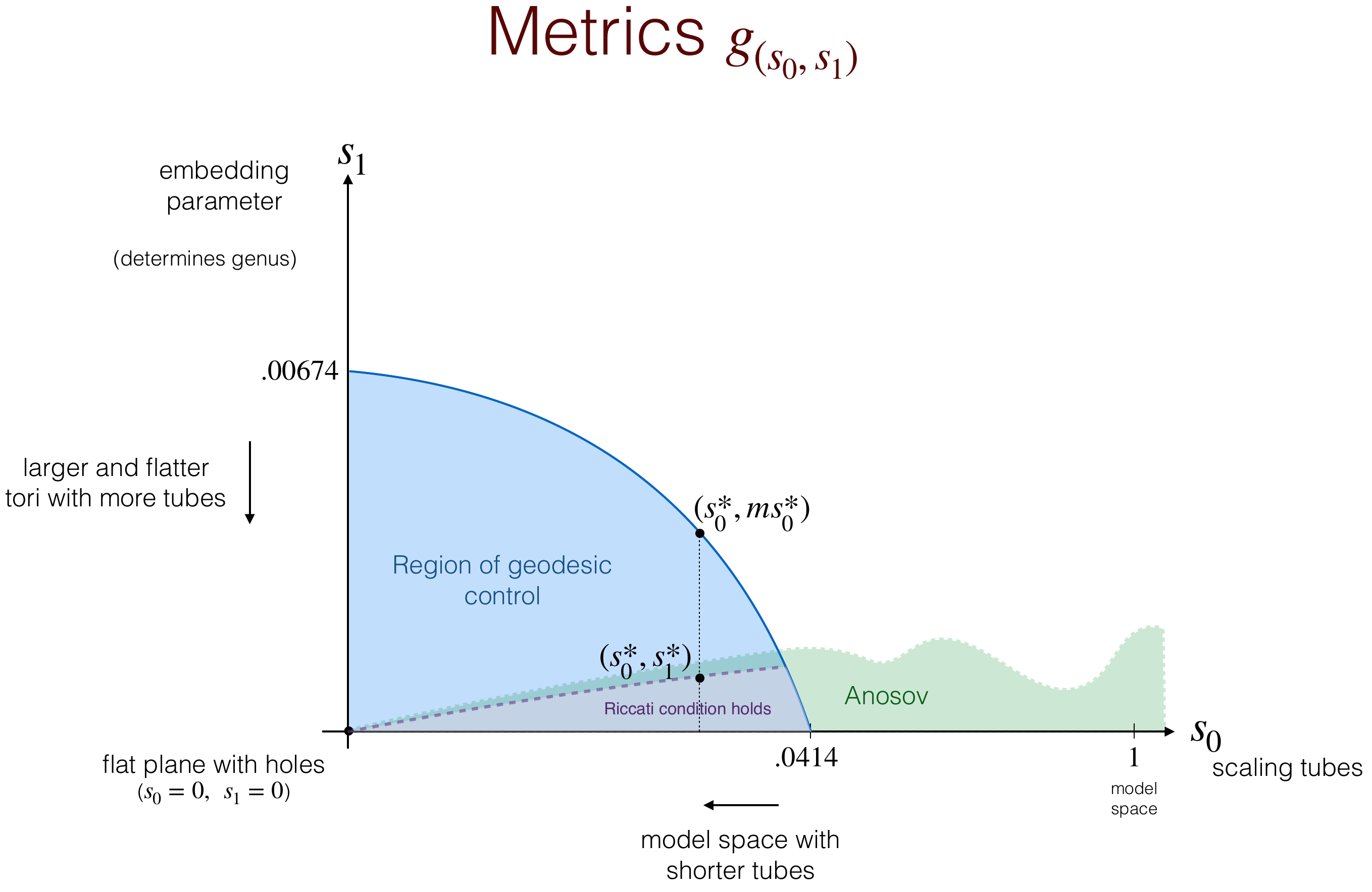}
    \caption{An illustration of various regions of interest in the $(s_0,s_1)$ parameter space of metrics.}\label{fig:Anosovmetrics}
    \end{figure}

The smallest genus embedded Anosov surface comes from the largest value of $s_1^e$, which is found near the intersection of this boundary curve and the geodesic control curve. 
This intersection point can be found numerically using a root finding algorithm. Following the same process of finding an embedded surface as above for a given $m$, yields the following:

\begin{theorem}\label{thm:genusbound}
    There exist smooth compact embedded surfaces of genus $\leq 2.3313 \times 10^{11}$ with Anosov geodesic flow. 
\end{theorem}
     
The embedded surfaces comes from a fundamental region consisting of $n$ by $3n^2$ copies of the basic tile with genus $g= 6n^3 +1$  as described in Section~\ref{sec:genuscalculation}. The above value of genus comes from $n = 3387$ so genus = $233,129,289,619$.

The purple region in parameter space for which our method proves Anosov does not contain all the Anosov metrics. Many of the bounds we use to carry out our calculations are rough. In Section~\ref{sec:SophisticatedEstimates}, we improve on some of the bounds and show that the set of Anosov metrics is strictly larger than the set we have identified here. 

Even without additional calculations, we can see that our method does not capture all the Anosov metrics. For example, the set of Anosov metrics includes a neighborhood of the $s_0$ axis since any model space $\M_{s_0}$ has an Anosov geodesic flow and being Anosov is an open condition for the model space metrics. Furthermore, there is an open neighborhood of Anosov systems around the set of points $(s_0^*, m s_0^*)$ on the geodesic control curve whose geodesic flow are Anosov.

\subsection{A note on validating numerical calculations in Mathematica}\label{sec:validatingMathematica}

The bound for the genus of the embedded Anosov system in Theorem \ref{thm:genusbound} is the result of a long chain of calculations involving upper and lower bounds. To insure that our final result is valid, we need to insure that Mathematica computes valid numeric upper (resp., lower) bounds at each stage of the calculation. 

Using numerical calculations in Mathematica  introduces small computational errors that are possibly problematic for ensuring correct bounds. For example, the number $\sqrt{3}$ stored numerically in machine precision is not exactly $\sqrt{3}$, and so the machine precision number is either a lower bound for the interval $(\sqrt{3},5)$ or an upper bound for the interval $(0,\sqrt{3})$ but not both. 
Our methods to guarantee our the numerically stored numbers are correct bounds are as follows. 

There are two types of computation we address: first, we do all arithmetic calculations using rational numbers, for which Mathematica is able to do computations exactly. Second, for any numerical solvers or functions that are not arithmetic, for upper bounds we ask Mathematica to carry out the computation to a greater precision  than what we are interested in, add in a small number larger than the computational precision but smaller than our specified working precision so that we are sure to have an upper bound, and then round up to the working precision and store as a rational number. Similarly for lower bounds.

Regardless of how one identifies parameter values, one can check that a proposed $g_{(s_0,s_1)}$ metric satisfies both the geodesic control and Riccati conditions that we use to prove Anosov. We carry out this confirmation for the metrics  that   come  from our optimization algorithm as a double check that our numerical methods have correctly worked.

\section{Refined estimates and improved genus} \label{sec:SophisticatedEstimates}
With the goal of lowering our bound on the genus of an embedded compact surface with Anosov geodesic flow, we refine our estimates at several stages of the proof:  the geodesic control condition and the formulas for the value of the coefficient  $C_3$ that appear in the geodesic control equation, and the bounds on negative curvature and the time spent in the negative curvature region. 

In this section we show that the conditions for geodesic control (Statement I of Theorem~\ref{thm:Cibounds->geodcontrol->finitehorizon}) still imply the same $(T_{ret},\Delta t)$ strongly finite horizon condition of Theorem~\ref{thm:Deltattheorem} under the less stringent conditions of $T_{gc}=2.5$ and an improved $C_3(s_0,s_1)$ formula. This allows us to find a larger region in the $(s_0,s_1)$ plane for which we have geodesic control.

We next introduce a  way to  better estimate the value  $u_{neg}$ of the Riccati solution in the   negatively curved tube. Note that the negative curvature bound $K_{neg}$ is achieved on the boundary of the tube $\T_{1/4}$, but a geodesic entering this tube in a good way spends most of its time in curvature which is more negative than $K_{neg}$. We capture the more negative curvature by picking an interior tube $T_{1/2}\subset T_{1/4}$ and obtaining a better bound on the negative curvature in this interior tube. We also get bounds on the time spent in this more negatively curved region. Improving the bounds on Ricatti solutions in the negatively curved region allows for more positive curvature leading to embedded surfaces of lower genus.

\subsection{Improvements on conditions for geodesic control}

\subsubsection{Improved $T_{gc}$}
\label{sect:improvedTgc}

The following theorem is an improvement of Theorem~\ref{thm:Cibounds->geodcontrol->finitehorizon}; the only difference is replacing $T_{gc}=3$ by $T_{gc}=2.5$. This makes Statement I weaker (i.e., it holds for more values of $(s_0,s_1)$), which in turn makes the corresponding I$\,\Rightarrow\,$III a stronger theorem. While the statement is a simple improvement, the proof is more involved and so we have reserved this result as a refinement of our basic method. 

Assume that the functions $C_i(s_0, s_1), i=1,2,3$, satisfy the bounding conditions of equation $(\ref{eqn:threebounds})$. Note that we will use an improved bounding formula for $C_3$ as described in Section~\ref{sec:ImprovedC3formula}. As before, for fixed value of $s$ and for $p \in \Omega$ and $v \in S_p \Omega$, let 
    $$T_{tube}(p,v)>0$$
be the first time either $\gamma_{0}(p, v, t)$ or $\gamma_{s}(p, v_s, t)$ leaves $\Omega$.  

\begin{theorem}\label{thm:Cibounds->geodcontrol->finitehorizon:Improved} (Improvement of Theorem~\ref{thm:Cibounds->geodcontrol->finitehorizon}.)
    Consider a metric $g_s$ and set $T_{gc}=2.5$. Then for the following statements, I $\Rightarrow$ II $\Rightarrow$ III.
        \begin{enumerate}
        \item[I.] For $s=(s_0,s_1)$ with $s_0, s_1 < 0.1$,
            \begin{equation}\label{eqn:contdepv2}
            C_1(s) \, T_{gc}  +  C_2^2(s)\, C_3(s) \, \frac{T_{gc}^2}{2}\,   \leq \frac{\Delta r}{4\sqrt{2}},
            \end{equation}
        \item[II.] for all $p \in \Omega$, $v \in S_p \Omega$, and $0 \leq t \leq \min\{ T_{tube}(p,v), T_{gc}\}$,
            \begin{equation}\label{eqn:Deltarover4v2}
            dist(\gamma_{s}(p, v_s, t), \gamma_{0}(p, v, t)) \leq \frac{\Delta r}{4}  
            \end{equation} 
        \item[III.] on $\M_{s_0}$ the  pair of sets $\T_{3/4} \subset \T_{1/4}$ is $T_{ret}$ strongly finite horizon in the $g_{s}$ metric with 
            $$
            T_{ret} = 2.30571. 
            $$
        \end{enumerate}
\end{theorem}

The case I $\Rightarrow$ II follows the same proof as for Theorem~\ref{thm:Cibounds->geodcontrol->finitehorizon} (see Section~\ref{sec:ThmgeodcontimpliesfhpProof}). Showing that II $\Rightarrow$ III is more complicated; it involves comparing geodesics $\gamma_s$ to reference straight line geodesics $\gamma_0$ but changing the comparision geodesic in certain cases to the point of exiting $\D_0$ rather than at the initial starting point. The following lemma summaries the information we will need about straight line geodesics. 

\begin{lemma} \label{lemma:diskarrangementangletimeboundsD1version2} 
    Every straight line geodesic that starts in $\bbR^2 \setminus \text{int}(\D_0)$ will: 
        \begin{itemize}
        \item intersect $\D_1$ at times $t_{-1} < 0 < t_{+1}$  with $|t_{\pm 1}|\leq 2.5$ and $t_{+1} - t_{-1} \leq 3$, and these time bounds are sharp.
        \item intersect $\D_{1/2}$ in a $\D_1-$ good way at times $t_{-1} < t_{-\frac 12} < 0 < t_{+\frac 12 } < t_{+1}$  with $|t_{\pm \frac 12}|  \leq 2.152851$  and $t_{+\frac 12} - t_{-\frac 12} \leq 2.30571$, and these time bounds are sharp (up to numerical precision).
        \end{itemize}
\end{lemma} 

\noindent This lemma follows from the same logic as in the proof of Corollary \ref{thm:D1StrongFiniteHorizon}: examine the trajectory that enters $\D_0$ with angle $\pi/6$. 

\begin{proof}[Proof of Theorem~\ref{thm:Cibounds->geodcontrol->finitehorizon:Improved} II $\Rightarrow$ III]

Consider such a metric $g_s$ and a geodesic $\gamma_s$ in that metric. We wish to show the strong finite horizon condition of Statement III under the assumption that $g_s$ geodesics stay $\Delta r/4$-close to partner $g_0$ geodesics only for time up to 2.5. We proceed by cases on the location of $\gamma_s(0)$.

First, suppose that that $\gamma_s(0)\in \bbR^2 \setminus \text{int}(\D_0)$.  
If $\gamma_0(t)$ intersects $\D_1$ before $\gamma_s(t)$ does, then by Lemma \ref{lemma:diskarrangementangletimeboundsD1version2} this intersection occurs in time $t_{+1} \leq 2.5$. Hence by the geodesic control condition (\ref{eqn:Deltarover4v2}), $\gamma_s(t_{+1})\in \D_{3/4}$ and $\gamma_s(t_{+\frac 12}) \in \D_{1/4}$.
If $\gamma_s(t)$ intersects $\D_1$ before $\gamma_0(t)$, then the  same argument as in the proof of Theorem \ref{thm:Deltattheorem} applies.

Now suppose that $\gamma_s(0)\in  \D_0$. If $\gamma_s(0) \in \D_{3/4}$, then it trivially satisfies the strongly finite horizon condition. If $\gamma_s(0)\in  \D_0\setminus  \D_{3/4}$, then we divide into two cases: either $\gamma_s$ intersects $\D_{3/4}$ before leaving the particular disk in $\D_0$ or it does not. Note that, since individual disks $D_0 \subset \D_0$ are tangent, a geodesic could move from one disk to another without leaving $\D_0$; we wish to still count this behavior in the phrase ``leaving $\D_0$''.

If $\gamma_s$ does intersect $\D_{3/4}$ before leaving $\D_0$, then examine the corresponding straight-line geodesic $\gamma_0$ and consider two subcases. If $\gamma_0$ intersects $\D_1$ before leaving $\D_0$, it will do so in time at most $\frac 12$ at which point, by geodesic control, $\gamma_s$ will intersect $\D_{3/4}$. If $\gamma_0$ does not intersect $\D_1$, it will leave $\D_0$  in time at most 1 and then move away from $\D_0$. By geodesic control, $\gamma_s$ must intersect $\D_{3/4}$ in time at most 1. 

The final, and most complicated, case is when $\gamma_s(0)\in  \D_0\setminus  \D_{3/4}$ but $\gamma_s$ leaves $\D_0$ without intersecting $\D_{3/4}$. Let $t^* > 0$ be the first time at which $\gamma_s(t^*) \in \partial \D_0$, and let  $\overline \gamma_0$ be the straight line geodesic that matches up with $\gamma_s(t^*)$:   
    $$
    \overline \gamma_0(0) = \gamma_s(t^*), \text{ and } \, \overline \gamma_0'(0) = \gamma_s'(t^*) / \norm{\gamma_s'(t^*) }_0.
    $$
(We assume here that $\gamma_s$ does not reach $\D_1$ before $\gamma_0$ does; otherwise the argument in the proof of Theorem \ref{thm:Deltattheorem} again applies.)

By Lemma \ref{lemma:diskarrangementangletimeboundsD1version2}, there are times $-2.5 \leq t_{-1}< t_{-\frac 12} < 0 <   t_{+\frac 12} < t_{+1} \leq 2.5$ for which $\overline\gamma_0(t_{\pm 1}) \in \D_1$ with $t_{+1} - t_{-1} \leq 3$ and $\overline\gamma_0(t_{\pm \frac 12}) \in \D_{1/2}$ with $t_{+\frac 12} - t_{-\frac 12} \leq 2.30571$. By geodesic control, $\gamma_s(t^* + t) \in \D_{1/4}$ for $t_{+\frac 12} \leq t \leq t_{+1}$ and $\gamma_s (t^*+t_{\pm 1}) \in \D_{3/4}$.

\begin{itemize}
    \item If $|t_{-\frac 12}| \geq t^*$ (i.e. $t^* + t_{-\frac 12} \leq 0$), then $t^* + t_{+\frac 12} \leq t_{+\frac 12} - t_{-\frac 12} \leq 2.30571$. 
    \item If $|t_{-\frac 12}| \leq t^*$ (i.e. $t^* + t_{-\frac 12} \geq 0$), then before leaving $\D_0, \gamma_s$ will intersect $\D_{1/4}$:  $\overline\gamma_0(t_{- \frac 12}) \in \D_{1/2}$, so $\gamma_s(t^* + t_{- \frac 12}) \in \D_{1/4}$ with $ 0 \leq  t^* + t_{- \frac 12} \leq t^*$. Then the time from $\gamma_s$ leaving $\D_{1/4}$ to next returning to $\D_{1/4}$ in a good way is bounded above by $t^* -(t^* +t_{- \frac 12} ) + t_{+  \frac 12} = t_{+\frac 12} - t_{-\frac 12} \leq 2.30571$. 
\end{itemize}

\noindent We have now proven the result for all cases. 
\end{proof}

Decreasing the value of $T_{gc}$ from 3 to 2.5 and improving the estimate for $C_3$ (see below) causes the geodesic control curve to shift to the right, thereby expanding the region   in $(s_0, s_1)$ space for which geodesic control holds (see Figure   ~\ref{fig:IncreasingGeodesicControl}): the boundary of the geodesic control region (red curve) has moved outward from the boundary gotten using the weaker estimates from earlier in the paper (blue curve). 

\begin{figure}[h]
    \centering
    \includegraphics[width=.5\textwidth]{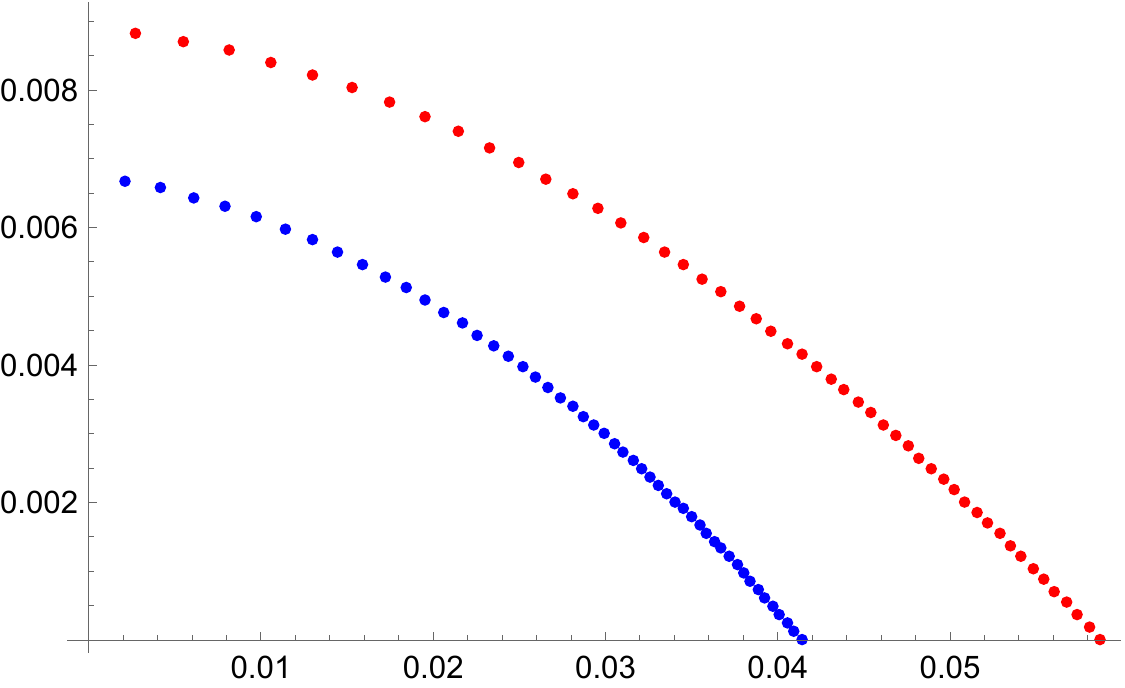}
    \caption{Comparing the boundary of the    geodesic control region using our refined estimates (red curve)  with the boundary arising from our initial estimates (blue curve).  }
    \label{fig:IncreasingGeodesicControl}
\end{figure}

\subsubsection{Improved $C_3(s_0,s_1)$ formula} \label{sec:ImprovedC3formula}

Recall that $C_3$ is an upper bound for the largest entry in the vector $(\Gamma_{11} +\Gamma_{12} +\Gamma_{22})$. Rather than estimating $C_3$  by using the matrix max norm as in Section~\ref{sec:boundC3} (see Corollary~\ref{cor:C3formula}), here we make a more precise estimate by using all of the terms arising in various multiplications and bounding them with quadratic terms (in $s_0, s_1$) using Mathematica. 

Recalling the definitions of the vectors $\Gamma_{ij} $ from equation~(\ref{Gammaeqation}), we have 
    $$
    \Gamma_{11} +\Gamma_{12} +\Gamma_{22}= Q_s^{-1} (  {\vb}_{11}  +{\vb}_{12} + {\vb}_{22})
    $$ 
with 
    \begin{align*}
    Q_s=\begin{pmatrix}
    E & F \\ 
    F&G
    \end{pmatrix},  \,
    {\vb}_{11} = 
    \begin{pmatrix}
    \frac 12 E_a \\ 
    F_a - \frac12 E_b  
    \end{pmatrix}, \, 
    {\vb}_{12} &= 
    \begin{pmatrix}
    \frac 12 E_b \\ 
     \frac12 G_a   
    \end{pmatrix}, \, 
    \vb_{22} = 
    \begin{pmatrix}
    F_b - \frac 12 G_a \\ 
    \frac12 G_b   \\ 
    \end{pmatrix}.
    \end{align*}
We express the first fundamental form using rectangular coordinates (see Section \ref{sec:rectangularcoordinates}). 

To bound $Q_s^{-1} (  {\vb}_{11}  +{\vb}_{12} + {\vb}_{22})$, we multiply 
    $\begin{bmatrix}
    G & -F \\ 
    -F&E
    \end{bmatrix}$
by $ ({\vb}_{11}  +{\vb}_{12} + {\vb}_{22}) $ and then bound each individual term in the resulting product by making it positive and using $\sin, \cos \leq 1$ and the rectangular coordinate bounds on the tube profile function $f$. Then we bound higher order terms (any term with $s_1^n$ for $n\geq 2$) by a constant times $s_1^2$ term. We then multiply by $\frac{1}{\det Q_s^{lb}}$ which is a lower bound for the deteminant (see Lemma \ref{lemma:DetLbQ}), producing a vector $C_3= 
    \begin{pmatrix}
         {C_3}_a\\
         {C_3}_b
    \end{pmatrix}$. Finally we define
    $$
    C_3^{improved}(s_0, s_1) = \max\{ {C_3}_a, {C_3}_b\}. 
    $$
Our calculations give
    \begin{equation}\label{eqn:improvedC3formulas}
    \begin{aligned}
     {C_3}_a &= \frac{2.135 s_0^2+0.874 s_0 s_1+13.612 s_0^3 s_1+4.100 s_1^2}{1 -2 s_1-0.155 s_0^2 s_1-1.946 s_1^2}\\
     \hbox{ and }&\\
     {C_3}_b &= \frac{2.135 s_0^2+0.331 s_0^4+1.748 s_0 s_1+4.269 s_0^2 s_1+0.068 s_0^3 s_1+4.762 s_1^2}{1 -2 s_1-0.155 s_0^2 s_1-1.946 s_1^2}\\
    \end{aligned}
    \end{equation}

\subsection{Improvements on bound for Riccati solution $u_{neg}$}

The global Anosov property is generated by all geodesics regularly encountering negative curvature for some uniform amount of time inside the tubes $\T_{1/4}$. Our initial approach bounded the negative curvature over all of $\T_{1/4}$ with a single value. However, a geodesic encounters more negative curvature the further into a tube it goes, and during a good time interval a trajectory will go so far as to reach $\T_{3/4}$, at which point there is significantly more negative curvature. In this section, we refine our analysis by dividing good time intervals inside $\T_{1/4}$ into two components: when the geodesic is in $\T_{1/4} \setminus \T_{1/2}$ (with the same curvature bound as before) and when it is in  $\T_{1/2}$ (with a better curvature bound). 

We give bounds on the curvature in these two regions denoted by $K_{1/4}= K_{neg}$ and $K_{1/2}$ with $K_{1/2}< K_{1/4}$ and bound the amount of time a geodesic spends in each of these regions (Section \ref{sec:MinTimeVariousRegions}). We then use these curvature and time bounds to give an improved estimate for the value of $u_{neg}$ of Ricatti solutions as geodesics pass through $\T_{1/4}$ in a $\T_{3/4}-$ good way (Section \ref{sec:improvedRicattisolution}). 

Here we extend the definition of strong finite horizon given in Definitions \ref{def:strongfhp} and \ref{def:strongfhpgeom} to capture this idea of bounding not only the time spent in a set $S_1$ but also the time spent in a subset $S_2$ of that set.   

\begin{definition}[Finite horizon condition: refined time version]\label{def:refinedstrongfhp}
    Let $M$ be a surface with Riemannian metric $g$, and let $S_2 \subset S_1$ be closed sets on $M$. We say $S_1, S_2$ have the \emph{$(T_{ret}, \Delta t_1, \Delta t_2)$-refined strong finite horizon property} if for any $g$-geodesic $\gamma$ there exists a sequence of times $t_i^{j \pm}$, $i \in \bbZ$, $j=1,2$ with $\lim_{i \to \pm \infty} t_i^{j\pm} = \pm \infty$, such that for all $i$,
        \begin{enumerate}
        \item $t_i^{1-} \leq t_i^{2-} \leq t_i^{2+} \leq t_i^{1+} < t_{i+1}^{1-} \leq t_{i+1}^{2-} \leq \ldots$ 
        \item for all $t \in [t_i^{1-}, t_i^{1+}]$, $\gamma(t) \in S_1$ and for all $t \in [t_i^{2-}, t_i^{2+}]$, $\gamma(t) \in S_2$
        \item $t_i^{1+}-t_i^{1-} \geq \Delta t_1$ and $t_i^{2+}-t_i^{2-} \geq \Delta t_2$
        \item $t_i^{2+}-t_i^{1-} \geq \frac12 (\Delta t_1+ \Delta t_2) $ and $t_i^{1+}-t_i^{2-} \geq \frac12 (\Delta t_1+ \Delta t_2)$
        \item $t_{i+1}^{1-} -t_i^{1+} \leq T_{ret}$
        \end{enumerate}
    We say such a geodesic enters $S_1$ in a $(\Delta t_1, \Delta t_2)$-good way at time $t_i^{1-}$.
\end{definition}

To summarize, a geodesic that enters $S_1$ in a $(\Delta t_1, \Delta t_2)$-good way will spend time at least $\Delta t_1$ in $S_1$ and of this time at least $\Delta t_2 < \Delta t_1$ will be spent inside $S_2$. 

To prove our Ricatti equation results (Section \ref{sec:improvedRicattisolution}), we also need   some control on how long a geodesic  spends in $S_1$ before entering $S_2$ and after leaving $S_2$. This control is provided by the lower bounds  on the time spent from entering $S_1$ to leaving $S_2$ and from entering $S_2$ to leaving $S_1$ (Condition 4).

\subsubsection{Improved curvature and time bounds}\label{sec:MinTimeVariousRegions}

We define 
    \begin{equation}\label{eqn:Krhodef}
    K_{\rho}\geq \max_{p\in \T_{\rho}} K(\rho)
    \end{equation}
to be an upper bound for the negative curvature on the set $\T_{\rho}$. This gives curvature  bounds $K_{1/4}$ on $\T_{1/4}$ and $K_{1/2}$ on $\T_{1/2}$ with $K_{1/2} \leq K_{1/4} <0$. Note that previously (see Section~\ref{sec:curvboundsinA14})  we used the notation  $  K_{neg}$ to denote an upper bound for the curvature on $\T_{1/4}$. Formulas for $K_\rho, \rho \leq 1/2$, come from Theorem~\ref{thm:knegformula}.

Here we show that if the pair of sets $\T_{3/4} \subset \T_{1/4}$ is $T_{ret}$ strongly finite horizon then the system will have the $(T_{ret},\Delta t_{1/4},\Delta t_{1/2})$ refined strongly finite horizon property.

\begin{theorem}\label{thm:timeintubes} (Improvement of Theorem~\ref{thm:Deltattheorem}.)
If a metric $g_s$ satisfies Condition I of Theorem~\ref{thm:Cibounds->geodcontrol->finitehorizon:Improved}, then $\T_{1/4}, \T_{1/2}$ has the $(T_{ret},\Delta t_{1/4},\Delta t_{1/2})$ refined strongly finite horizon property in the metric $g_s$ with 
    $$T_{ret} = 2.30571$$
and
    \begin{align*}
        \Delta t_{1/4} &=  \lambda_s^{lb} * 2 \sqrt{(r_{1/4})^2- (r_{3/4}+\frac{\Delta r}{20})^2} \\
        \Delta t_{1/2} &=  \lambda_s^{lb} * 2 \sqrt{(r_{1/2})^2- (r_{3/4}+\frac{\Delta r}{20})^2}.\\
        \frac 12 (\Delta t_{1/4}  + \Delta t_{1/2} )  &=\lambda_s^{lb} * \left \{\sqrt{(r_{1/4})^2- (r_{3/4}+\frac{\Delta r}{20})^2}+ \sqrt{(r_{1/2})^2- (r_{3/4}+\frac{\Delta r}{20})^2}\right \} \\
    \end{align*}
\end{theorem}

\begin{proof}
These results follow from Theorem \ref{thm:Deltattheorem} and   Corollary \ref{cor:deltatinTrho} once  we make a small modification to the theorem's proof with $\lambda_s^{lb}$ given by equation~(\ref{lemma:lengthratio2}).  Since here we assume $g_s$ has geodesic control for time $2.5$ rather than for time $3$, the estimate for $ \max\{ |\Delta a(t,s)|, |\Delta b(t,s)|\} $ in equation (\ref{eqn:DeltaabBounds}) becomes 
    \begin{align*}
        &\leq \frac 12  C_1 + \frac 18  C_2^2 C_3 \\
        &\leq \frac 15 ( 2.5 C_1 + \frac {2.5^2}2  C_2^2 C_3 ) \\
        & \leq \frac 15 ( \frac{\Delta r}{ 4 \sqrt{2}} )
        = \frac{\Delta r}{ 20 \sqrt{2}}.
    \end{align*}
The $\Delta t_{1/4}$ amd $\Delta t_{1/2}$ results   now follow from Corollary \ref{cor:deltatinTrho} with the 24 replaced by 20. 
    
To prove Condition 4 of the refined strongly finite horizon property (Definition \ref{def:refinedstrongfhp}), we need a lower bound on the time a good geodesic spends from entering $\T_{1/4}$ to leaving $\T_{1/2}$ and  from entering $\T_{1/2}$ to leaving $\T_{1/4}$. We denote such  bounds by  $\Delta t_{\frac 14, \frac 12}$ and  $\Delta t_{\frac 12, \frac 14}$ respectively. 
    
Checking the proof of Theorem \ref{thm:Deltattheorem}, one sees that a lower bound on the time spent from entering a tube $T_\rho$ to crossing the $x=0$ half-way line is given by 
    $$
     \lambda_s^{lb} *  \sqrt{(r_{\rho})^2- (r_{3/4}+\frac{\Delta r}{20})^2}. 
    $$
Applying that result twice, once to the tube $T_{1/4}$ and once to the tube $T_{1/2}$ gives that 
    \begin{align*}
         \Delta t_{\frac 14, \frac 12}= \Delta t_{\frac 12, \frac 14}&=\lambda_s^{lb} * \left \{\sqrt{(r_{1/4})^2- (r_{3/4}+\frac{\Delta r}{20})^2}+ \sqrt{(r_{1/2})^2- (r_{3/4}+\frac{\Delta r}{20})^2}\right \} \\
         &= \frac 12( \Delta t_{1/4} + \Delta t_{1/2}).
    \end{align*} 
\end{proof}

From entering $\T_{1/4}$ to leaving $\T_{1/2}$ the geodesic spends time at least $\Delta t_{\frac 14, \frac 12}$. Of that time, at least $\Delta t_{1/2}$ is spent in $\T_{1/2}$ where the curvature is bounded above by $K_{1/2}$. The remaining time $\Delta t_{\frac 14, \frac 12}- \Delta t_{1/2}$ is \emph{not} a lower bound on the time spent in the annulus $\T_{1/4} \setminus \T_{1/2}$ as it 
may include a component when the geodesic is in $\T_{1/2}$. All we can conclude is that during this time interval, the curvature the geodesic encounters is bounded above by $K_{1/4}$. 

Note that by geodesic control, once a geodesic leaves $\T_{1/2}$ it will pass through $\T_{1/4}$ and then leave $\T_0$.

\subsubsection{Improved Ricatti solution bounds}\label{sec:improvedRicattisolution}

For an $(s_0,s_1)$ parameter value such that $\T_{1/4}, \T_{1/2}$ has the $(T_{ret},\Delta t_{1/4},\Delta t_{1/2})$ refined strongly finite horizon property with respective curvature upper bounds $K_{1/2}\leq K_{1/4}<0$, we show that there exists a lower bound $u_{lb}^* > 0$ for solutions of the Ricatti equation along geodesics that pass through $T_{1/4}$ in a $(\Delta t_1, \Delta t_2)$ good way. This lower bound is an improvement over our earlier approach when we only considered the curvature $K_{neg}=K_{1/4}$ in $\T_{1/4}$ and got a lower bound $u_{neg}$ (see  equation (\ref{eqn:uneg})). It will be the case that $u_{lb}^* > u_{neg} >0$. 

\begin{definition}
    Define 
        \begin{equation}\label{eqn:KLdefinition}
        K^* (t) = 
            \begin{cases}
            K_{1/4} &  t\in [0, \tau_1^*] \notag \\
            K_{1/2} & t\in (\tau_1^*, \tau_2^*]\notag\\
            K_{1/4} & t\in (\tau_2^*,\tau_f^*] \notag \\
            \end{cases} ,
        \end{equation}
    with 
        $$
        \tau_1^* =  \frac 12 (\Delta t_{1/4} - \Delta t_{1/2}), \quad 
        \tau_2^* =  \tau_1^* + \Delta t_{1/2}, \quad 
        \tau_f^*= \tau_2^* + \tau_1^* = \Delta t_{1/4 }
        $$
    and with  $K_{1/2} \leq K_{1/4} <0$    upper bounds for the curvature on $T_{1/2}$ and   $T_{1/4} $ respectively.

    Let $u^* (t), t\in[0,\tau_f^*] $, be the solution of the Ricatti equation
        $$
        u'(t) = - K^*(t) - u^2(t)
        $$
    with initial condition $u^* (0) = 0$, and set 
        $$
        u^{*}_{lb}  = \min\{ u^* (\tau_f^*), .99 \sqrt{-K_{1/4}} \}.
        $$
\end{definition}

We put the $.99$ in the definition to simplify the proof of the Kourganoff conditon in Theorem \ref{thm:ImprovedAnosov}. In our computations, the above minimum turns out to be the $ u^* (\tau_f^*)$ value. 

By Theorem~\ref{thm:timeintubes}, when the metric $g_s$ satisfies Condition I of Theorem~\ref{thm:Cibounds->geodcontrol->finitehorizon:Improved}, the values for  the $\Delta t$ terms lead to the following $\tau^*$ values: 
    \begin{align*}
    \tau_1^* &= 0.114016 \sqrt{1 -2 s_1-6.284 s_0 s_1^2}, \\
    \tau_2^* &= 0.558007 \sqrt{1 -2 s_1-6.284 s_0 s_1^2}, \\
    \tau_f^* &= 0.672023 \sqrt{1 -2 s_1-6.284 s_0 s_1^2}. \\
    \end{align*} 
    
\begin{proposition}
    Let $\T_{1/4}, \T_{1/2}$ have the $(T_{ret},\Delta t_{1/4},\Delta t_{1/2})$ refined strongly finite horizon property in the metric $g_s$.     Let $\gamma_s$ be a geodesic that enters $\T_{1/4}$ in a $(\Delta t_1, \Delta t_2)$-good way at time 0 and then  leaves $\T_{1/4}$ at time $\tau_f^s$. Let $u^s(t), t\in[0, \tau_f^s] $, be the Ricatti solution with initial condition $u^s(0) = 0 $ and curvature $K^s(t) = K(\gamma_s(t)) $. 

    Then 
        $$
        u^s(\tau_f^s )\geq u^{*}_{lb}.
        $$
\end{proposition}

\begin{proof}
Denote by $\tau_1^s$ and $\tau_2^s$ the times at which $\gamma_s(t)$ enters and then leaves $T_{1/2}$ (a specific tube in $\T_{1/2}$). Then $\tau_2^s - \tau_1^s$ is the time $\gamma_s$ spends in $T_{1/2}$.

Definition \ref{def:refinedstrongfhp} implies
    \begin{align*}
    \tau_f^s&\geq \tau_f^* =\Delta t_{1/4},\\
    \tau_2^s & \geq \tau_2^*=\Delta t_{\frac14,\frac12}, \\
    \tau_2^s - \tau_1^s & \geq \tau_2^* - \tau_1^* =   \Delta t_{1/2},\\
    \tau_f^s-\tau_1^s  & \geq \tau_f^* - \tau_1^* = \Delta t_{\frac 12, \frac 14}.
    \end{align*}
There are two cases to consider determined by the value of $\tau_1^s$.

\medskip
    
\noindent Case 1: $\tau_1^s \leq  \tau_1^*$. Then for $t\in [0, \tau_f^*]$, we have  
    $$
    K^s(t) \leq K^*(t) 
    $$
since
\begin{itemize}
    \item $K^s(t) \leq K_{1/4} = K^*(t)$ for $t\in [0, \tau_1^s]$;
    \item by the ordering of times, $0< \tau_1^s \leq \tau_1^* \leq \tau_2^* \leq \tau_2^s$ so $K^s(t) \leq K_{1/2} \leq  K^*(t)$ for $t\in [\tau_1^s, \tau_2^s]$; and
    \item $K^s(t) \leq K_{1/4} = K^*(t)$ for $t\in [\tau_2^*, \tau_f^*]$.
\end{itemize}
Hence, by the Comparison Theorem, $u^s(t) \geq u^* (t)$
for all $t\in [0, \tau_f^*]$ and thus
    $$
    u^s(\tau_f^*) \geq u^* (\tau_f^*).
    $$

We still need to account for the time from $\tau_f^*$ to $\tau^s_f$; we continue to examine $u^s(t)$ for $t \in [\tau_f^*,\tau^s_f]$. There are two cases. 

\begin{itemize}
    \item If $u^s(\tau_f^*) \geq \sqrt{-K_{1/4}}$, then since $\gamma_s(t)$ is in $\T_{1/4}$ (potentially also in $\T_{1/2}$), its curvature satisfies $K^s(t) \leq K_{1/4}$. Solving the Ricatti equation gives   $u^s(t ) \geq \sqrt{-K_{1/4}}$ so that 
        $$
        u^s(\tau_f^s ) \geq \sqrt{-K_{1/4}} \geq u^{*}_{lb}, 
        $$
    (although potentially ($u^s(\tau_f^s )< u^*(\tau_f^* )$).

    \item If $u^s(\tau_f^*) <  \sqrt{-K_{1/4}}$, then $u^s(\tau_f^*)' >  0$ and $u^s(t)$ will increase so  long as $u^s(t) < \sqrt{-K^s(t)}$. So either $u^s (t)$ will attain the value  $\sqrt{-K_{1/4}}$ , after which it will never decrease below that value, or it will keep increasing until time  $\tau_f^s$. In either case,   
        $$
        u^s( \tau_f^s) > u^* (\tau_f^*) \geq u^{*}_{lb}.  
        $$
\end{itemize}

\noindent Case 2: $\tau^s_1 > \tau_1^*$.  We will repeat the above argument with slight modifications. 

By assumption, $\tau^s_1-\tau_1^*$ is positive.  Then for $t\in (0, \tau^s_1-\tau_1^*], \, u^s(t) 
> 0$ and hence
    $$
    u^s(\tau_1^s-\tau_1^*) > u^*(0) = 0. 
    $$

Shifting the time along $\gamma_s$ by $(\tau_1^s-\tau_1^*)$ allows us to compare the curvature again to $K^*$.  Since
    \begin{equation*}
    K^s((\tau_1^s-\tau_1^*)+t) \leq K^*(t)\, \text{ for } \, t\in [0, \tau_f^*]
    \end{equation*}
we can use the argument of  Case 1 to show that 
    $$
    u^s( \tau_f^s)   \geq u^{*}_{lb}.  
    $$ 
\end{proof}

The following is a refinement of Definition~\ref{def:uRiccatisoln} as expressed using Lemma~\ref{lemma:uknegukpos}.

\begin{definition}\label{def:improvedstrictlyinvariantcones}
    The numbers $\Delta t_{1/2}, \Delta t_{1/4}, T_{ret}>0$ and $K_{1/2} < K_{1/4} < 0 < K_{pos}$  satisfy the {improved strictly invariant cone condition} if 
        $$
        u^{* }_{lb} + u_{pos}(T_{ret}) > 0 
        $$ 
    where $u_{pos}$ is the solution of the Ricatti equation with $K(t) \equiv K_{pos} $ and initial condition $u_{pos}(0) =0$.
\end{definition}

\subsection{Improved Anosov condition}

The following is an improvement on Theorem~\ref{thm:Anosov} because part (4) is a weaker assumption on $(s_0,s_1)$ than equation~(\ref{eqn:strictlyinvconecondition}).

\begin{theorem}\label{thm:ImprovedAnosov}
    Let $s=(s_0,s_1)$ be parameter values  for the metric $g_s$ on the model space $\M_0$ such that 
    \begin{enumerate}
        \item the sets $\T_{1/4} \subset \T_{1/2}$ have the $(T_{ret},\Delta t_{1/4},\Delta t_{1/2})$ refined strongly finite horizon property,   
        \item the Gaussian curvature inside $\T_{1/4}$ is bounded above by $K_{1/4}<0$, and inside $\T_{1/2}$  by $K_{1/2}<0$, 
        \item the Gaussian curvature on $\M_0$ is bounded above by $K_{pos}$,
        \item the improved strongly invariant cone condition (Definition \ref{def:improvedstrictlyinvariantcones}) holds for the numbers $\Delta t_{1/2}, \Delta t_{1/4}, T_{ret}>0$ and $K_{1/2} < K_{1/4} < 0 < K_{pos}$. 
    \end{enumerate}
    Then the geodesic flow is Anosov.
\end{theorem}

Again this result would be an immediate consequenence of Kourganoff's Theorem~\ref{thm:Kournagoff} if the times spent in $\T_{1/4}$ were bounded. 

\begin{proof}
The proof is the same as in Theorem \ref{thm:Anosov} with one modification: we change the definition of $\Delta t$. Let $\Delta t$ be the time for a Ricatti solution with initial condition 0 and $K(t) \equiv K_{1/4}$ to have value greater than $.99 \sqrt {-K_{1/4}}$. We divide  $[t_i^{1-},t_i^{1+}]$ when $t_i^{1+} - t_i^{1-} \geq   2 \Delta t$. 

Consider a geodesic segment of length between $[\Delta t, 2\Delta t)$ that ends on $\partial \T_{1/4}$. The corresponding Ricatti solution will have value $u\geq .99 \sqrt{-K_{1/4}}\geq u_{lb}^*$ when it reaches $\partial \T_{1/4}$. The improved strongly invariant cone condition now implies that the Korganoff conditions hold. 
\end{proof}

\subsection{Effects on Anosov analysis and algorithm}

Using the refined strongly finite horizon formulation, we are able to lower our bound on the genus of embedded Anosov surfaces. 

\begin{theorem}\label{thm:improvedgenusbound}
    There exists a  smooth compact embedded surface of genus $17,288,843,803$  with Anosov geodesic flow. 
\end{theorem} 

\begin{proof}
The proof of this result parallels  that of  Theorem~\ref{thm:AnosovCondition}. 

\begin{enumerate} 
    \item For a given $m$ value, there exists a unique parameter value $s_0^*$ for which 
        \begin{equation}
        C_1(s) T_{gc} + C_2^2(s) C_3(s) \frac {T_{gc}^2}2 = \frac{\Delta r}{4\sqrt{2}}
        \end{equation}
    with $s= (s_0, m s_0)$ and $T_{gc} = 2.5$. The $C_1, C_2$ functions are as before (Corollary \ref{cor:C1C2formulas}) while  $C_3(s_0, ms_0)$ is now given by equation (\ref{eqn:improvedC3formulas}). 

    For any $(s_0,s_1)$ with $s_0 \leq s_0^*$ and $s_1 \leq ms_0^*$, the sets $\T_{1/2} \subset \T_{1/4}$ have the $(T_{ret},\Delta t_{1/4}, \Delta t_{1/2})$ refined  strongly finite horizon property in the  $g_{(s_0,m s_1)}$ metric, with values of $T_{ret}, \Delta t_{1/4}, \Delta t_{1/2}$ given in Theorem~\ref{thm:timeintubes}. 

    \item There exists an $s_1^* \leq m s_0^*$ such that for any $s_1 \leq s_1^*$,   the metric $g_{(s_0^*,s_1)}$ has the improved strictly invariant cone condition (Definition ~\ref{def:improvedstrictlyinvariantcones}) 
        $$
        u^{* }_{lb} + u_{pos}(T_{ret}) > 0,
        $$
    and hence by Theorem \ref{thm:ImprovedAnosov} its geodesic flow will be Anosov. 

    \item There exist  $s_1^{e}=\frac{1}{\sqrt3 n} \leq s_1^*$ for which the metric $g_{(s_0^*,s_1^e)}$ comes from a compact embedded surface of genus $6 n^3 +1$.

    \item Using the numerical methods of Section \ref{sec:numericallycomputegenus}, we compute values for $s_0^*, s_1^*$ and $s_1^e$ and then iterate the calculations to maximize $s_1^e$ and hence minize the genus. Our calculations give $n = 1423$ leading to genus $17,288,843,803$. 
\end{enumerate} 
\end{proof}

\begin{figure}[h]
    \centering
    \includegraphics[width=.5\textwidth]{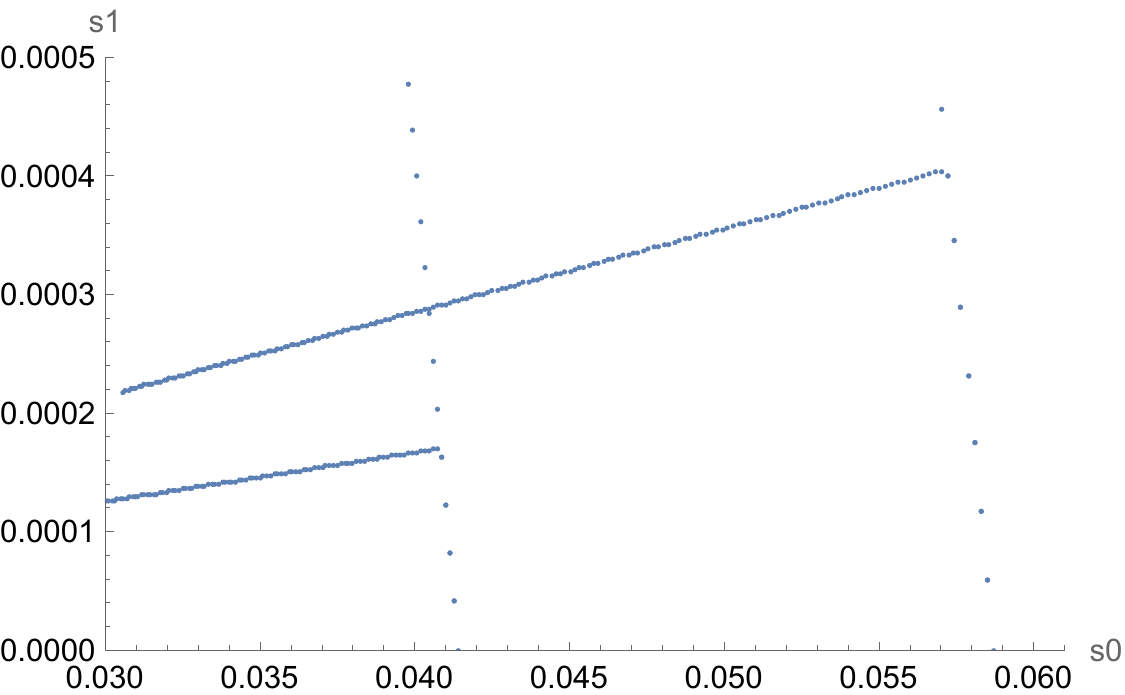}
    \caption{Geodesic control and Riccati condition curves for refined estimates as compared with original estimates.}
    \label{fig:ImprovedAnosov}
\end{figure}

The improvement in the bound on genus due to our refined estimates can be visualized in Figure~\ref{fig:ImprovedAnosov}. Decreasing the value of $T_{gc}$ from 3 to 2.5 and improving the estimate for $C_3$ causes the geodesic control curve to shift to the right (Figure~\ref{fig:IncreasingGeodesicControl}). Using the refined finite horizon condition and improved strictlly invariant cone condition causes the Anosov curve $(s_0^*, s_1^*)$ to move upward. The intersection of the geodesic control and Anosov curves, which corresponds to our surface of minimal genus,  moves to the right (larger $s_0$ value) and up (larger $s_1^*$ value).

\bibliographystyle{abbrv}
\bibliography{references}

\begin{thebibliography}{10}

\bibitem{Anosov-67}
D.~V. Anosov.
\newblock Geodesic flows on closed {Riemann} manifolds with negative curvature.
\newblock Proc. {Steklov} {Inst}. {Math}. 90, 235 p. (1967)., 1967.

\bibitem{Arnold-63}
V.~I. Arnol'd.
\newblock Small denominators and problems of stability of motion in classical
  and celestial mechanics.
\newblock {\em Russ. Math. Surv.}, 18(6):85--191, 1963.

\bibitem{Dolgopyat-98}
D.~Dolgopyat.
\newblock On decay of correlations in {Anosov} flows.
\newblock {\em Ann. Math. (2)}, 147(2):357--390, 1998.

\bibitem{Donnay-Pugh-03}
V.~Donnay and C.~Pugh.
\newblock Anosov geodesic flows for embedded surfaces.
\newblock {\em Geometric methods in dynamics (II): Ast\'erisque}, 287:61--69,
  2003.

\bibitem{Donnay-Pugh-04}
V.~Donnay and C.~Pugh.
\newblock Finite horizon {Riemann} structures and ergodicity.
\newblock {\em Ergodic Theory Dyn. Syst.}, 24(1):89--106, 2004.

\bibitem{Donnay-Visscher-18}
V.~Donnay and D.~Visscher.
\newblock A new proof of the existence of embedded surfaces with {Anosov}
  geodesic flow.
\newblock {\em Regul. Chaotic Dyn.}, 23(6):685--694, 2018.

\bibitem{Donnay-Visscher-mathematica}
V.~Donnay and D.~Visscher.
\newblock Anosov embedded surfaces: Mathematica computations.
\newblock DOI: 10.5281/zenodo.20834665, 2026.

\bibitem{Eberlein-73-I}
P.~Eberlein.
\newblock When is a geodesic flow of {Anosov} type, {I}.
\newblock {\em J. Differ. Geom.}, 8:437--463, 1973.

\bibitem{Eberlein-73-II}
P.~Eberlein.
\newblock When is a geodesic flow of {Anosov} type, {II}.
\newblock {\em J. Differ. Geom.}, 8:565--577, 1973.

\bibitem{Guglielmo-Ruggiero-26}
G.~Guglielmo and R.~Ruggiero.
\newblock Path {Connectivity} of {Anosov} {Metrics} on {Surfaces}.
\newblock Preprint, {arXiv}:2601.08656, 2026.

\bibitem{Gulliver-75}
R.~Gulliver.
\newblock On the variety of manifolds without conjugate points.
\newblock {\em Trans. Am. Math. Soc.}, 210:185--201, 1975.

\bibitem{Hadamard-98}
J.~Hadamard.
\newblock Les surfaces {\`a} courbures oppos{\'e}es et leurs lignes
  g{\'e}od{\'e}siques.
\newblock {\em Journ. de Math. (5)}, 4:27--73, 1898.

\bibitem{Hedlund-39}
G.~Hedlund.
\newblock The dynamics of geodesic flows.
\newblock {\em Bull. Am. Math. Soc.}, 45:241--260, 1939.

\bibitem{Hopf-39}
E.~Hopf.
\newblock Statistik der geod{\"a}tischen {Linien} in {Mannigfaltigkeiten}
  negativer {Kr{\"u}mmung}.
\newblock Ber. {Verh}. {S{\"a}chs}. {Akad}. {Leipzig} 91, 261-304 (1939).,
  1939.

\bibitem{Hopf-48}
E.~Hopf.
\newblock Closed surfaces without conjugate points.
\newblock {\em Proc. Natl. Acad. Sci. USA}, 34:47--51, 1948.

\bibitem{Hunt-MacKay-03}
T.~J. Hunt and R.~S. MacKay.
\newblock Anosov parameter values for the triple linkage and a physical system
  with a uniformly chaotic attractor.
\newblock {\em Nonlinearity}, 16(4):1499--1510, 2003.

\bibitem{Jane-Ruggiero-14}
D.~Jane and R.~Ruggiero.
\newblock Boundary of {Anosov} dynamics and evolution equations for surfaces.
\newblock {\em Math. Nachr.}, 287(17-18):2002--2020, 2014.

\bibitem{Klingenberg-74}
W.~Klingenberg.
\newblock Riemannian manifolds with geodesic flow of {Anosov} type.
\newblock {\em Ann. Math. (2)}, 99:1--13, 1974.

\bibitem{Kourganoff-16-linkage}
M.~Kourganoff.
\newblock Anosov geodesic flows, billiards and linkages.
\newblock {\em Commun. Math. Phys.}, 344(3):831--856, 2016.

\bibitem{Kourganoff-16-embed}
M.~Kourganoff.
\newblock Embedded surfaces with {Anosov} geodesic flows, approximating
  spherical billiards.
\newblock Preprint, {arXiv}:1612.05430, 2016.

\bibitem{Kourganoff-18}
M.~Kourganoff.
\newblock Uniform hyperbolicity in nonflat billiards.
\newblock {\em Discrete Contin. Dyn. Syst.}, 38(3):1145--1160, 2018.

\bibitem{Liverani-04}
C.~Liverani.
\newblock On contact {Anosov} flows.
\newblock {\em Ann. Math. (2)}, 159(3):1275--1312, 2004.

\bibitem{Ratner-74}
M.~Ratner.
\newblock Anosov flows with {Gibbs} measures are also {Bernoullian}.
\newblock {\em Isr. J. Math.}, 17:380--391, 1974.

\end{thebibliography}
 
\end{document}